\documentclass[11pt]{article}
\usepackage{geometry}
\usepackage{latexsym}
\usepackage{amsthm}
\usepackage{amsmath}
\usepackage{amssymb}
\usepackage{titlesec}
\usepackage{times}
\usepackage{booktabs}

\usepackage[super,square]{natbib}
\usepackage{graphicx}
\usepackage{subfigure}
\usepackage{epstopdf}
\numberwithin{equation}{section}

\newtheorem{definition}{Definition}[section]
\newtheorem{rem}{Remark}[section]
\newtheorem{theorem}{Theorem}[section]
\newtheorem{lemma}{Lemma}[section]

\newtheorem{prop}{Proposition}[section]

\newcommand{\f}{\mbox{\boldmath$f$}}

\newcommand{\n}{\mbox{\boldmath$n$}}
\newcommand{\hf}{\frac{1}{2}}
\newcommand{\nrm}[1]{\left\| #1 \right\|}

\newcommand\bu{\boldsymbol{u}}
\newcommand\bv{\boldsymbol{v}}

\title{\Large Convergence analysis of a positivity-preserving 
numerical scheme for the Cahn-Hilliard-Stokes system with Flory-Huggins energy potential}       

\author{
Yunzhuo Guo \thanks{School of Mathematical Sciences, Beijing Normal University, Beijing 100875, P.R. China (
yunzguo@mail.bnu.edu.cn)}
\and
Cheng Wang\thanks{Department of Mathematics, The University of Massachusetts, North Dartmouth, MA  02747, USA (Corresponding Author: cwang1@umassd.edu)}
	\and
Steven M. Wise\thanks{Department of Mathematics, The University of Tennessee, Knoxville, TN 37996, USA (swise1@utk.edu)}
\and	
Zhengru Zhang\thanks{Laboratory of Mathematics and Complex Systems, Beijing Normal University, Beijing 100875, P.R. China (zrzhang@bnu.edu.cn)}
}
\date{ }

\begin{document}

\maketitle
\begin{abstract}
A finite difference numerical scheme is proposed and analyzed for the Cahn-Hilliard-Stokes system with Flory-Huggins energy functional. A convex splitting is applied to the chemical potential, which in turns leads to the implicit treatment for the singular logarithmic terms and the surface diffusion term, and an explicit update for the expansive concave term. The convective term for the phase variable, as well as the coupled term in the Stokes equation, are approximated in a semi-implicit manner. In the spatial discretization, the marker and cell (MAC) difference method is applied, which evaluates the velocity components, the pressure and the phase variable at different cell locations. Such an approach ensures the divergence-free feature of the discrete velocity, and this property plays an important role in the analysis. The positivity-preserving property and the unique solvability of the proposed numerical scheme are theoretically justified, utilizing the singular nature of the logarithmic term as the phase variable approaches the singular limit values. An unconditional energy stability analysis is standard, as an outcome of the convex-concave decomposition technique. A convergence analysis with accompanying error estimate is provided for the proposed numerical scheme. In particular, a higher order consistency analysis, accomplished by supplementary functions, is performed to ensure the separation properties of numerical solution. In turn, using the approach of rough and refined error (RRE) estimates, we are able to derive an optimal rate convergence. To conclude, several numerical experiments are presented to validate the theoretical analysis.

\bigskip


\end{abstract}


\section{Introduction}
The Cahn-Hilliard-Stokes (CHS) system, a gradient flow equation coupled with incompressible fluid motion, can be used to describe the phase separation and flow of a very viscous binary fluid\cite{diegel15a}. Let~$\Omega \in \mathbb{R}^d$, $d=2, 3$, be an open domain. The following CHS system with Flory-Huggins potential is considered:
	\begin{eqnarray} 
&&  \partial_t \phi + \nabla \cdot ( \phi \bu ) = \Delta \mu , 
	\label{equation-CHS-1} 
	\\
&& - \Delta \bu + \bu = - \nabla p - \gamma  \phi \nabla \mu , 
    \label{equation-CHS-2}   
	\\
&& \nabla \cdot \bu = 0 ,  
	\label{equation-CHS-3}   
	\\
&& \mu  = \delta_\phi E = \ln ( 1 + \phi ) - \ln ( 1 - \phi ) - \theta_0 \phi 
   - \varepsilon^2 \Delta \phi ,  
   	\label{equation-CHS-4}     
	\end{eqnarray}   
with no-flux and no-penetration free-slip boundary conditions,
	\[
\partial_{n} \phi=\partial_{n} \mu=0, \quad \boldsymbol{u} \cdot \boldsymbol{n} = \partial_n(\bu\cdot \boldsymbol{\tau}) = 0  \quad\quad\mbox{on} \ \partial \Omega \times(0, T].
	\]
In this system, $\phi$ is a binary fluid concentration, $\mu$, $p$ and $\bu$ describe the chemical potential, pressure and fluid velocity vector, respectively. The parameter $\gamma>0$ is related to surface tension.  Observe that equations \eqref{equation-CHS-1} -- \eqref{equation-CHS-4} correspond to a simplified version of a model studied by others, obtained by assuming that two fluids have the same densities, and the gravity effects may be ignored~\cite{Della_nonlocalCHHS_2018, lee02a}. 

For the fluid part of the physical system, the no-penetration boundary condition, $\boldsymbol{u} \cdot \boldsymbol{n} =0$ on $\partial \Omega$, is natural. Meanwhile, both the no-slip boundary condition, $\bu\cdot \boldsymbol{\tau} =0$, and free-slip boundary condition, $\partial_n(\bu\cdot \boldsymbol{\tau}) = 0$ (on $\partial \Omega$), are physically reasonable. On the other hand, the Stokes operator with the free-slip boundary condition is symmetric, due to the homogeneous Neumann boundary condition for the pressure field induced by this boundary condition. As a result, the analysis with free-slip boundary condition becomes simpler than the one with no-slip boundary condition. For simplicity of presentation, we only focus on the free-slip boundary condition for the velocity vector in this article, although the analysis could be similarly extended the one with no-slip boundary condition; the technical details are left to interested readers. 

For any $\phi \in H^1 (\Omega)$, with the point-wise bound $-1 < \phi < 1$, the Flory-Huggins free energy functional is given by 
	\begin{equation}
	\label{CH energy}
F(\phi)=\int_{\Omega}\left( ( 1+ \phi) \ln (1+\phi) + (1-\phi) \ln (1-\phi) - \frac{\theta_0}{2} \phi^2 +\frac{\varepsilon^2}{2}|\nabla \phi|^2\right) d {\bf x} ,
	\end{equation}
in which $\varepsilon$, $\theta_0$ are positive constants associated with the diffuse interface width.  The following dissipation property is valid for the energy functional \eqref{CH energy}: 
	\begin{equation}
\partial_t F(\phi)= -\|\nabla \mu\|^2-\frac{1}{\gamma}\left( \|\bu\|^2+\|\nabla \bu\|^2  \right). 
	\label{dissipation property}
	\end{equation}
It is clear that the logarithmic free energy functional has a singularity near the values of $\pm 1$, which poses a great challenge in the numerical design. As an alternate approach, a non-singular polynomial energy has also been widely used
	\begin{equation}
F(\phi) =  \int_{\Omega}\left(\frac{1}{4}\left(\phi^{2}-1\right)^{2}+\frac{\varepsilon^{2}}{2}|\nabla \phi|^{2}\right) d {\bf x} . 
	\label{polynomial energy-1}
	\end{equation}
Similar to \eqref{CH energy}, this model has a double-well potential, which can be regarded as a polynomial approximation to the original one, with a larger error in the actual physical situation~\cite{chen19a}. A finite element analysis of \eqref{equation-CHS-1}-\eqref{equation-CHS-2}, with an added time derivative in the Stokes equation and polynomial energy \eqref{polynomial energy-1} was reported in a recent paper~\cite{diegel15a}.

One can show that, for these particular flow boundary conditions, if the fields are sufficiently regular, it follows that
	\begin{align}
-\Delta p = \gamma\nabla\cdot\left(\phi\nabla\mu \right) \quad &\mbox{in} \ \Omega,
	\label{pressure-problem-1}
	\\
-\partial_n p = \gamma\phi \partial_n \mu = 0 \quad &\mbox{on} \ \partial\Omega.
	\label{pressure-problem-2}
	\end{align}
In short, one can separate the pressure and velocity calculations. Taking advantage of this property, let us define a Helmholtz-type projection as follows:
	\begin{equation}
{\cal P}_H:\left\{\f\in\left[H^1(\Omega)\right]^3 \ \middle| \ \f\cdot\n = 0 \ \mbox{on} \ \partial\Omega \right\} \to    \Big\{\bv\in \left[H^1(\Omega)\right]^3 \ \Big| \ \nabla\cdot \bv =  0 \ \mbox{in}  \ \Omega, \  \bv\cdot\n =  0 \ \mbox{on}  \ \partial\Omega\Big\},
	\end{equation}
where ${\cal P}_H(\f) := \f+\nabla p$, where $p\in \mathring{H}^2_N(\Omega)\cap H^1(\Omega)$ is the unique solution to  $-\Delta p = \nabla\cdot\f$ in $\Omega$, as in \eqref{pressure-problem-1} -- \eqref{pressure-problem-2}. Here
	\[
H_N^2(\Omega) := \left\{\phi\in H^2(\Omega) \ \middle| \ \partial_n\phi = 0 \ \mbox{on} \ \partial\Omega \ \right\} \quad \mbox{and} \quad  \mathring{H}_N^2(\Omega) := \left\{\phi\in H_N^2(\Omega) \ \middle| \ (\phi,1) = 0  \right\}. 
	\]
Clearly, 
	\[
\left({\cal P}_H(\f), \f-{\cal P}_H(\f)\right)_{L^2} = 0.
	\]
From this we can prove the $L^2$ stability of the projection. Of course, sufficiently regular solutions to the CHS system~\eqref{equation-CHS-1} -- \eqref{equation-CHS-4} satisfy 
	\[
- \Delta \bu + \bu = -\gamma{\cal P}_H\left(\phi\nabla\mu\right) , 
	\]
assuming the no flux, no-penetration, and free-slip boundary conditions. Thus, the Cahn-Hilliard-Stokes system can be reformulated to, effectively, remove the velocity:  
	\begin{equation}
	\begin{aligned}
& \phi_t + \nabla \cdot ( \bu \phi ) = \Delta \mu  ,      
	\\
& \bu = - {\cal S}^{-1} {\cal P}_H  \gamma  ( \phi \nabla \mu ),
   	\end{aligned}
	\label{equation-CHS-reformulate} 
	\end{equation} 
where
	\[
{\cal S} :=  - \Delta + I , 
	\]
with the appropriate boundary conditions. One can observe that equation~\eqref{equation-CHS-reformulate} is, in essence, a Cahn-Hilliard-type equation, a modified gradient flow.

For this PDE system, a positivity-preserving property, that is, $1 + \phi >0$ and $1 - \phi>0$, can be theoretically justified, due to the logarithmic terms appearing in $\mu$. Of course, the numerical analysis of  the Cahn-Hilliard equation, by itself,  is an interesting topic, and recent works have been devoted to that equation with an assumed Flory-Huggins potential: for example, the finite difference method~\cite{chen19a} and the finite element approach~\cite{barrett99, yang_CH_2017}. 

The question of energy stability has always been an essential issue for any numerical approximation to a gradient flow coupled with fluid motion, and some existing works has been reported~\cite{diegel17, han15, shen2015}. Meanwhile, most existing numerical efforts have been based on the polynomial approximation in the energy potential, so that singularities can be avoided with respect to the phase variable. For the Flory-Huggins energy potential~\eqref{CH energy} and the corresponding CHS system~\eqref{equation-CHS-1} -- \eqref{equation-CHS-4}, the preservation of both the point-wise positivity (for the logarithmic arguments) and the energy stability turns out to be a very challenging issue. This  comes from the highly nonlinear, singular, and coupled nature of the PDE system. In this work, a fully discrete finite difference scheme is proposed and analyzed for solving the CHS system with logarithmic Flory-Huggins potential. Four theoretical properties will be justified for the numerical scheme: positivity-preserving, unique solvability, unconditional energy stability (in the physical free energy), and optimal rate convergence. 

In more details, the numerical approximation to the chemical potential is based on the convex-concave decomposition of the Flory-Huggins energy functional, which dates back to Eyre~\cite{eyre98}. This approach ensures a theoretical justification of its positivity-preserving property, because of an implicit treatment of the nonlinear singular logarithmic term. In particular, the singular and convex nature of the logarithmic term prevents the numerical solution reaching the singular limit values, so that a point-wise positivity is preserved for the phase variable. See the related works~\cite{chen22a, chen19b, Dong2020b, dong19b, dong20a, duan22a, LiuC2021b, LiuC2021a, LiuC2022a, QianWangZhou_JCP20, Yuan2021a, ZhangJ2021} of the positivity-preserving analysis for various gradient flow models with singular energy potential. The linear expansive term is explicitly updated, for the sake of unique solvability, due to the negative eigenvalues involved. The surface diffusion term is implicitly treated, which comes from its convexity. Meanwhile, the other parts of the CHS system have to be handled very carefully, to ensure the desired theoretical properties. The convective term in the phase field dynamic equation is discretized in a semi-implicit way: explicit treatment for the phase variable and implicit treatment for the velocity vector. The static Stokes equation equation is implicitly computed, with the chemical potential determined by the convex splitting approach. The full numerical system turns out to be the gradient of a strictly convex energy functional, which in turn guarantees the unique solvability of the numerical solution. This symmetric feature represents a key difference between the current work and the related work in~\cite{chen22c}, in which the discretization of the Cahn-Hilliard-Navier-Stokes system leads to a non-symmetric numerical system, due to the fluid convection terms. As a result of the unique solvability and positivity-preserving property, the energy dissipation of the numerical scheme could be derived by a standard energy estimate.

In the present paper, an optimal rate convergence analysis and error estimate of the proposed numerical scheme are provided, which will be the first such result for the singular energy potential phase field model coupled with fluid motion. As illustrated by a few related existing works~\cite{chen19a, chen22b, chen16,  diegel17, liuY17} for the fluid-phase field coupled system with a polynomial approximation energy potential, the standard $\ell^\infty (0,T; \ell^2) \cap \ell^2 (0,T; H_h^2)$ error estimate does not work for the CHS system~\eqref{equation-CHS-1} -- \eqref{equation-CHS-4}, due to the lack of control for the highly nonlinear convection term. Instead, we  have to perform an $\ell^\infty (0,T; H_h^1) \cap \ell^2 (0,T; H_h^3 )$ error estimate, and such an estimate in a higher order Sobolev norm is necessary to make the error term associated with the nonlinear convection term have a non-positive inner product with the appropriate error test function. 

In addition to the positivity-preserving property, the  separation property of the numerical solution, i.e., a uniform distance between the numerical solution and the singular limit values (-1 and 1) is needed in the nonlinear error estimate. However, such a uniform bound is not directly available in any global-in-time analysis. To overcome this difficulty, a combination of rough and refined error (RRE) estimates must be applied. This RRE technique has been successfully applied to various nonlinear PDEs~\cite{duan22b, duan22a, duan20a, LiX2022b, LiX2022a, LiuC2021a}. In more details, a higher order asymptotic expansion, up to the second order temporal accuracy, is performed with a careful linearization technique. Such a higher order asymptotic expansion enables one to obtain a rough error estimate, so that the $\ell^\infty$ bound for the phase variable could be derived. This bound then plays a crucial role in the subsequent analysis. Namely, the refined error estimate is carried out to accomplish the desired convergence result. 

The rest of the paper is organized as follows. In Section 2, the standard finite difference spatial approximation is recalled. In Section 3, we propose the fully discrete finite difference scheme and establish the positivity-preserving property, unique solvability and unconditional energy stability. The convergence analysis of the numerical scheme, with first order temporal accuracy and second order spatial accuracy, is provided in Section 4. Some numerical experiments are presented in Section 5. Finally, some concluding remarks are given in Section 6.

	\section{The spatial discretization}

The standard centered finite difference spatial approximation is applied. We present the numerical approximation on the computational domain~$\Omega=(0, L_{x}) \times (0, L_{y})\times (0, L_{z})$~. The notation of two-dimensional domain could be naturally extended.  More relevant details and descriptions can be found in the related reference works~\cite{chen16, hu09, shen12, wise09a}. 

	\subsection{Basic definitions}

For simplicity, we consider $\Omega=(0, L_{x}) \times (0, L_{y})\times (0, L_{z})$, and assume that $h=L_{x}/N_{x}=L_{y}/N_{y}=L_{z}/N_{z}$, where $h$ is the spatial size, and  $N_x$, $N_y$, $N_z$ are given integers. We define the following:
	\begin{definition}
For any positive integer $N$, the following point sets are defined:
	\[
E_{N}:=\{i \cdot h | i=0, \ldots, N \}, \quad C_{N}:=\{(i-1 / 2) \cdot h | i=1, \ldots, N \},
	\]
	\[
C_{\bar{N}}:=\{(i-1 / 2) \cdot h | i=0, \ldots, N+1\}.
	\]
The two points belonging to $C_{\bar{N}} \setminus C_{N}$ are the so-called ghost points.

Define the function spaces
	\[
\mathcal{C}_{\Omega}:=\{\phi: C_{\bar{N}_{x}} \times C_{\bar{N}_{y}} \times C_{\bar{N}_{z}} \rightarrow \mathbb{R}\}, 
	\]
	\[
\mathcal{E}_{\Omega}^{x}:=\{\phi: E_{N_{x}} \times C_{N_{y}} \times C_{{N}_{z}} \rightarrow \mathbb{R}\},\quad
\mathcal{E}_{\Omega}^{y}:=\{\phi: C_{N_{x}} \times E_{N_{y}} \times E_{{N}_{z}} \rightarrow \mathbb{R}\}.
	\]
	\[
\mathcal{E}_{\Omega}^{z}:=\{\phi: C_{N_{x}} \times C_{N_{y}} \times E_{{N}_{z}} \rightarrow \mathbb{R}\}, \quad
\mathcal{E}_{\Omega}:=\mathcal{E}_{\Omega}^{x} \times \mathcal{E}_{\Omega}^{y}\times \mathcal{E}_{\Omega}^{z}.
	\]
	\end{definition}
The functions of $\mathcal{C}_{\Omega}$ are called cell-centered functions. In the component form, cell-centered functions are identified via~$\ \phi_{i,j,k}:=\phi(\xi_i, \xi_j, \xi_k)$, where~$\xi_i:=(i-\frac12)\cdot h$. The functions of $\mathcal{E}_{\Omega}^{x}$, etc., are called face-centered functions. In the component form, face-centered functions are identified via~$f_{i+\frac{1}{2},j,k}:= f(\xi_{i+\frac{1}{2}},\xi_j,\xi_k)$, etc.

The discrete boundary conditions, associated with cell-centered function and edge-centered function,  respectively, are proposed in following definition.
	\begin{definition}
A discrete function $\phi \in \mathcal{C}_{\Omega}$ is said to satisfy homogeneous Neumann boundary conditions, and we write $\boldsymbol{n} \cdot \nabla_{h} \phi=0$, iff $\phi$ satisfies
	\[
	\begin{array}{ll}
{\phi_{0,j,k}=\phi_{1,j,k},} & {\phi_{N_x,j,k}=\phi_{N_x+1,j,k},} 
	\\
{\phi_{i,0,k}=\phi_{i,1,k},} & {\phi_{i,N_y,k}=\phi_{i,N_y+1,k},} 
	\\
{\phi_{i,j,0}=\phi_{i,j,0},} & {\phi_{i,j,N_z}=\phi_{i,j,N_z+1}.}
	\end{array}
	\]
A discrete function $\boldsymbol{f}=(f^{x}, f^{y}, f^{z})^{T} \in \mathcal{E}_{\Omega}$ is said to satisfy no-penetration boundary conditions, $\boldsymbol{n}\cdot\boldsymbol{f}=0$, iff we have
	\[
	\begin{array}{ll}
{f_{1 / 2, j, k}^{x}=0,} & {f_{N_{x}+1 / 2, j,k}^{x}=0,} 
	\\
{f_{i, 1 / 2, k}^{y}=0,} & {f_{i, N_{y}+1 / 2,k}^{y}=0,} 
	\\
{f_{i, j, 1 / 2}^{z}=0,} & {f_{i, j,N_{k}+1 / 2}^{z}=0.} 
	\end{array}
	\]
\end{definition}

	\begin{definition}
	\label{free slip}
A discrete function $\boldsymbol{f}=(f^{x}, f^{y}, f^{z})^{T} \in \mathcal{E}_{\Omega}$ is said to satisfy free-slip boundary conditions iff we have
	\[
	\begin{array}{ll}
{f_{i+1 / 2, 0, k}^{x}=f_{i+1 / 2, 1, k}^{x},} & {f_{i+1 / 2, N_y+1, k}^{x}=f_{i+1 / 2, N_y, k}^{x},} \\
{f_{i+1 / 2, j, 0}^{x}=f_{i+1 / 2, j, 1}^{x},} & {f_{i+1 / 2, j, N_z+1}^{x}=f_{i+1 / 2, j, N_z}^{x},} \\
{f_{0, j+1/2, k}^{y}=f_{1, j+1/2, k}^{y},} & {f_{N_x+1, j+1/2, k}^{y}=f_{N_x, j+1/2, k}^{y},} \\
{f_{i, j+1/2, 0}^{y}=f_{i, j+1/2, 1}^{y},} & {f_{i, j+1/2, N_z+1}^{y}=f_{i, j+1/2, N_z}^{y},}\\
{f_{0, j, k+1/2}^{z}=f_{1, j, k+1/2}^{z},} & {f_{N_x+1, j, k+1/2}^{z}=f_{N_x, j, k+1/2}^{z},}\\
{f_{i, 0, k+1/2}^{z}=f_{i, 1, k+1/2}^{z},} & {f_{i, N_y+1, k+1/2}^{z}=f_{i, N_y, k+1/2}^{z}.}
	\end{array}
	\]
The two-dimensional notation is similar:
    \[
	\begin{array}{ll}
{f_{i+1/2, 0}^{x}=f_{i+1/2, 1}^{x},} & {f_{i+1/2, N_y+1}^{x}=f_{i+1/2, N_y}^{x},} \\
{f_{0, j+1/2}^{y}=f_{1, j+1/2}^{y},} & {f_{N_x+1, j+1/2}^{y}=f_{N_x, j+1/2}^{y}.}
	\end{array}
	\]
\end{definition}

	\subsection{Discrete operators, inner products, and norms}
The standard center difference operators are defined as follows:
	\begin{definition}
Define $d_{x}: \mathcal{E}_{\Omega}^{x} \rightarrow \mathcal{C}_{\Omega}$ component-wise via
	\[
d_{x} f_{i, j,k}:=\frac{1}{h}\left(f_{i+\frac{1}{2}, j,k}-f_{i-\frac{1}{2}, j,k}\right),
	\]
with $d_{y}: \mathcal{E}_{\Omega}^{y} \rightarrow \mathcal{C}_{\Omega} $ and $d_{z}: \mathcal{E}_{\Omega}^{z} \rightarrow \mathcal{C}_{\Omega} $ defined analogously. Then we have the discrete  divergence: 
	\[
\nabla_{h} \cdot: \mathcal{E}_{\Omega} \rightarrow \mathcal{C}_{\Omega}, \quad\quad \nabla_{h} \cdot \boldsymbol{f}:=d_{x} f^{x}+d_{y} f^{y}+d_{z} f^{z},
	\]
where $\boldsymbol{f}=(f^x, f^y,f^z)^T \in \mathcal{E}_{\Omega}$.

Define $A_{x}: \mathcal{C}_{\Omega} \rightarrow \mathcal{E}_{\Omega}^{x}$ component-wise via
	\[
A_{x} \phi_{i+\frac{1}{2}, j,k}:=\frac{1}{2}\left(\phi_{i+1, j,k}+\phi_{i, j,k}\right),
	\]
while $A_{y}: \mathcal{C}_{\Omega} \rightarrow \mathcal{E}_{\Omega}^{y}$ and $A_{z}: \mathcal{C}_{\Omega} \rightarrow \mathcal{E}_{\Omega}^{z}$ could be analogously introduced. Then we have a discrete average: 
	\[
A_{h}: \mathcal{C}_{\Omega} \rightarrow \mathcal{E}_{\Omega},\quad A_{h} \phi:=\left(A_{x} \phi, A_{y} \phi, A_{z} \phi \right)^{T} .
	\]

We define $D_{x}: \mathcal{C}_{\Omega} \rightarrow \mathcal{E}_{\Omega}^{x}$ component-wise via
	\[
D_{x} \phi_{i+\frac{1}{2}, j,k}:=\frac{1}{h}\left(\phi_{i+1, j,k}-\phi_{i, j,k}\right),
	\]
while $D_{y}: \mathcal{C}_{\Omega} \rightarrow \mathcal{E}_{\Omega}^{y}$ and $D_{z}: \mathcal{C}_{\Omega} \rightarrow \mathcal{E}_{\Omega}^{z}$ could be similarly introduced. The discrete gradient becomes 
	\[
\nabla_{h}: \mathcal{C}_{\Omega} \rightarrow \mathcal{E}_{\Omega},\quad\quad \nabla_{h} \phi:=\left(D_{x} \phi, D_{y} \phi, D_{z} \phi \right)^{T} .
	\]
The standard discrete Laplace operator is defined as
	\[
\Delta_{h}: \mathcal{C}_{\Omega} \rightarrow \mathcal{C}_{\Omega}, \quad\quad \Delta_{h} \phi:=\nabla_{h} \cdot \nabla_{h} \phi.
	\]	
	\end{definition}
	
	\begin{rem}
We can also define, in a straightforward way, the discrete Laplacian for face centered functions, $\Delta_h \mathcal{E}_\Omega^x \to \mathcal{E}_\Omega^x$, et cetera. For instance, if $g\in \mathcal{E}_\Omega^x$, then
	\[
\Delta_h g_{i+\hf,j,k} = \frac{g_{i+\frac{3}{2},j,k}+g_{i-\hf,j,k}+g_{i+\hf,j+1,k}+g_{i+\hf,j-1,k}+g_{i+\hf,j,k+1}+g_{i+\hf,j,k-1}  -6 g_{i+\hf,j,k} }{h^2}	,
	\]
and likewise for functions in $\mathcal{E}_\Omega^y$ and $\mathcal{E}_\Omega^z$.
	\end{rem}

Now we are ready to introduce the following grid inner products and norms.

	\begin{definition}
Define
	\[
(\phi, \psi):=h^3 \sum_{i=1}^{N_x} \sum_{j=1}^{N_y} \sum_{k=1}^{N_z} \phi_{i,j,k} \psi_{i,j,k}, \quad \forall \, \phi, \psi \in \mathcal{C}_{\Omega} , 
	\]
and
	\[
[f, g]_{x}:=\frac{1}{2} h^3 \sum_{i=1}^{N_x} \sum_{j=1}^{N_y} \sum_{k=1}^{N_z} (f_{i+\frac{1}{2}, j,k} g_{i+\frac{1}{2}, j,k}+f_{i-\frac{1}{2}, j,k} g_{i-\frac{1}{2}, j,k}), \quad \forall \, f,g \in \mathcal{E}_{\Omega}^x ,
	\]
with $[\cdot,\cdot]_y$ and $[\cdot,\cdot]_z$ formulated analogously.

For any $\boldsymbol{f}=(f^x,f^y,f^z)^T, \ \boldsymbol{g}=(g^x,g^y,g^z)^T \in \mathcal{E}_{\Omega}$, the discrete inner product becomes 
	\[
(\boldsymbol{f}, \boldsymbol{g}):=[f^{x}, g^{x}]_{x}+[f^{y}, g^{y}]_{y}+[f^{z}, g^{z}]_{z}.
	\]
	\end{definition}
	
	\begin{definition} 
For any $\boldsymbol{f}\in \mathcal{E}_{\Omega}$, we define the norm
	\[
\|\boldsymbol{f}\|_2 := \sqrt{(\boldsymbol{f},\boldsymbol{f})}.
	\]
In addition, for $\phi\in\mathcal{C}_{\Omega}$ we introduce the following norms:
	\[
\|\phi\|_{\infty}:=\max _{i, j,k}|\phi_{i, j,k}|,
	\]
	\[
\|\phi\|_p:=(|\phi|^p,1)^{\frac{1}{p}}, \quad 1\leq p < \infty,
	\]
	\[
\left\|\nabla_{h} \phi\right\|_{p}:=\left(\left[\left|D_{x} \phi\right|^{p}, 1\right]_{x}+\left[\left|D_{y} \phi\right|^{p}, 1\right]_{y}+\left[\left|D_{z} \phi\right|^{p}, 1\right]_{z}\right)^{\frac{1}{p}} , \quad 1\leq p < \infty.
	\]
	\end{definition}

Observe that $(\nabla_{h} \phi,\nabla_{h} \phi)=\|\nabla_{h} \phi\|_2^2$, for the case  $p=2$.

In addition, an $( \cdot , \cdot )_{-1,h}$ inner product and $\| \cdot \|_{-1, h}$ norm need to be introduced to facilitate the analysis in later sections. For any $\varphi  \in\mathring{\mathcal C}_{\Omega} := \left\{ f \in {\mathcal C}_{\Omega} \ \middle| \ ( f , 1 ) = 0 \right\}$, we define 
	\begin{equation} 
( \varphi_1, \varphi_2  )_{-1,h} = ( \varphi_1 ,  (-\Delta_h )^{-1} \varphi_2 ) , \quad \| \varphi  \|_{-1, h } = \sqrt{ ( \varphi ,  ( - \Delta_h )^{-1} (\varphi) ) } ,
	\end{equation} 
where the operator $\Delta_h$ is paired with discrete homogeneous Neumann boundary conditions.

We have the following Poincar\'{e}-type inequality:
	\begin{prop}
Suppose that $\Omega = (0,L)^3$, for simplicity. There is a constant $C>0$, independent of $h>0$, such that
	\[
\nrm{\phi}_2 \le C\nrm{\nabla_h\phi}_2,
	\]
for all $\phi\in \mathring{\mathcal C}_{\Omega} := \left\{ f \in {\mathcal C}_{\Omega} \ \middle| \ ( f , 1 ) = 0 \right\}$.
	\end{prop}

	\subsection{Summation by parts formulas and a discrete Sobolev embedding}
	
For $\phi,\psi \in \mathcal{C}_{\Omega}$ and a velocity vector field $\boldsymbol{u} \in \mathcal{E}_{\Omega}$, the following summation by parts formulas can be derived through standard calculations.
	\begin{lemma}
Suppose $\phi,\psi \in \mathcal{C}_{\Omega}$ and velocity vector field $\boldsymbol{u} \in \mathcal{E}_{\Omega}$.  If $\psi$ satisfies the homogeneous Neumann boundary conditions $\boldsymbol{n} \cdot \nabla_{h} \phi=0$, then 
	\[
(\phi,\Delta_h\psi)=-(\nabla_h\phi,\nabla_h\psi).
	\]
If $\boldsymbol{u}\cdot\boldsymbol{n}=0$ on the boundary, we have 
	\[
(\phi, \nabla_h\cdot\boldsymbol{u})=-(\nabla_h\phi, \boldsymbol{u}).
	\]
	\end{lemma}
	
The following discrete Sobolev inequality has been derived in the existing works~\cite{guan17a, guan14a, wang11a}, for the discrete grid function with periodic boundary condition; an extension to the discrete homogeneous Neumann boundary condition can be made in a similar fashion.

\begin{lemma} \cite{guan17a, guan14a, wang11a} \label{lem: Sobolev-1}
	For a grid function $ f \in \mathcal{C}_{\Omega}$ satisfying the discrete homogeneous Neumann boundary condition,  we have the following discrete Sobolev inequality:
	\begin{align}
	\| f \|_4 \le  C	\| f \|_{H_h^1} , \quad \mbox{with} \, \, \, 
	\| f \|_{H_h^1}^2 : = \| f \|_2^2 + \| \nabla_h f \|_2^2 , \label{Sobolev-1}
	\end{align}
in which the positive constant $C$ only depends on the domain $\Omega$. 
\end{lemma}

\section{The fully discrete numerical scheme}
For simplicity, we consider the cuboid $\Omega=(0,L)^3$ with $h=L/N$, for some $h>0$. Let $s=\frac{T}{M}>0$ be the time step size. The fully discrete scheme is proposed as follows: for $0\leq n \leq M-1$, given $\phi^n \in \mathcal{C}_{\Omega}$, find functions $(\phi^{n+1}, \mu^{n+1}, p^{n+1})\in [\mathcal{C}_{\Omega}]^3$, each satisfying the discrete homogeneous Neumann boundary conditions, and $\bu^{n+1}\in \mathcal{E}_\Omega$, satisfying discrete no-penetration and free-slip boundary conditions, such that
	\begin{eqnarray} 
  && 
  \phi^{n+1}-\phi^{n}=s \Delta_{h} \mu^{n+1}-s \nabla_{h} \cdot( A_h\phi^{n} \boldsymbol{u}^{n+1}),
  \label{discrete-CHS-0.1} 
\\
  && 
  \mu^{n+1}=\ln(1+\phi^{n+1})-\ln(1-\phi^{n+1})-\theta_0\phi^n-\varepsilon^{2} \Delta_{h} \phi^{n+1},
    \label{discrete-CHS-0.2}   
\\
  && 
  (-\Delta_h + I)\boldsymbol{u}^{n+1}+\nabla_h p^{n+1}+\gamma A_h\phi^n \nabla_h \mu^{n+1}=0 ,  
  \label{discrete-CHS-0.3}    \\
  &&  
  \nabla_h \cdot \boldsymbol{u}^{n+1}=0.  \label{discrete-CHS-0.4}
	\end{eqnarray}

	\subsection{Positivity-preserving property and unique solvability}
	
We begin this subsection with some preliminary definitions and results for the discrete version of the Stokes problem with no-penetration, free-slip boundary conditions.
	
	\begin{definition}
Suppose that $\Omega = (0,L)^3$ and $\f\in\mathcal{E}_\Omega$ satisfies discrete no-penetration boundary condition on $\partial\Omega$. Let $p\in\mathring{\mathcal{C}}_\Omega :=\{\phi \in \mathcal{C}_{\Omega} \ |\  (\phi,1 )=0\}$ be the unique solution to the problem
	\[
-\Delta_h p = \nabla_h\cdot \f 	,
	\]
subject to the discrete homogeneous Neumann boundary condition $\n \cdot \nabla_n p =0$. The discrete Helmholtz projection $\mathcal{P}_{H}^h:\left\{ \f\in \mathcal{E}_\Omega \ \middle| \ \f\cdot\n =0 \ \mbox{on} \ \partial\Omega \right\}\to \mathcal{E}_\Omega$ is defined as follows:
	\[
\mathcal{P}_{H}^h (\f) := \f+\nabla_h p .
	\]
	\end{definition}
	
The proof of the following facts are straightforward:
	\begin{lemma}
With the same assumptions as in the last definition, it follows that
	\[
\nabla_h \cdot 	\mathcal{P}_{H}^h (\f) = 0.
	\]
	\end{lemma}

	\begin{lemma}
Suppose that $\Omega = (0,L)^3$ and $\bu\in\mathcal{E}_\Omega$. Then
	\[
\nabla_h\cdot\left( \Delta_h\bu \right) = \Delta_h\left( \nabla_h \cdot \bu \right),
	\]
where the symbol $\Delta_h$ on the left is the discrete Laplacian whose domain is face-centered functions ($\mathcal{E}_\Omega^x$, $\mathcal{E}_\Omega^y$, and $\mathcal{E}_\Omega^z$) and the symbol $\Delta_h$ on the right is the discrete Laplacian whose domain is cell-centered functions ($\mathcal{C}_\Omega$).
	\end{lemma}
	
The proof of the following lemma uses standard facts about the MAC mesh points and  the previous few results.

	\begin{lemma}
Suppose that $\Omega = (0,L)^3$ and $\f\in\mathcal{E}_\Omega$ satisfies discrete no-penetration boundary conditions on $\partial\Omega$. Then the following two discrete problems are uniquely solvable and equivalent:
	\begin{enumerate}
	\item 
Find $\bu\in\mathcal{E}_\Omega$ that satisfies discrete no-penetration and discrete free-slip boundary conditions and $p\in\mathcal{C}_\Omega$ such that
	\begin{align*}
-\Delta_h \bu + \bu + \nabla_h p & = -\f ,
	\\
\nabla_h \cdot \bu & = 0.
	\end{align*}
	\item
Find $\bu\in\mathcal{E}_\Omega$ that satisfies discrete no-penetration and discrete free-slip boundary conditions such that
	\[
-\Delta_h \bu + \bu = -\mathcal{P}_H^h(\f).
	\]
	\end{enumerate}
	\end{lemma}

	\begin{lemma} 
	\label{lem: 3-1} 
For any $\phi^n \in \mathcal{C}_{\Omega}$, define a linear operator $\mathcal{L}_h: \mathring{\mathcal{C}}_{\Omega} \rightarrow \mathring{\mathcal{C}}_{\Omega}:=\{\phi \in \mathcal{C}_{\Omega} \ |\  (\phi,1 )=0\}$ via
	\begin{equation}
\mathcal{L}_h(\mu)=s \nabla_{h} \cdot( A_h\phi^{n} \boldsymbol{u}_\mu)-s \Delta_{h} \mu , 
	\label{Positivity-operator}
	\end{equation}
where $\boldsymbol{u}_\mu\in\mathcal{E}_\Omega$ is the unique vector grid function that satisfies discrete no-penetration and free-slip boundary conditions and the equation
	\begin{equation}
\mathcal{S}_h\bu_\mu = -\gamma \mathcal{P}_H^h(A_h\phi^{n} \nabla_h \mu),
	\label{discrete Stokes solver}
	\end{equation}
where $\mathcal{S}_h := -\Delta_h+I$. Then the following conclusions are valid: (i) For any $\phi \in \mathring{\mathcal{C}}_{\Omega}$, there is a unique $\mu \in \mathring{\mathcal{C}}_{\Omega}$ that satisfies $\mathcal{L}_h(\mu)=\phi$, and (ii) for any $\mu \in \mathring{\mathcal{C}}_{\Omega}$, we have $\|\mathcal{L}_h^{-1}(\mu)\|_\infty \leq Cs^{-1}h^{-3/2}\|\mu\|_\infty$. 
\end{lemma}

	\begin{proof}
Clearly $\mathcal{L}_h$ is linear. Given $\mu_1, \mu_2 \in \mathring{\mathcal{C}}_{\Omega},$ a careful calculation reveals that
	\begin{equation}
	\begin{aligned}
(\mu_1, \mathcal{L}_h(\mu_2))&=(\mu_1, s \nabla_{h} \cdot( A_h\phi^{n} \boldsymbol{u}_{\mu_2})-s \Delta_{h} \mu_2) 
	\\
& =s(\nabla_h \mu_1, \nabla_h \mu_2)-s(A_h \phi^n \nabla_h \mu_1, \boldsymbol{u}_{\mu_2}) 
	\\
& =s(\nabla_h \mu_1, \nabla_h \mu_2)+\frac{s}{\gamma}({\cal S}_h \boldsymbol{u}_{\mu_1} ,\boldsymbol{u}_{\mu_2})
	\\
& =s(\nabla_h \mu_1, \nabla_h \mu_2)+\frac{s}{\gamma}(\boldsymbol{u}_{\mu_1}, \boldsymbol{u}_{\mu_2})+\frac{s}{\gamma}(\nabla_h \boldsymbol{u}_{\mu_1}, \nabla_h \boldsymbol{u}_{\mu_2}) , 
	\end{aligned}
	\label{positivity-1-1}
	\end{equation}
where summation by parts formulas have been repeatedly applied. We conclude that the operator is symmetric:
	\[
(\mu_1, \mathcal{L}_h(\mu_2))=(\mathcal{L}_h(\mu_1), \mu_2).
	\]
The expansion \eqref{positivity-1-1} implies that 
	\begin{equation}
(\mu, \mathcal{L}_h(\mu))=s\|\nabla_h \mu\|_2^2+\frac{s}{\gamma}\|\boldsymbol{u}_ \mu\|_2^2+\frac{s}{\gamma}\|\nabla_h \boldsymbol{u}_ \mu\|_2^2 \geq s\|\nabla_h \mu\|_2^2 \geq sC_1^2 \|\mu\|_2^2,
	\label{positivity-1-2}
	\end{equation}
where $C_1$ is the constant associated with the discrete Poincar\'e inequality. Thus $\mathcal{L}_h$ is SPD on the space $\mathring{\mathcal{C}}_\Omega$ and is, therefore, invertible.

Furthermore, equation \eqref{positivity-1-2} reveals that
	\[
\lambda_{min}(\mathcal{L}_h) \geq sC_1^2 \quad \mbox{and} \quad \lambda_{max}(\mathcal{L}^{-1}_h) \leq \frac{1}{sC_1^2},
	\]
where $\lambda_{min}, \lambda_{max}$ refers to the smallest and largest positive eigenvalues of a symmetric, positive definite operator. Then we get 
$$
  \|\mathcal{L}^{-1}_h(\mu)\|_2 \leq \frac{1}{sC_1^2} \|\mu\|_2, 
$$ 
for any $\mu \in \mathring{\mathcal{C}}_{\Omega}$. By the 3-D inverse inequality, the following result is obtained: 
	\[
\left\|\mathcal{L}_{h}^{-1}(\mu)\right\|_{\infty} \leq \frac{C_2\left\|\mathcal{L}_{h}^{-1}(\mu)\right\|_{2}}{h^{3/2}} \leq C_2s^{-1}C_{1}^{-2} h^{-3/2}\|\mu\|_{2} \leq C_2 s^{-1} C_{1}^{-2}|\Omega|^{\frac{1}{2}} h^{-3/2}\|\mu\|_{\infty}, 
	\]
where the last step comes from an obvious fact, $\|f\|_{2} \leq|\Omega|^{\frac{1}{2}}\|f\|_{\infty}$. The proof is complete.
\end{proof}

The positivity-preserving and unique solvability properties are established in the following theorem. 

	\begin{theorem}
Assume that  $\phi^n\in\mathcal{C}_\Omega$ is given, with $\| \phi^n \|_\infty \leq M$ and $-1< \overline{\phi^n} =:\beta < 1$. There exists a unique  solution $\phi^{n+1} \in \mathcal{C}_{\Omega}$ to \eqref{discrete-CHS-0.1} -- \eqref{discrete-CHS-0.4}, with $(\phi^{n+1}-\phi_0,1)=0$ and $\left\|\phi^{n+1}\right\|_{\infty}<1$.
	\end{theorem}
	\begin{proof}
For any $\phi \in A_h:=\left\{\phi\in\mathcal{C}_\Omega\ \middle| \ \left\| \phi \right\|_\infty \leq 1, \ (\phi-\beta,1)=0\right\}$, define
	\[
\mathcal{J}(\phi):=(\mathcal{L}^{-1}_h(\phi-\phi^n),\phi-\phi^n)+(1+\phi,\ln(1+\phi))+(1-\phi,\ln(1-\phi))+\frac{\varepsilon^{2}}{2}\|\nabla_{h}\phi\|_2^2-\theta_0(\phi,\phi^n).
	\]
The solution of the numerical scheme is a minimizer of this discrete functional. Subsequently, we define	 
	\[
\mathcal{F}(\psi):=\mathcal{J}(\psi+\beta), \quad \forall \, \psi\in \mathring{A}_h:=\left\{ \psi\in\mathcal{C}_\Omega \ \middle| \ (\phi,1)=0, \ -1-\beta \leq \psi \leq 1-\beta \right\}.
	\]
It is clear that, if $\psi \in \mathring{A}_h$ minimizes $\mathcal{F}$, then $\psi+\beta \in A_h$ minimizes $\mathcal{J}$.

Next, let us define the following closed domain:
	\[
\mathring{A}_{h,\delta} := \left\{ \psi\in \mathcal{C}_\Omega \ \middle| \  (\psi,1)=0, \ -1-\beta+\delta \leq \psi \leq 1-\beta-\delta \right\},
	\]
where $\delta \in (0,1/2)$ and is sufficiently small. Since $\mathring{A}_{h,\delta}$ is a bounded, compact and convex set in $\mathring{\mathcal{C}}_{\Omega}$, there exists a (not necessarily unique) minimizer of $\mathcal{F}$ over $\mathring{A}_{h,\delta}$. The key point of the positivity analysis is that such a minimizer could not occur on the boundary of $\mathring{A}_{h,\delta}$, if $\delta$ is sufficiently small. To be more explicit, by the boundary of $\mathring{A}_{h,\delta}$, we mean the locus of points $\psi \in \mathring{A}_{h,\delta}$ such that $\left\|\psi+\beta\right\|_{\infty}=1-\delta$.

To get a contradiction, suppose that the minimizer of $\mathcal{F}$, call it $\phi^*$, occurs at a boundary point of $\mathring{A}_{h,\delta}$. There is at least one grid point $\vec{\alpha}_{0}=\left(i_{0}, j_{0}, k_{0}\right)$ such that $|\phi^*_{\vec{\alpha}_{0}}+\beta|=1-\delta$. First, we assume that $\phi^*_{\vec{\alpha}_{0}}+\beta=\delta-1$, so that the grid function $\phi^*$ has a global minimum at $\vec{\alpha}_{0}$. Suppose that $\phi^*$ achieves its maximum at $\vec{\alpha}_{1}=\left(i_{1}, j_{1}, k_{1}\right)$. By the fact that $\bar{\phi^*}=0$, we have $\phi^*_{\vec{\alpha}_{1}} \geq 0$ and then 
\begin{equation}
1-\delta \geq \varphi_{\vec{\alpha}_{1}}^{\star}+\beta \geq \beta.\label{positivity-2-1}
\end{equation}
Since $\mathcal{F}$ is smooth over $\mathring{A}_{h,\delta}$, for all $\psi \in \mathring{\mathcal{C}}_{\Omega}$, the directional derivative turns out to be 
\begin{equation}
\begin{aligned}
d_s\mathcal{F}(\phi^*+s\psi)|_{s=0}&=(\ln(1+\phi^*+\beta)-\ln(1-\phi^*-\beta),\psi)-(\theta_0\phi^n+\varepsilon^{2}\Delta_h\phi^*,\psi)\\
&+(\mathcal{L}^{-1}_h(\phi^*-\phi^n+\beta),\psi).
\end{aligned}
\label{positivity-2-2}
\end{equation}
Pick the direction $\psi$ as
$$
\psi_{i,j,k}=\delta_{i, i_{0}} \delta_{j, j_{0}}\delta_{k, k_{0}}-\delta_{i, i_{1}} \delta_{j, j_{1}}\delta_{k, k_{1}}.
$$
Then the derivative can be expressed as
	\begin{equation}
	\begin{aligned}
\frac{1}{h^3} d_s\mathcal{F}(\phi^*+s\psi)|_{s=0}=&\ln(1+\phi^*_{\vec{\alpha}_0}+\beta)-\ln(1-\phi^*_{\vec{\alpha}_0}-\beta)-\ln(1+\phi^*_{\vec{\alpha}_1}+\beta)+\ln(1-\phi^*_{\vec{\alpha}_1}-\beta)
	\\
&-\theta_0(\phi^n_{\vec{\alpha}_0}-\phi^n_{\vec{\alpha}_1})-\varepsilon^{2}(\Delta_h \phi^*_{\vec{\alpha}_0}-\Delta \phi^*_{\vec{\alpha}_1})
	\\
&+\mathcal{L}^{-1}_h(\phi^*-\phi^n+\beta)_{\vec{\alpha}_0}-\mathcal{L}^{-1}_h(\phi^*-\phi^n+\beta)_{\vec{\alpha}_1}.
	\end{aligned}
	\label{positivity-2-3}
	\end{equation}
By the fact that $\beta+\phi^*_{\vec{\alpha}_0}=-1+\delta$ and \eqref{positivity-2-1}, we have
	\begin{equation}
\ln(1+\phi^*_{\vec{\alpha}_0}+\beta)-\ln(1-\phi^*_{\vec{\alpha}_0}-\beta)-\ln(1+\phi^*_{\vec{\alpha}_1}+\beta)+\ln(1-\phi^*_{\vec{\alpha}_1}-\beta) \leq \ln\frac{\delta}{2-\delta}-\ln\frac{1+\beta}{1-\beta}.
	\label{positivity-2-4}
	\end{equation}
Since $\phi^*$ takes a minimum at the grid point $\vec{\alpha}_0$ and a maximum at the grid point $\vec{\alpha}_1$, it is obvious that 
	\begin{equation}
\Delta_h \phi^*_{\vec{\alpha}_0} \geq 0, \quad \Delta_h \phi^*_{\vec{\alpha}_1} \leq 0, \ \Longrightarrow -\varepsilon^{2}(\Delta_h \phi^*_{\vec{\alpha}_0}-\Delta \phi^*_{\vec{\alpha}_1}) \leq 0.
	\label{positivity-2-5}
	\end{equation}
By the assumption that $\| \phi^n \| \leq M$, the following inequality is straightforward: 
	\begin{equation}
-2M\leq \phi^n_{\vec{\alpha}_0}-\phi^n_{\vec{\alpha}_1} \leq 2M.\label{positivity-2-6}
	\end{equation}
Setting $\mu^*=\mathcal{L}^{-1}_h(\phi^*-\phi^n+\beta)$, we obtain 
	\begin{equation}
	\begin{aligned}
\| \mathcal{L}_h(\mu^*) \|_\infty &= s\|  \nabla_{h} \cdot ((1+\gamma(A_h\phi^n)^2) \nabla_{h} \mu^*)+ \nabla_{h} \cdot ( A_h\phi^n \nabla_{h} p_{\mu^*} ) \|_\infty
	\\
&=\| \phi^*-\phi^n+\beta \|_\infty 
	\\
& \leq M+1.
	\end{aligned}
	\label{positivity-2-7}
	\end{equation}
Therefore, an application of Lemma~\ref{lem: 3-1} implies that 
	\begin{equation}
\| \mathcal{L}^{-1}_h(\phi^*-\phi^n+\beta) \|_\infty=\| \mu^* \|_\infty \leq Cs^{-1}h^{-3/2}\|\mathcal{L}_h(\mu^*)\|_\infty \leq Cs^{-1}h^{-3/2}(M+1).
	\label{positivity-2-8}
	\end{equation}
A combination of \eqref{positivity-2-3} to \eqref{positivity-2-8} leads to  
	\begin{equation}
\frac{1}{h^3} d_s\mathcal{F}(\phi^*+s\psi)|_{s=0} \leq \ln\frac{\delta}{2-\delta}-\ln\frac{1+\beta}{1-\beta}+2\theta_0 M+2Cs^{-1}h^{-3/2}(M+1).
	\label{positivity-2-9}
	\end{equation}
Notice that right hand side of \eqref{positivity-2-9} is singular as $s,h \ \rightarrow 0$. Meanwhile, for any fixed $s, h>0$, we may choose $\delta \in (0, 1/2)$ sufficiently small so that
	\begin{equation}
d_s\mathcal{F}(\phi^*+s\psi)|_{s=0} < 0 . 
	\end{equation}
This contradicts the assumption that $\mathcal{F}$ has a minimum at $\phi^*$, since the directional derivative is negative in a direction pointing into the interior of $\mathring{A}_{h,\delta}$.

Using similar arguments, we can also prove that the global minimum of $\mathcal{F}$ over $\mathring{A}_{h,\delta}$ could not occur at a boudary point $\phi^*$ such that $\phi^*_{\vec{\alpha}_0}+\beta=1-\delta$. The details are left to interested readers.

A combination of these two facts reveals that the global minimum of $\mathcal{F}$ over $\mathring{A}_{h,\delta}$ could only possibly occur at interior point if $\delta$ sufficiently small. Therefore, there must be a solution $\phi+\beta \in A_h$ that minimizes $\mathcal{J}$ over $A_h$, which is equivalent to the numerical solution of \eqref{discrete-CHS-0.1} -- \eqref{discrete-CHS-0.4}. The existence of the numerical solution is established.

Finally, since $\mathcal{J}$ is a strictly convex function over $A_h$, the uniqueness analysis of numerical solution is straightforward.
	\end{proof}

	\subsection{Unconditional energy stability}
Now we establish an unconditional energy stability of the proposed numerical scheme. For any $\phi \in \mathcal{C}_{\Omega}$, its discrete energy is defined as
	\[
F_h(\phi) =\left( (1+\phi)\ln(1+\phi)+(1-\phi)\ln(1-\phi)-\frac{\theta_0}{2}\phi^2,1 \right)+\frac{\varepsilon^{2}}{2}\|\nabla_{h}\phi\|_2^2.
	\]
The following discrete energy dissipation result is valid.
	\begin{theorem} 
	\label{thm: ener stab} 
Numerical solutions of \eqref{discrete-CHS-0.1} -- \eqref{discrete-CHS-0.4} are unconditionally energy stable in the sense that  
	\begin{equation}
F_h(\phi^{n+1})-F_h(\phi^n)\leq -\frac{\varepsilon^{2}}{2}\|\nabla_{h}(\phi^{n+1}-\phi^n)\|_2^2-s\|\nabla_{h}\mu^{n+1}\|_2^2-\frac{s}{\gamma}(\|\boldsymbol{u}^{n+1}\|_2^2+\|\nabla_h \boldsymbol{u}^{n+1}\|_2^2) . \label{energy-0}
	\end{equation}
	\end{theorem}

	\begin{proof}
The following definitions are introduced for simplicity of presentation: 
	\begin{equation}
	\begin{aligned}
&
G(\phi) :=F_1(\phi)-F_2(\phi) , 
	\\
&
F_1(\phi):=(1+\phi)\ln(1+\phi)+(1-\phi)\ln(1-\phi),\quad F_2(\phi):=\frac{\theta_0}{2}\phi^2 , 
	\\
&
f_1 (\phi) := F_1' (\phi) =\ln(1+\phi)-\ln(1-\phi),\quad f_2 (\phi) := F_2' (\phi) =\theta_0\phi . 
	\end{aligned}
	\end{equation}
Taking an inner production with \eqref{discrete-CHS-0.1} by the chemical potential $\mu^{n+1}$ gives 
\begin{equation}
\begin{aligned}
\frac{1}{s}(\phi^{n+1}-\phi^{n},\mu^{n+1})&=(\Delta_h\mu^{n+1},\mu^{n+1})-(\nabla_{h}\cdot(A_h\phi^n\boldsymbol{u}^{n+1}),\mu^{n+1}) \\
& =-\|\nabla_{h}\mu^{n+1}\|_2^2-(\nabla_{h}\cdot(A_h\phi^n\boldsymbol{u}^{n+1}),\mu^{n+1}).
\end{aligned}\label{energy-1}
\end{equation}
Meanwhile, taking an inner production with \eqref{discrete-CHS-0.2} by $\phi^{n+1}-\phi^{n}$ yields 
\begin{equation}
\begin{aligned}
&(\phi^{n+1}-\phi^n,\mu^{n+1})\\
&=(f_1(\phi^{n+1})-f_2(\phi^n),\phi^{n+1}-\phi^n)-(\varepsilon^{2} \Delta_{h} \phi^{n+1},\phi^{n+1}-\phi^n) \\
&=(f_1(\phi^{n+1})-f_2(\phi^n),\phi^{n+1}-\phi^n)+\varepsilon^{2}(\nabla_{h}\phi^{n+1},\nabla_{h}(\phi^{n+1}-\phi^n))\\
&=(f_1(\phi^{n+1})-f_2(\phi^n),\phi^{n+1}-\phi^n)+\frac{\varepsilon^{2}}{2}(\|\nabla_{h}(\phi^{n+1}-\phi^n)\|_2^2+\|\nabla_{h}\phi^{n+1}\|_2^2-\|\nabla_{h}\phi^n\|_2^2).
\end{aligned}
\label{energy-2}
\end{equation}
A combination of \eqref{energy-1} and \eqref{energy-2} results in 
\begin{equation}
\begin{aligned}
(f_1(\phi^{n+1})-&f_2(\phi^n),\phi^{n+1}-\phi^n)+\frac{\varepsilon^{2}}{2}(\|\nabla_{h}(\phi^{n+1}-\phi^n)\|_2^2+\|\nabla_{h}\phi^{n+1}\|_2^2-\|\nabla_{h}\phi^n\|_2^2)\\
&+s\|\nabla_{h}\mu^{n+1}\|_2^2-s(A_h\phi^n\boldsymbol{u}^{n+1},\nabla_{h}\mu^{n+1})=0.
\end{aligned}
\label{energy-3}
\end{equation}
On the other hand, the convexity of $F_1$ and $F_2$ reveals the following inequalities: 
	\begin{equation}
F_{1}(\phi^{n+1})-F_{1}(\phi^{n})  \le f_{1}\left(\phi^{n+1}\right)\left(\phi^{n+1}-\phi^{n}\right) , 
\label{energy-4}
\end{equation}
\begin{equation}
F_{2}(\phi^{n+1})-F_{2}(\phi^{n}) \ge f_{2} (\phi^{n})\left(\phi^{n+1}-\phi^{n}\right) , 
\label{energy-5}
\end{equation}
which in turn lead to 
	\begin{equation}
(G(\phi^{n+1})-G(\phi^{n}), 1) \leq (f_{1}(\phi^{n+1})-f_{2}(\phi^{n}), \phi^{n+1}-\phi^{n}).
	\label{energy-6}
	\end{equation}
As a result, a combination of~\eqref{energy-1} and \eqref{energy-6} implies that 
	\begin{equation}
	\begin{aligned}
(G(\phi^{n+1})-G(\phi^{n}), 1)&+\frac{\varepsilon^{2}}{2}(\|\nabla_{h}(\phi^{n+1}-\phi^n)\|_2^2+\|\nabla_{h}\phi^{n+1}\|_2^2-\|\nabla_{h}\phi^n\|_2^2)
	\\
&+s\|\nabla_{h}\mu^{n+1}\|_2^2-s(A_h\phi^n\boldsymbol{u}^{n+1},\nabla_{h}\mu^{n+1}) 
	\\
& \leq 0.
	\end{aligned}
	\label{energy-7}
	\end{equation}
Finally, the following estimate could be derived: 
\begin{equation}
\begin{aligned}
F_h(\phi^{n+1})-F_h(\phi^{n})& \leq -\frac{\varepsilon^{2}}{2}\|\nabla_{h}(\phi^{n+1}-\phi^n)\|_2^2-s\|\nabla_{h}\mu^{n+1}\|_2^2+s(\boldsymbol{u}^{n+1},A_h\phi^n\nabla_{h}\mu^{n+1})
	\\
&=-\frac{\varepsilon^{2}}{2}\|\nabla_{h}(\phi^{n+1}-\phi^n)\|_2^2-s\|\nabla_{h}\mu^{n+1}\|_2^2-\frac{s}{\gamma}(\boldsymbol{u}^{n+1}, {\cal S}_h\boldsymbol{u}^{n+1}+\nabla_h p^{n+1})
	\\
&=-\frac{\varepsilon^{2}}{2}\|\nabla_{h}(\phi^{n+1}-\phi^n)\|_2^2-s\|\nabla_{h}\mu^{n+1}\|_2^2-\frac{s}{\gamma}(\|\boldsymbol{u}^{n+1}\|_2^2+\|\nabla_h\boldsymbol{u}^{n+1}\|_2^2) , 
	\end{aligned}
	\label{energy-8}
	\end{equation}
where summation by parts formulas have been repeatedly applied. This finishes the proof.
	\end{proof}

\section{Convergence analysis}
Now we proceed into the convergence analysis. Let $(\Phi, \boldsymbol{U}, P)$ be the exact PDE solution for the CHS system \eqref{equation-CHS-1} -- \eqref{equation-CHS-4}. With sufficiently regular initial data, it is reasonable to assume that the exact solution has regularity of class $\mathcal{R}$, where
\begin{equation}
\Phi \in \mathcal{R} := H^4 \left(0,T; C(\Omega)\right) \cap H^3 \left(0,T; C^2(\Omega)\right) \cap L^\infty \left(0,T; C^6(\Omega)\right).\label{regularity}
\end{equation}
In addition, we assume that the following separation property is valid for the exact solution:
\begin{equation} 
1 + \Phi  \ge \epsilon_0 ,\quad  1 - \Phi \ge \epsilon_0 ,  \quad \mbox{for some $\epsilon_0 > 0$, at a point-wise level} . \label{separation property}
\end{equation}  

Define $\Phi_N(\cdot,t)=\mathcal{P}_N\Phi(\cdot,t)$, $\boldsymbol{U}_N(\cdot,t)=\mathcal{P}_N\boldsymbol{U}(\cdot,t)$, $P_N(\cdot,t)=\mathcal{P}_NP(\cdot,t)$, the spatial Fourier projection of the exact solution into ${\cal B}^K$, the space of trigonometric polynomials of degree to and including  $K$ with $N=2K +1$. The following projection approximation is standard: if $(\Phi, \boldsymbol{U}, P) \in L^\infty(0,T;H^\ell_{\rm per}(\Omega))$, for any $\ell\in\mathbb{N}$ with $0 \le k \le \ell$, 
\begin{equation} 
\begin{aligned} 
  & 
\| \Phi_N - \Phi\|_{L^\infty(0,T;H^k)} \leq C h^{\ell-k} \| \Phi \|_{L^\infty(0,T;H^\ell)},  \, \,  \| \boldsymbol{U}_N - \boldsymbol{U} \|_{L^\infty(0,T;H^k)} \le C h^{\ell-k} \| \boldsymbol{U} \|_{L^\infty(0,T;H^\ell)} , 
\\
  & 
\| P_N - P \|_{L^\infty(0,T;H^k)} \le C h^{\ell-k} \| P \|_{L^\infty(0,T;H^\ell)} .  
\end{aligned} \label{Fourier-approximation}
\end{equation}
In fact, the Fourier projection estimate does not automatically preserve the positivity of $1 + \Phi_N$ and $1 - \Phi_N$; on the other hand, we could enforce the phase separation property that $1 + \Phi_N \ge \frac{1}{2} \epsilon_0$, $1 - \Phi_N \ge \frac{1}{2} \epsilon_0$,  if $h$ is taken sufficiently small. 

We denote $\Phi_N(\cdot, t_n)$ by $\Phi^n_N$. Since $\Phi^n_N\in {\cal B}^K$, the mass conservative property is available at the discrete level:
	\begin{equation} 
\overline{ \Phi_N^n} = \frac{1}{|\Omega|}\int_\Omega \, \Phi_N ( \cdot, t_n) \, d {\bf x} = \frac{1}{|\Omega|}\int_\Omega \, \Phi_N ( \cdot, t_{n-1}) \, d {\bf x} = \overline{ \Phi_N^{n-1}} ,\label{mass conservative_1}
	\end{equation}
for any $n\in \mathbb{N}$. On the other hand, the numerical solution \eqref{discrete-CHS-0.1} is also mass conservative at the discrete level:
\begin{equation} 
\overline{\phi^n} = \overline{\phi^{n-1}} ,  \, \, \, 
\overline{\phi^n} = \overline{\phi^{n-1}} , \quad \forall \ n \in \mathbb{N} . \label{mass conservative_2}
\end{equation}
In turn, the error grid function is defined as
\begin{equation}
e_\phi^n := \mathcal{P}_N \Phi_N^n - \phi^n,  \, \, \, 
e_{\boldsymbol{v}}^n := \mathcal{P}_N \boldsymbol{U}_N^n - \boldsymbol{u}^n , \, \, \, 
e_p^n := \mathcal{P}_N P_N^n - p^n , 
\quad \forall \ n \in N .
\label{error function-1}
\end{equation}
It follows that  $\overline{e_\phi^n} =0$, for any $n \in N$,  so that the discrete norm $\| \, \cdot \, \|_{-1,h}$ is well defined for the error grid function $e_\phi^n$.

The following theorem is the main result of this section.
\begin{theorem}\label{convergernce theorem}
Given initial data $\Phi(\cdot,t=0) \in C^6(\Omega)$, suppose the exact solution for CHS equation \eqref{equation-CHS-1} -- \eqref{equation-CHS-4} is of regularity class $\mathcal{R}$. Provided that $s$ and $h$ are sufficiently small and under a requirement $C_1h \leq s \leq C_2h$, we have
\begin{equation}
\|\nabla_h e_\phi^n\|_2+(s\sum^n_{k=1}\|\nabla_h \Delta_h e_\phi^k\|_2^2)^{\frac{1}{2}} 
 \leq C(s+h^2) , 
\label{convergernce theorem-1}
\end{equation}
for all positive integers $n$, such that $t_n=n \cdot s< T$, where $C>0$ is independent of $s$, $h$ and n. 
\end{theorem}

\subsection{Higher order truncation error estimate}
By consistency, the projection solution $(\Phi_N, \boldsymbol{U}_N, P_N)$ solves the discrete equations \eqref{discrete-CHS-0.1} -- \eqref{discrete-CHS-0.4} with a first order accuracy in time and second order accuracy in space. Meanwhile, it is observed that this leading local truncation error will not be sufficient to obtain an $\ell^\infty$ bound for the numerical solution to recover the separation property, as well as a $W_h^{1,4}$ bound to pass through the convergence estimate. To overcome this difficulty, we build a higher order consistency analysis via a perturbation term. In more details, we need to construct supplementary fields $\Phi_{\Delta t}$, $\boldsymbol{U}_{\Delta t}$, $P_{\Delta t}$ and define the following profiles
\begin{equation}
 \hat{\Phi} = \Phi_N + s \mathcal{P}_N \Phi_{\Delta t} , \, \,  \,  
  \hat{\boldsymbol{U}} = \mathcal{PH}_h ( \boldsymbol{U}_N + s \mathcal{P}_N \boldsymbol{U}_{\Delta t} ) ,   \, \, \,  
 \hat{P} = P_N + s \mathcal{P}_N P_{\Delta t} ,  
 \label{perturbation term}
\end{equation}
in which a special interpolation operator $\mathcal{PH}_h$, which will be introduced later, enforces the divergence-free condition at a discrete level. 

The following truncation error analysis for the temporal discretization can be obtained by using a straightforward Taylor expansion as well as estimate \eqref{Fourier-approximation} for the projection solution:
\begin{eqnarray}
&&
\frac{\Phi_N^{n+1}-\Phi_N^{n}}{s}=\Delta\mathcal{V}_N^{n+1} -\nabla\cdot(\Phi_N^{n}\boldsymbol{U}_N^{n+1})+s (G_{\phi}^{(0)})^n+O(s^2)+O(h^{m_0}),
\label{consistency-2.1}\\
&&
\mathcal{V}_N^{n+1}=\ln(1+\Phi_N^{n+1})-\ln(1-\Phi_N^{n+1})-\theta_0 \Phi_N^n-\varepsilon^{2}\Delta \Phi_N^{n+1},
\label{consistency-2.2}\\
&&
(-\Delta+I)\boldsymbol{U}_N^{n+1}=-\nabla P_N^{n+1}-\gamma \Phi_N^n \nabla\mathcal{V}_N^{n+1}+s (G_v^{(0)})^n+O(\Delta t^2)+O(h^{m_0}),\label{consistency-2.3}\\
&&
\nabla \cdot \boldsymbol{U}_N^{n+1}=0.\label{consistency-2.4}
\end{eqnarray}
Here $m_0\geq 4$ and $G_{\phi}^{(0)}$, $G_{v}^{(0)}$ can be assumed to be smooth enough in the sense that their derivatives are bounded.

The correction function $(\Phi_{\Delta t}, \boldsymbol{U}_{\Delta t}, P_{\Delta t})$ is given by solving the following equation:
\begin{eqnarray}
&&
\partial_t \Phi_{\Delta t}  = 
   - \nabla \cdot ( \Phi_N \boldsymbol{U}_{\Delta t} + \Phi_{\Delta t} \boldsymbol{U}_N ) 
 + \Delta \mathcal{V}_{\Delta t} - G_\phi^{(0)}  , 
     \label{consistency-3.1} \\
&&
  \mathcal{V}_{\Delta t} =  \frac{\Phi_{\Delta t} }{1 + \Phi_N } 
   + \frac{\Phi_{\Delta t} }{1 - \Phi_N }     
 - \theta_0 \Phi_{\Delta t} 
   - \varepsilon^2 \Delta \Phi_{\Delta t} ,    \label{consistency-3.2}  \\
&&
  (-\Delta+I)\boldsymbol{U}_{\Delta t} =  - \nabla P_{\Delta t} 
  - \gamma (  \Phi_N \nabla \mathcal{V}_{\Delta t} 
  + \Phi_{\Delta t} \nabla \mathcal{V}_N ) - G_v^{(0)} , 
   \quad 
   \nabla \cdot \boldsymbol{U}_{\Delta t} = 0 .  \label{consistency-3.3}  
\end{eqnarray}
Existence of a solution of the above linear, convection-diffusion type PDE is straightforward. Since the correction function depends only on the projection solution $(\Phi_N, \boldsymbol{U}_N, P_N)$ with enough regularity, the derivatives of $(\Phi_{\Delta t}, \boldsymbol{U}_{\Delta t}, P_{\Delta t})$ in various orders are bounded. Subsequently, an application of the semi-implicit discretization implies that
\begin{eqnarray}
&&
\frac{\Phi_{\Delta t}^{n+1}-\Phi_{\Delta t}^{n}}{s}=-\nabla\cdot(\Phi_N^n \boldsymbol{U}_{\Delta t}^{n+1} + \Phi_{\Delta t}^n \boldsymbol{U}_N^{n+1} ) 
 + \Delta \mathcal{V}^{n+1}_{\Delta t} - (G_\phi^{(0)})^n+O(s) , 
\label{consistency-4.1}\\
&&
\mathcal{V}_{\Delta t}^{n+1}=\frac{\Phi_{\Delta t}^{n+1} }{1 + \Phi_N^{n+1} } 
   + \frac{\Phi_{\Delta t}^{n+1} }{1 - \Phi_N^{n+1} }     
 - \theta_0 \Phi_{\Delta t}^n
   - \varepsilon^2 \Delta \Phi_{\Delta t}^{n+1} ,
\label{consistency-4.2}\\
&&
(-\Delta+I)\boldsymbol{U}_{\Delta t}^{n+1} =  - \nabla P_{\Delta t}^{n+1} 
  - \gamma (  \Phi_N^n \nabla \mathcal{V}_{\Delta t}^{n+1} 
  + \Phi_{\Delta t}^n \nabla \mathcal{V}_N^{n+1} ) - (G_v^{(0)})^n+O(s), \label{consistency-4.3}\\
&&
   \nabla \cdot \boldsymbol{U}_{\Delta t}^{n+1} = 0. \label{consistency-4.4}
\end{eqnarray}
Therefore, a combination of \eqref{consistency-2.1}-\eqref{consistency-2.4} and \eqref{consistency-4.1}-\eqref{consistency-4.4} leads to a second order temporal truncation error of $ \hat{\Phi}_1 = \Phi_N + s \mathcal{P}_N \Phi_{\Delta t},\ 
  \hat{\boldsymbol{U}}_1 =  \boldsymbol{U}_N + s \mathcal{P}_N \boldsymbol{U}_{\Delta t}, \   
 \hat{P}_1 = P_N + s \mathcal{P}_N P_{\Delta t}$:
\begin{eqnarray}
&&
\frac{\hat{\Phi}_1^{n+1}-\hat{\Phi}_1^n}{s}=-\nabla\cdot(\hat{\Phi}_1^n \hat{\boldsymbol{U}}_1^{n+1} + \hat{\Phi}_1^n \hat{\boldsymbol{U}}_1^{n+1} ) 
 + \Delta \hat{\mathcal{V}}^{n+1}_1 + O(s^2) , 
\label{consistency-5.1}\\
&&
\hat{\mathcal{V}}^{n+1}_1=\ln(1+\hat{\Phi}_1^{n+1})-\ln(1-\hat{\Phi}_1^{n+1})-\theta_0\hat{\Phi}_1^n-\varepsilon^2\Delta \hat{\Phi}^{n+1}_1,
\label{consistency-5.2}\\
&&
(-\Delta+I)\hat{\boldsymbol{U}}_1^{n+1} =  - \nabla \hat{P}_1^{n+1} 
  - \gamma (  \hat{\Phi}_1^n \nabla \hat{\mathcal{V}}^{n+1}_1)+O(s^2), 
   \quad 
   \nabla \cdot \hat{\boldsymbol{U}}_1^{n+1} = 0. \label{consistency-5.3}
\end{eqnarray}
In the derivation of \eqref{consistency-5.1}-\eqref{consistency-5.3}, the following linearized expansions have been utilized:
\begin{eqnarray}
&&
\ln ( 1 \pm \hat{\Phi}_1 )  = \ln ( 1 \pm \Phi_N \pm \hat{\Phi}_{\Delta t}  )  
	=  \ln ( 1 \pm \Phi_N ) + \frac{\hat{\Phi}_{\Delta t} }{  1 \pm \Phi_N } + O (s^2) , 
\label{consistency-6.1}\\
&&
\hat{\Phi}_1^n \hat{\boldsymbol{U}}_1^{n+1}  
   = \Phi_N^n \boldsymbol{U}_N^{n+1}  
   + s ( \Phi_{\Delta t}^n \boldsymbol{U}_N^{n+1}  
     + \Phi_N^n \boldsymbol{U}_{\Delta t}^{n+1} ) + O (s^2)  ,
\label{consistency-6.2}\\
&&
\hat{\Phi}_1^n \nabla \widehat{\mathcal{V}}_1^{n+1}   
   = \Phi_N^n \nabla \mathcal{V}_N^{n+1}  
   + s ( \Phi_{\Delta t}^n \mathcal{V}_N^{n+1}     
     + \Phi_N^n \mathcal{V}_{\Delta t}^{n+1}  ) + O (s^2)  . \label{consistency-6.3}
\end{eqnarray}

In terms of the spatial discretization, the velocity profile $\hat{\boldsymbol{U}}_1$ is not divergence-free at a discrete level, so that its discrete inner product with the pressure gradient may not vanish. To overcome the difficulty, we propose a spatial interpolation operator $\mathcal{PH}_h$ defined as follow, for any $\boldsymbol{u}\in H^1(\Omega)$, $\nabla\cdot \boldsymbol{u}=0$:

There is a exact stream function vector $\boldsymbol{\psi}=(\psi_1,\ \psi_2,\ \psi_3)^T$ so that $\boldsymbol{u}=\nabla^{\perp} \boldsymbol{\psi}$,
\begin{equation}
\mathcal{PH}_h(\boldsymbol{u})=\nabla_h^{\perp} \boldsymbol{\psi}=(D_y \psi_3-D_z \psi_2, D_z \psi_1-D_x \psi_3, D_x \psi_2-D_y \psi_1)^T . 
\end{equation}
This definition guarantees $\nabla_h \cdot \mathcal{PH}_h(\boldsymbol{u})=0$ at a point-wise level. Consequently, we obtain the definition of \eqref{perturbation term} and the higher order truncation error for $(\hat{\Phi},\ \hat{\boldsymbol{U}},\ \hat{P})$:
\begin{eqnarray}
&&
\frac{\hat{\Phi}^{n+1}-\hat{\Phi}^n}{s}=-\nabla_h\cdot(A_h\hat{\Phi}^n\hat{\boldsymbol{U}}^{n+1})+\Delta_h \hat{\mathcal{V}}^{n+1}+\tau_{\phi}^{n+1} , 
\label{consistency-7.1}\\
&&
\hat{\mathcal{V}}^{n+1}=\ln(1+\hat{\Phi}^{n+1})-\ln(1-\hat{\Phi}^{n+1})-\theta_0\hat{\Phi}^n-\varepsilon^2\Delta_h\hat{\Phi}^{n+1},
\label{consistency-7.2}\\
&&
(-\Delta_h+I)\hat{\boldsymbol{U}}^{n+1}=-\nabla_h \hat{P}^{n+1}-\gamma(A_h\phi^n\nabla_h\hat{\mathcal{V}}^{n+1})+\tau_v^{n+1},
\label{consistency-7.3}\\
&&
\nabla_h\cdot\hat{\boldsymbol{U}}^{n+1}=0 , \label{consistency-7.4}
\end{eqnarray}
where
\begin{equation}
\|\tau_\phi^{n+1}\|_2,\ \|\tau_v^{n+1}\|_2 \leq O(s^2+h^2).
\end{equation}

The reason for such a higher truncation error estimate is to derive an $\ell^\infty$ bound for the numerical solution, which is needed to obtain the separation property in the rough error estimate. With such a property for the constructed approximate solution and the numerical solution, the nonlinear error term could be appropriately analyzed in the $\ell^\infty (0,T; H_h^1)$ convergence estimate. 

\begin{rem}
Trivial initial data $\Phi_{\Delta t} (\, \cdot \, , t=0) \equiv 0$ are imposed, as in \eqref{consistency-4.1}-\eqref{consistency-4.3}. Therefore using similar process in \eqref{mass conservative_1}-\eqref{mass conservative_2}, we have
\begin{eqnarray}
&
\phi^0 \equiv  \hat{\Phi}^0 ,  \quad \overline{\phi^k} = \overline{\phi^0} ,  \quad \forall \, k \ge 0,
\label{consistency-8.1}
\\
 &
\overline{\hat{\Phi}^k} = \frac{1}{|\Omega|}\int_\Omega \, \hat{\Phi} ( \cdot, t_k) \, d {\bf x}  = \frac{1}{|\Omega|}\int_\Omega \, \hat{\Phi}^0 \, d {\bf x} = \overline{\phi^0} ,  \quad \forall \, k \ge 0,  
	\label{consistency-8.2}
\end{eqnarray}
where the first step is based on the fact that $\hat{\Phi} \in {\cal B}^K$, and the second step comes from the mass conservative property of $\hat{\Phi}$ at the continuous level. These two properties will be used in the later analysis. 

In addition, since $\hat{\Phi}$ is mass conservative at a discrete level, we observe that the local truncation error $\tau_\phi$ has a similar property: 
\begin{equation} 
  \overline{\tau_\phi^{n+1}} = 0 ,  \quad \forall \, n \ge 0 . \label{consistency-8.3}
\end{equation}  
\end{rem}
\begin{rem}
Since the correction function $\Phi_{\Delta t}$ is bounded, we recall the separation property \eqref{separation property} for the exact solution, and obtain a similar property for $\hat{\Phi}$ if $s$ and $h$ sufficiently small: 
\begin{equation}
1+\hat{\Phi} \geq \epsilon_0^*:=\frac{\epsilon_0}{2}, \quad 1-\hat{\Phi} \geq \epsilon_0^*.\label{separation property hat}
\end{equation}
Such a uniform bound will be used in the convergence analysis. 

In addition, since the correction function is only based on the projection solution $(\Phi_N, \boldsymbol{U}_N, P_N)$ with enough regularity, its discrete $W_h^{1,\infty}$ norm will stay bounded:
\begin{equation} 
  \| \hat{\Phi}^k \|_\infty \leq C^\star , \, \, \, 
  \| \hat{\boldsymbol{U}}^k \|_\infty \leq C^\star   , \, \, \, 
  \| \nabla_h \hat{\Phi}^k \|_\infty \leq C^\star , \, \, \, 
  \| \nabla_h \hat{\boldsymbol{U}}^k \|_\infty \leq C^\star  ,  \quad \forall \, k \ge 0 .  
    \label{assumption:W1-infty bound}
\end{equation}  
\end{rem}

\subsection{Rough error estimate}
Instead of a direct analysis for the error function \eqref{error function-1}, we introduce the perturbed numerical error function with second order truncation error:
\begin{equation}
\tilde{\phi}^n := \mathcal{P}_h \hat{\Phi}^n - \phi^n ,  \, \, \, 
\tilde{\boldsymbol{u}}^n := \mathcal{P}_h \hat{\boldsymbol{U}}^n - \boldsymbol{u}^n , \, \, \, 
\tilde{p}^n :=   \mathcal{P}_h \hat{P}^n - p^n , 
\quad \forall \ m \in \mathbb{N} .
\label{error function-2}
\end{equation}
In turn, subtracting the numerical scheme \eqref{discrete-CHS-0.1} -- \eqref{discrete-CHS-0.4} from \eqref{consistency-7.1} -- \eqref{consistency-7.4} gives 
\begin{eqnarray}
\frac{\tilde{\phi}^{n+1}-\tilde{\phi}^n}{s}&=&\Delta_h \tilde{\mu}^{n+1}-\nabla_h\cdot(A_h\tilde{\phi}^n \hat{\boldsymbol{U}}^{n+1}+A_h\phi^n\tilde{\boldsymbol{u}}^{n+1})+\tau_{\phi}^{n+1},\label{error equation-1.1} \\
\tilde{\mu}^{n+1}&=&\ln(1+\hat{\phi}^{n+1})-\ln(1+\phi^{n+1})-\ln(1-\hat{\phi}^{n+1})+\ln(1-\phi^{n+1}) \nonumber \\
&&-\theta_0 \tilde{\phi}^n-\varepsilon^2\Delta_h\tilde{\phi}^{n+1},\label{error equation-1.2} \\
(-\Delta_h+I)\tilde{\boldsymbol{u}}^{n+1}&=&-\nabla_h \tilde{p}^{n+1}-\gamma(A_h\phi^n\nabla_h\tilde{\mu}^{n+1}+A_h \tilde{\phi}^n \nabla_h \hat{\mathcal{V}}^{n+1})+\tau_v^{n+1},\label{error equation-1.3} \\
\nabla_h\cdot\tilde{\boldsymbol{u}}^{n+1}&=&0.\label{error equation-1.4}
\end{eqnarray}
where
$$
\|\tau_{\phi}^{n+1}\|_2,\, \|\tau_v^{n+1}\|_2 \leq O(s^2+h^2).
$$

Since $\hat{\mathcal{V}}^{n+1}$ only depends on the exact solution and correction function, we assume a discrete $W_h^{1,\infty}$ bound
\begin{equation}
\|\hat{\mathcal{V}}^{n+1}\|_{W^{1,\infty}_h} \leq C^*.  \label{assumption:W1-infty bound mu}
\end{equation}
In addition, we make the following a-priori assumption for the previous time step 
\begin{equation}
  \| \tilde{\phi}^n \|_2 + \| \nabla_h \tilde{\phi}^n \|_2 \leq s^\frac{15}{8} + h^\frac{15}{8}  . 
   \label{a priori-1} 
\end{equation}
Such an a-priori assumption will be recovered by the convergence analysis in the next time step,  which will be demonstrated later. In turn, this a-priori assumption leads to an $\ell^\infty$ bound, based on the inverse inequality and the linear refinement requirement $C_1h \leq s \leq C_2h$: 
\begin{equation}
\|\tilde{\phi}^n\|_\infty \leq \frac{C\|\tilde{\phi}^n\|_{H^1_h}}{h^\frac{1}{2}} \leq C(s^\frac{11}{8} + h^\frac{11}{8}) \leq 1.\label{a priori-2} 
\end{equation}
The following lemma states the rough error estimate; the detailed proof will be provided in Appendix \ref{appendix A}.

\begin{lemma} \label{rough convergence lemma}
We make the regularity assumption of $\hat{\mathcal{V}}^{n+1}$ \eqref{assumption:W1-infty bound mu}, as well as the a-priori assumption \eqref{a priori-1}. For the numerical error evolutionary system \eqref{error equation-1.1} -- \eqref{error equation-1.4}, a rough error estimate is valid: 
\begin{equation}
  \| \tilde{\phi}^{n+1} \|_2 + \| \nabla_h \tilde{\phi}^{n+1} \|_2 \leq C(s^\frac{5}{4} + h^\frac{5}{4}).\label{rough convergence}  
\end{equation}
\end{lemma}

As a direct consequence of the rough error estimate \eqref{rough convergence}, an application of 3-D inverse inequality, combined with a discrete Sobolev inequality (given by~\eqref{Sobolev-1} in Lemma~\ref{lem: Sobolev-1}), reveals that 
\begin{eqnarray}
&&
\| \tilde{\phi}^{n+1} \|_\infty  \leq \frac{C \| \tilde{\phi}^{n+1} \|_{H_h^1} }{h^\frac{1}{2}}
\le \hat {C}_1 ( s^\frac{3}{4} + h^\frac{3}{4} )  \le \frac{\epsilon_0^\star}{2} , 
   \label{convergence-rough-1.1}  
\\
&&
\| \tilde{\phi}^{n+1} \|_4  \leq C\| \tilde{\phi}^{n+1} \|_{H^1_h}  \leq C(s^\frac{5}{4} + h^\frac{5}{4}).
\label{convergence-rough-1.2}  
\end{eqnarray}
Furthermore, a combination of \eqref{convergence-rough-1.1} and separation property \eqref{separation property hat} leads to a separation property of the numerical solution at the next time step $t^{n+1}$
\begin{equation}
\frac{\epsilon_0^*}{2} \leq 1+\phi^{n+1} \leq 2, \quad \frac{\epsilon_0^*}{2} \leq 1-\phi^{n+1} \leq 2.\label{separation property numerical}
\end{equation}
Such a uniform bound will play a very important role in the refined error estimate.

\begin{rem}
It is noticed that the accuracy order in \eqref{rough convergence} is at least half order lower than the a-priori estimate \eqref{a priori-1}, as well as the lower rate of the $\ell^\infty$ error in \eqref{convergence-rough-1.1} , which comes from an application of the inverse inequality. In particular, the first order temporal truncation error is not sufficient to ensure the phase separation property; this is the reason why a complex process to construct higher order truncation error is needed. On the other hand, the a-priori assumption could not be covered by the lower accuracy rate in \eqref{rough convergence}.  Instead, such a separation property \eqref{separation property numerical} will lead to a much sharper refined estimate.
\end{rem}

\subsection{Refined error estimate}
Before proceeding into the refined error estimate, the following preliminary result for the nonlinear error term is needed. For simplicity of presentation, the detailed proof will be provided in Appendix \ref{appendix B}. 

\begin{lemma}\label{nonlinear error term}
Define
\begin{equation}
\mathcal{L}^{n+1}=\ln (1+\hat{\Phi}^{n+1})-\ln(1+\phi^{n+1})-\ln(1-\hat{\Phi}^{n+1})+\ln(1-\phi^{n+1})-\theta_{0} \tilde{\phi}^n . \label{nonlinear error-1}
\end{equation}
Based on the separation property \eqref{separation property numerical} for numerical solution, we have
\begin{equation}
\|\nabla_h\mathcal{L}^{n+1}\|_2 \leq 4(\epsilon_0^*)^{-1}\|\nabla_h\tilde{\phi}^{n+1}\|_2+C(\epsilon_0^*)^{-2}\|\tilde{\phi}^{n+1}\|_4+\theta_0\|\nabla_h\tilde{\phi}^n\|_2.
\label{nonlinear error-2}
\end{equation}
\end{lemma}

Now we carry out the refined error estimate. Taking a discrete inner product with \eqref{error equation-1.1} by $-2\Delta_h \tilde{\phi}^{n+1}$ leads to
\begin{eqnarray} 
  &&
  \frac{1}{s} \left( \| \nabla_h \tilde{\phi}^{n+1} \|_2^2 - \|\nabla_h \tilde{\phi}^n \|_2^2 
  +   \| \nabla_h ( \tilde{\phi}^{n+1}  - \tilde{\phi}^n ) \|_2^2  \right)      
  + 2 (  \tilde{\boldsymbol{u}}^{n+1} , A_h \phi^n  \nabla_h \Delta_h \tilde{\phi}^{n+1} )   
  \nonumber 
\\ 
  &=&
   2(\nabla_h \tilde{\mu}^{n+1}, \nabla_h \Delta_h \tilde{\phi}^{n+1})
  - 2  ( \tau_\phi^{n+1} , \Delta_h \tilde{\phi}^{n+1} )
  - 2 ( A_h \tilde{\phi}^n \hat{\boldsymbol{U}}^{n+1}, \nabla_h \Delta_h \tilde{\phi}^{n+1} ) ,  
  \label{convergence-1} 
\end{eqnarray}
where summation-by-parts formulas have been recalled. The Cauchy inequality could be applied to the local truncation error term:
\begin{equation}
-2( \tau_\phi^{n+1}, \Delta_h \tilde{\phi}^{n+1}) \leq 2\|\tau_\phi^{n+1}\|_{-1,h} \cdot \|\nabla_h \Delta_h \tilde{\phi}^{n+1}\|_2 \leq 4\varepsilon^{-2}\|\tau_\phi^{n+1}\|_{-1,h}^2+\frac{\varepsilon^2}{4}\|\nabla_h \Delta_h \tilde{\phi}^{n+1}\|_2^2. \label{convergence-2} 
\end{equation}
The third term on the right-hand-side could be bounded in a similar way
\begin{equation}
\begin{aligned}
-2( A_h\tilde{\phi}^n \hat{\boldsymbol{U}}^{n+1}, \nabla_h \Delta_h \tilde{\phi}^{n+1} )&\leq 2\|\hat{\boldsymbol{U}}^{n+1}\|_\infty \cdot \|\tilde{\phi}^n\|_2 \cdot \|\nabla_h \Delta_h \tilde{\phi}^{n+1}\|_2 \\
&\leq 2C^*\|\tilde{\phi}^n\|_2 \cdot \|\nabla_h \Delta_h \tilde{\phi}^{n+1}\|_2 \\
&\leq 4(C^*)^2\varepsilon^{-2}\|\tilde{\phi}^n\|_2^2 + \frac{\varepsilon^2}{4}\|\nabla_h \Delta_h \tilde{\phi}^{n+1}\|_2^2.
\end{aligned}\label{convergence-3}  
\end{equation}
For the chemical potential diffusion term, the standard Cauchy inequality indicates that
\begin{equation}
\begin{aligned}
2(\nabla_h \tilde{\mu}^{n+1}, \nabla_h \Delta_h \tilde{\phi}^{n+1})=&2(\nabla_h\mathcal{L}^{n+1}, \nabla_h \Delta_h \tilde{\phi}^{n+1})-2\varepsilon^2\|\nabla_h \Delta_h \tilde{\phi}^{n+1}\|_2^2 \\
\leq& 4\varepsilon^{-2}\|\nabla_h\mathcal{L}^{n+1}\|_2^2-\frac{7\varepsilon^2}{4}\|\nabla_h \Delta_h \tilde{\phi}^{n+1}\|_2^2 . 
\end{aligned}\label{convergence-4}  
\end{equation}
A similar estimate could be performed for the nonlinear convection error term
\begin{align} 
   & 
    \nabla_h \Delta_h \tilde{\phi}^{n+1} 
    = \varepsilon^{-2} ( \nabla_h {\cal L}^{n+1}  - \nabla_h \tilde{\mu}^{n+1} )  , 
    \label{convergence-5-1}       
\\
  & 
   (  \tilde{\boldsymbol{u}}^{n+1} ,  
  A_h \phi^n   \nabla_h {\cal L}^{n+1}  )
  \ge   - \| \tilde{\boldsymbol{u}}^{n+1} \|_2  \cdot \| \phi^n  \|_\infty 
  \cdot \| \nabla_h {\cal L}^{n+1}  \|_2       
   \nonumber 
\\ 
  \ge &  -  \| \tilde{\boldsymbol{u}}^{n+1} \|_2 
  \cdot \| \nabla_h {\cal L}^{n+1}  \|_2   
  \ge -  \frac{1}{4\gamma} \| \tilde{\boldsymbol{u}}^{n+1} \|_2^2  
  - \gamma\| \nabla_h {\cal L}^{n+1}  \|_2^2.
   \label{convergence-5-2}      
\end{align}  
A combination with \eqref{proof lemma 6} gives 
\begin{align} 
  & 
   2 (  \tilde{\boldsymbol{u}}^{n+1} ,  
  A_h \phi^n  \nabla_h \Delta_h \tilde{\phi}^{n+1} )  
  = 2   \varepsilon^{-2} 
  \left(  (  \tilde{\boldsymbol{u}}^{n+1} ,  
  A_h \phi^n   \nabla_h {\cal L}^{n+1}  ) 
   -  (  \tilde{\boldsymbol{u}}^{n+1} ,  
  {\cal A}_h \phi^n   \nabla_h  \tilde{\mu}^{n+1}  ) \right)  \nonumber  
\\
  \ge & 
    2   \varepsilon^{-2}  \Big( \frac{1}{2\gamma} \| \tilde{\boldsymbol{u}}^{n+1} \|_2^2 
   + \frac{1}{\gamma} \|\nabla_h \tilde{\boldsymbol{u}}^{n+1} \|_2^2
  - \gamma \| \nabla_h {\cal L}^{n+1}  \|_2^2
  - \frac{2}{\gamma} \| \tau_v^{n+1} \|_2^2  
   - C\| \tilde{\phi}^n \|_2^2 \Big)  .   
   \label{convergence-5-3}      
\end{align}  
Substituting \eqref{convergence-2}-\eqref{convergence-4} and \eqref{convergence-5-3} into \eqref{convergence-1}, combined with an application of Lemma \ref{nonlinear error term}, we obtain 
\begin{equation}
\begin{aligned}
\frac{1}{s} &\left( \| \nabla_h \tilde{\phi}^{n+1} \|_2^2 - \|\nabla_h \tilde{\phi}^n \|_2^2 + \| \nabla_h ( \tilde{\phi}^{n+1} - \tilde{\phi}^n ) \|_2^2  \right)
+\frac{5\varepsilon^2}{4}\|\nabla_h \Delta_h \tilde{\phi}^{n+1}\|_2^2
+ \frac{\varepsilon^{-2}}{\gamma} (\| \tilde{\boldsymbol{u}}^{n+1} \|_2^2+2\| \nabla_h \tilde{\boldsymbol{u}}^{n+1} \|_2^2) \\
&\quad
\leq 4\varepsilon^{-2}\left(\|\tau_\phi^{n+1}\|_{-1,h}^2+\| \tau_v^{n+1} \|_2^2\right)
+2\varepsilon^{-2}\left(2(C^*)^2+C\right) \|\tilde{\phi}^n\|_2^2
+\varepsilon^{-2}(4+2\gamma)\|\nabla_h\mathcal{L}^{n+1}\|_2^2\\
&\quad
\leq 4\varepsilon^{-2}\left(\|\tau_\phi^{n+1}\|_{-1,h}^2+\| \tau_v^{n+1} \|_2^2\right)
+2\varepsilon^{-2}\left(2(C^*)^2+C\right) \|\tilde{\phi}^n\|_2^2\\
&\quad
\qquad+(12+6\gamma)\varepsilon^{-2}\left(   16(\epsilon_0^*)^{-2}\|\nabla_h\tilde{\phi}^{n+1}\|_2^2+C(\epsilon_0^*)^{-4}\|\tilde{\phi}^{n+1}\|_4^2+\theta_0^2\|\nabla_h\tilde{\phi}^n\|_2^2   \right)\\
&\quad
\leq 4\varepsilon^{-2}\left(\|\tau_\phi^{n+1}\|_{-1,h}^2+\| \tau_v^{n+1} \|_2^2\right)
+2C\varepsilon^{-2}\left(2(C^*)^2+C\right) \|\nabla_h\tilde{\phi}^n\|_2^2 \\
&\quad
\qquad +(12+6\gamma)\varepsilon^{-2}\left(  \left( 16(\epsilon_0^*)^{-2}+C(\epsilon_0^*)^{-4}\right)\|\nabla_h\tilde{\phi}^{n+1}\|_2^2+\theta_0^2\|\nabla_h\tilde{\phi}^n\|_2^2   \right) , 
\end{aligned}
\label{convergence-5-4}  
\end{equation}
where the 3-D discrete Sobolev inequality,  $\| \cdot \|_4 \leq C\| \cdot \|_{H_h^1}$ (given by~\eqref{Sobolev-1} in Lemma~\ref{lem: Sobolev-1}), and the discrete Poincar\'e inequality have been used in last step. Therefore, with sufficiently small $s$ and $h$, an application of discrete Gronwall inequality leads to the desired higher order convergence estimate 
\begin{equation}
  \| \nabla \tilde{\phi}^{n+1} \|_2 + \Bigl(  \varepsilon^2 s   \sum_{k=1}^{n+1} 
    \| \nabla_h \Delta_h \tilde{\phi}^k \|_2^2 
   \Bigr)^{1/2}  \le C ( s^2 + h^2 ) , 
	\label{convergence-7}
\end{equation} 
based on the higher order truncation error accuracy, $\| \tau_\phi^{n+1} \|_{-1, h}$, $\| \tau_v^{n+1} \|_2 \leq C (s^2 + h^2)$. This completes the refined error estimate. 

With the higher order convergence estimate \eqref{convergence-7} in hand, the a-priori assumption in~\eqref{a priori-1} is recovered at the next time step $t^{n+1}$:  
\begin{equation} 
\| \tilde{\phi}^{n+1} \|_2,  \| \nabla_h \tilde{\phi}^{n+1} \|_2 \le C (s^2 + h^2 ) \le s^\frac{15}{8} + h^\frac{15}{8},  
	\label{a priori-8}  
\end{equation} 
provided that $s$ and $h$ are sufficiently small, in which a discrete Poincar\'e inequality has been used again. Therefore, an induction analysis could be applied. This finishes the higher order convergence analysis.

As a result, the error estimate \eqref{convergernce theorem-1} for variable $\phi$ is a direct consequence of \eqref{a priori-8}, combined with the boundedness of supplementary fields $\Phi_{\Delta t}$, as well as the projection approximation \eqref{Fourier-approximation}. This completes the proof of Theorem \ref{convergernce theorem}. 

\section{Numerical experiments}

In this section, we present a few numerical results, including a convergence test and some sample computations in a 2-D domain. A  full approximation storage (FAS) nonlinear multigrid method is used to solve the nonlinear equations in the numerical scheme \eqref{discrete-CHS-0.1} -- \eqref{discrete-CHS-0.4}. See~\cite{collins13} for details about a similar solver. The first example demonstrates the robustness of the multigrid solver. The phase decomposition phenomenon, as well as the energy stability and mass conservation property of the proposed numerical scheme, will be verified in details. In another experiment we test the convergence order of the numerical scheme~\eqref{discrete-CHS-0.1} -- \eqref{discrete-CHS-0.4}.  The computational domain is taken as $\Omega=(0,1)^2$, and the physical parameters are set as: $\theta_0=3, \gamma = 1.0$. See~\cite{chen19b} for comparison.

\subsection{Spinodal decomposition, energy decay and mass conservation}

In this subsection, we choose random initial data to display the phase decomposition phenomenon, energy decay and mass conservation.  We set $\varepsilon=0.01$, $h=\frac{1}{128}$, $s=2*10^{-5}$ and initial data as
	\begin{equation}
\phi_{i,j}=0.2+0.02*r_{i,j}, 
	\label{rand initial}
	\end{equation}
where $r_{i,j}$ is a random field of values that are uniformly distributed in $[-1,1]$. Figure \ref{Spinodal decomposition} describes evolution of the phase variable at some selected time levels with the initial condition \eqref{rand initial}. 

\begin{figure}[h]
\center
\subfigure[t=0]{
\includegraphics[height=3cm,width=4cm]{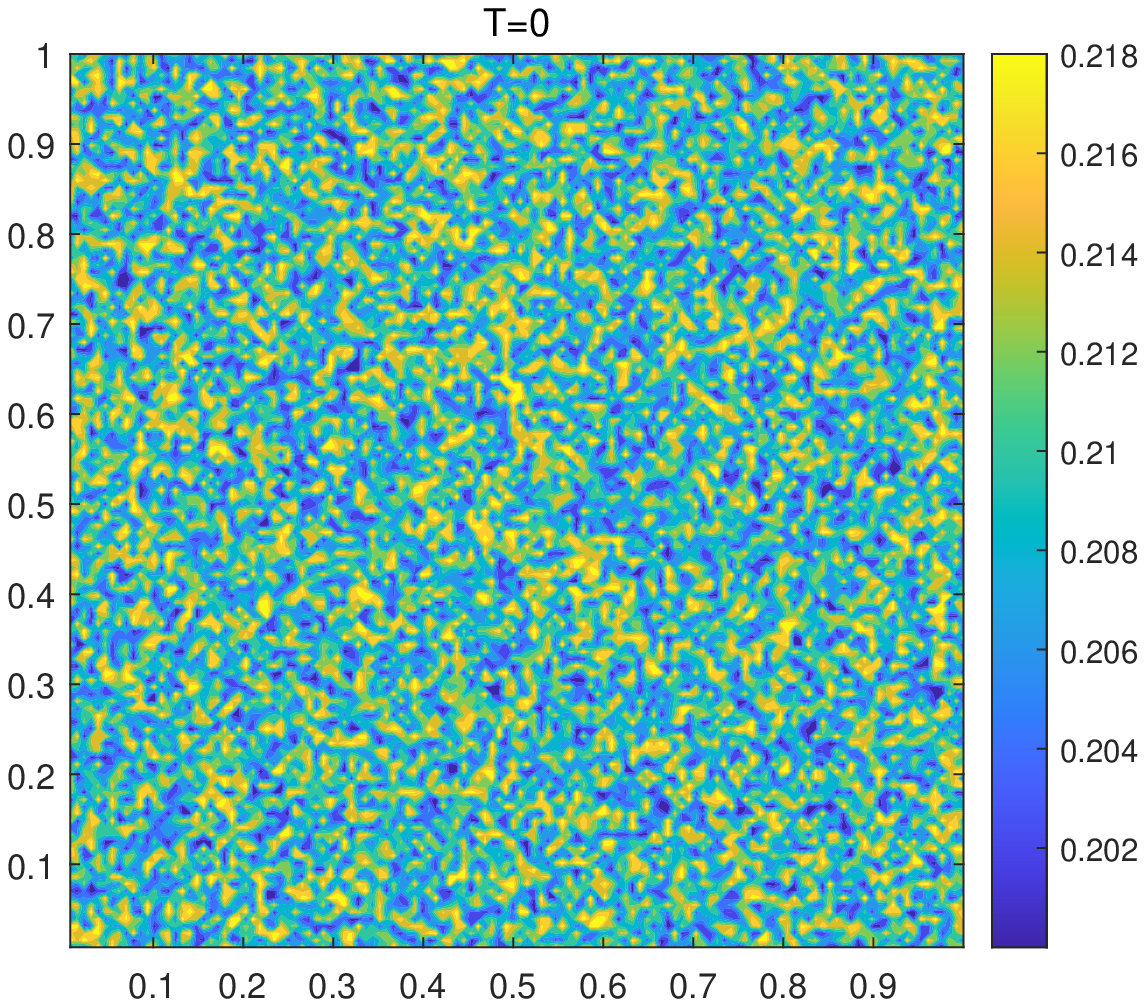}   
}
\subfigure[t=0.004]{
\includegraphics[height=3cm,width=4cm]{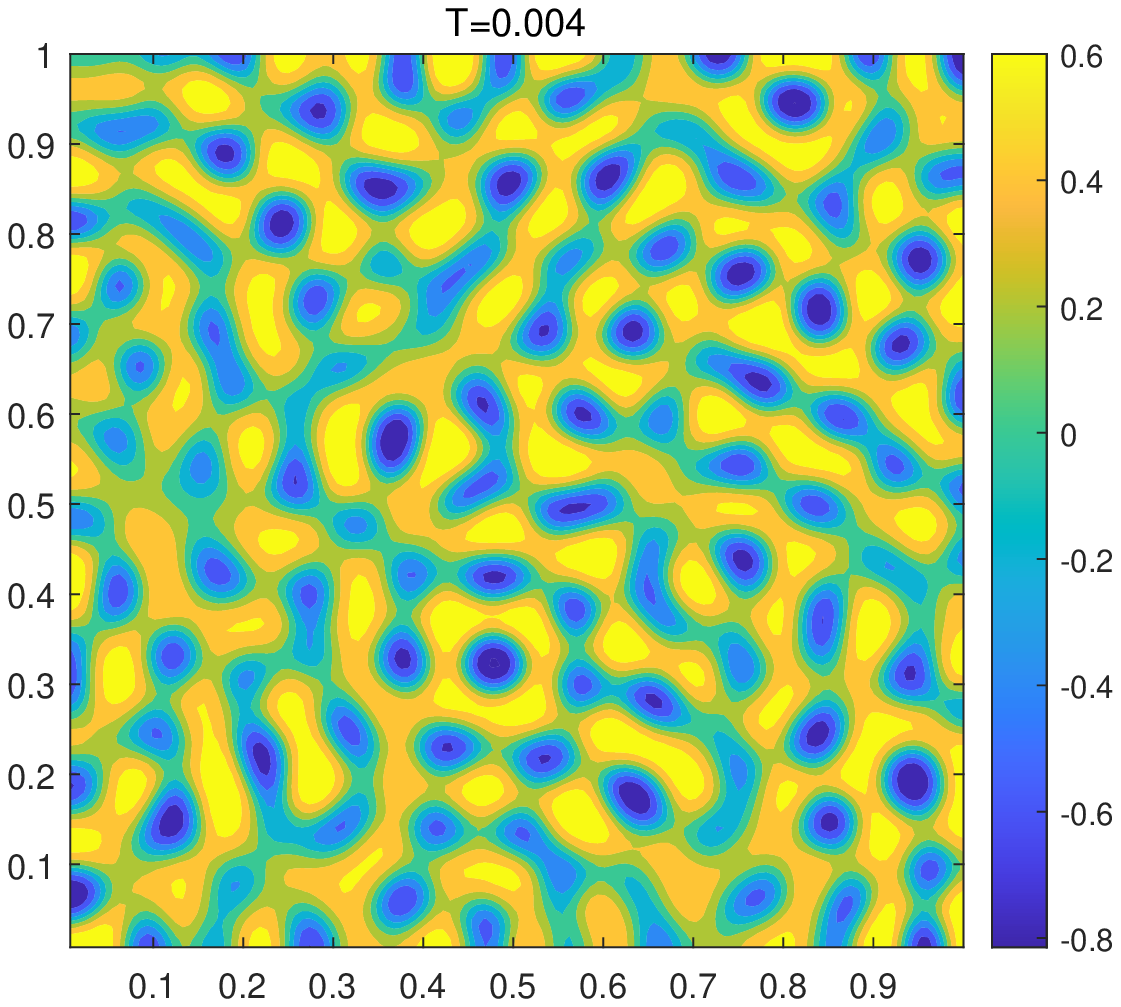}   
}
\subfigure[t=0.01]{
\includegraphics[height=3cm,width=4cm]{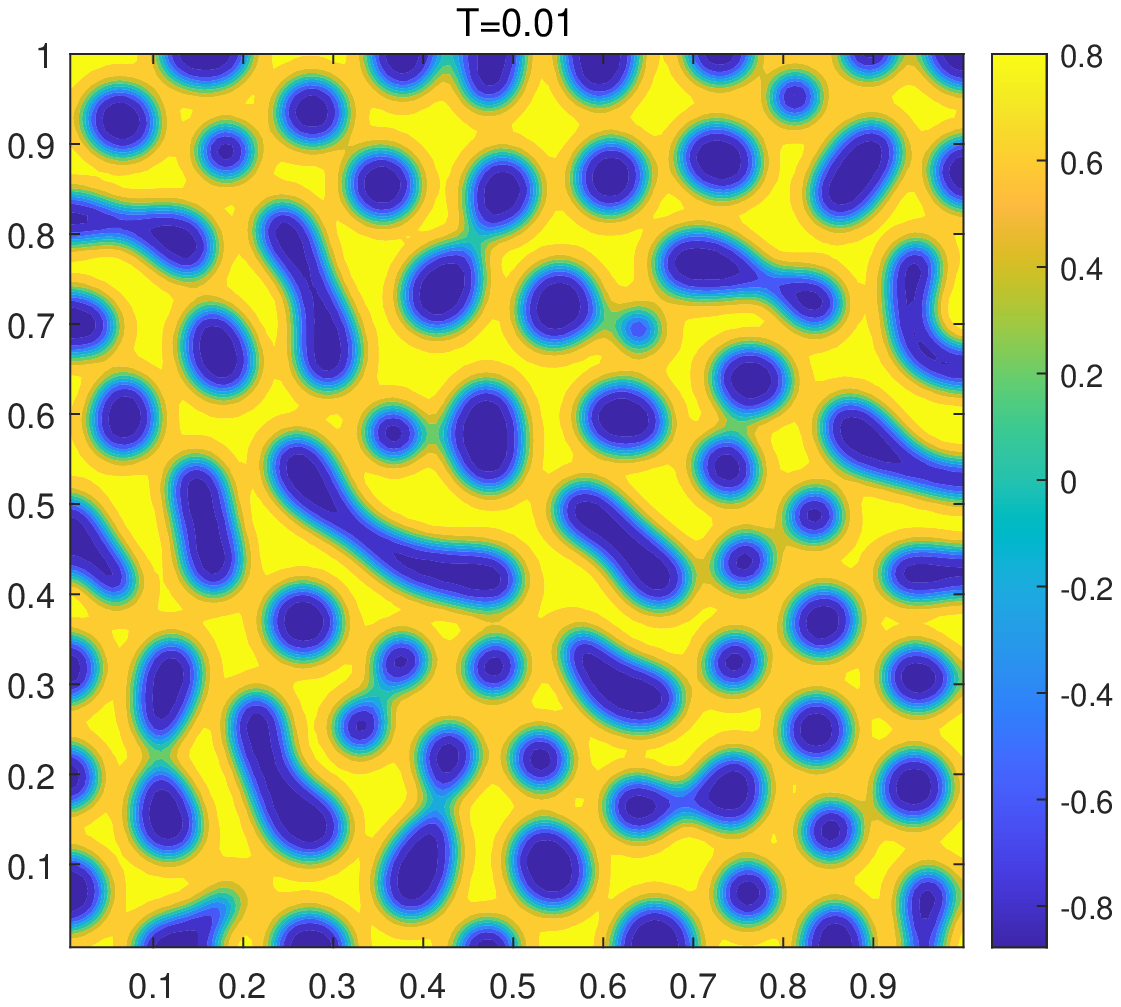}   
}

\subfigure[t=0.1]{
\includegraphics[height=3cm,width=4cm]{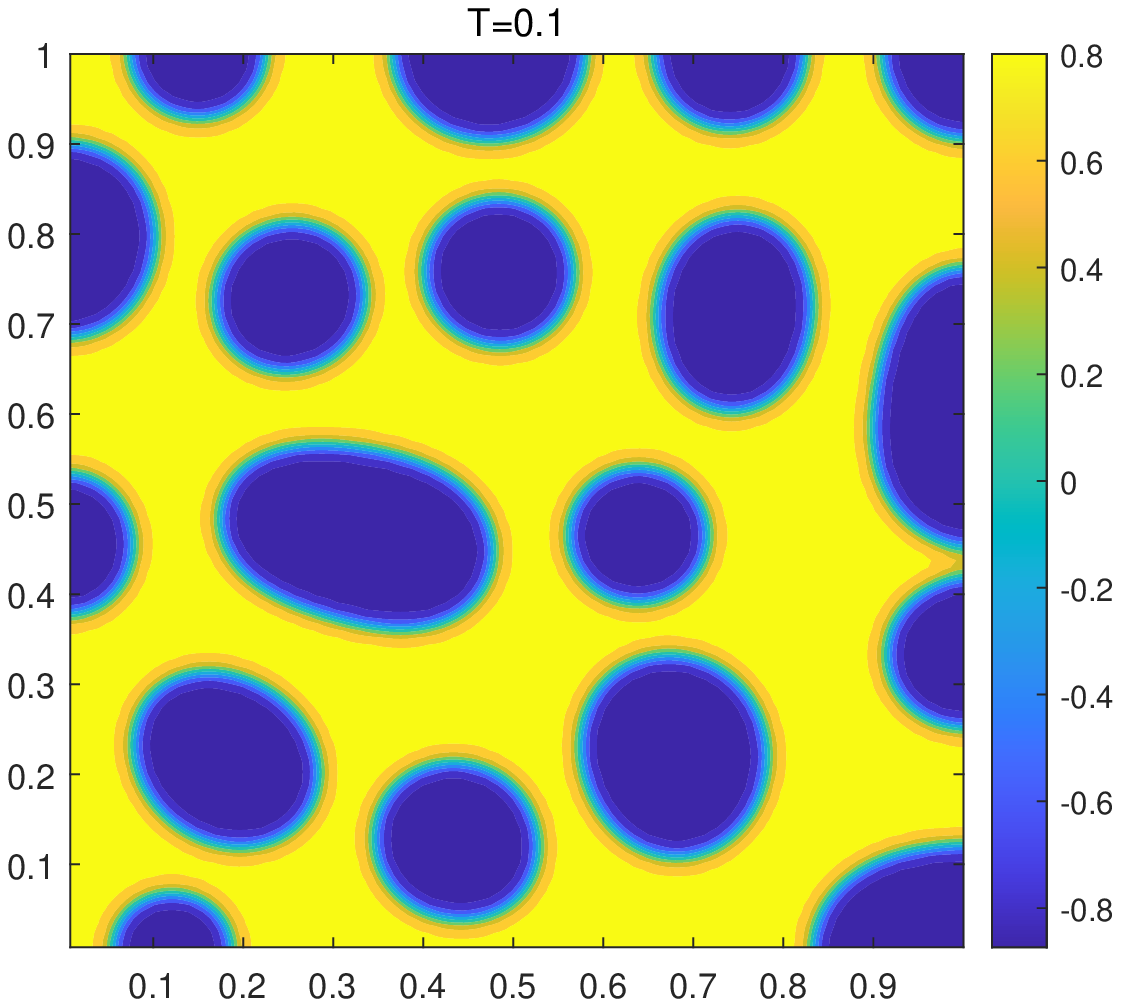}   
}
\subfigure[t=0.3]{
\includegraphics[height=3cm,width=4cm]{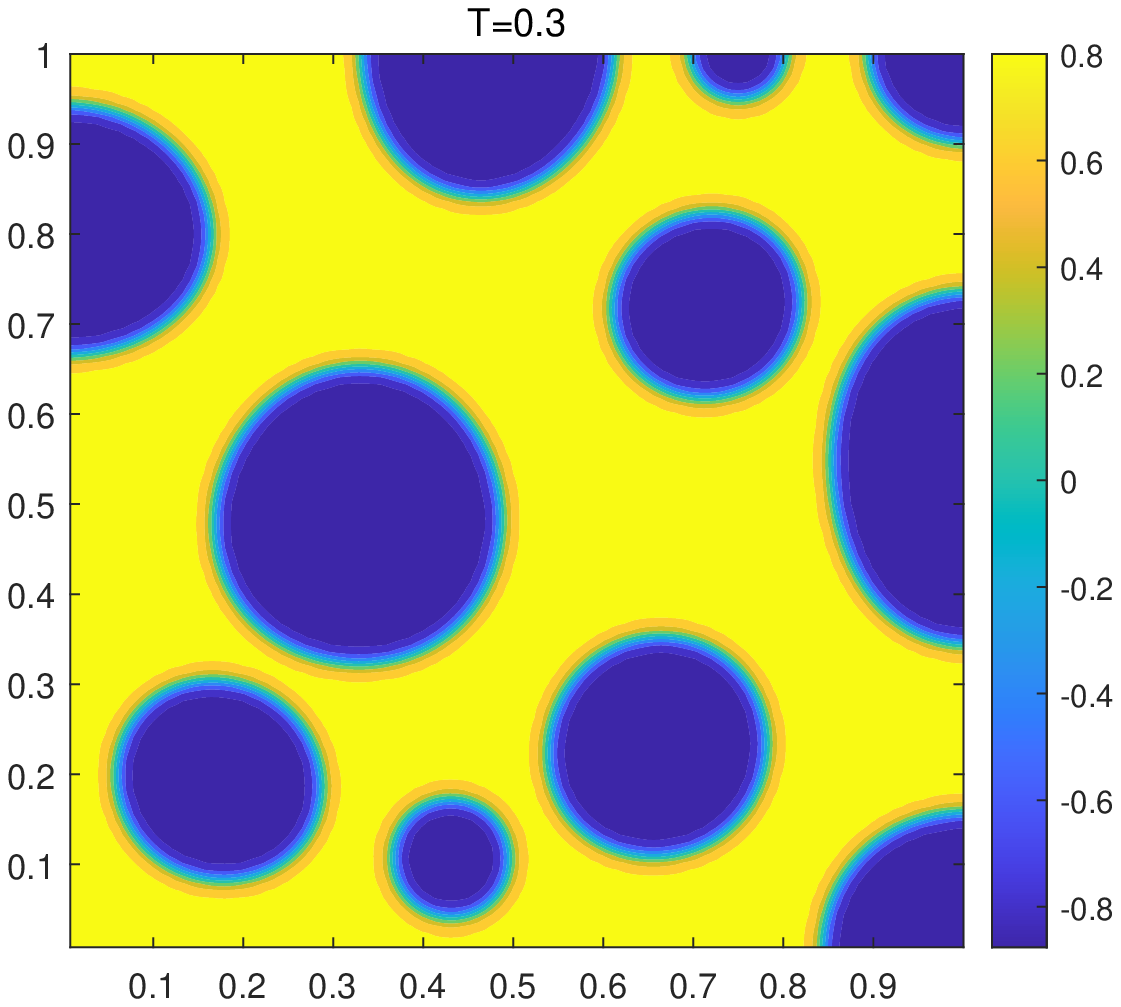}   
}
\subfigure[t=0.5]{
\includegraphics[height=3cm,width=4cm]{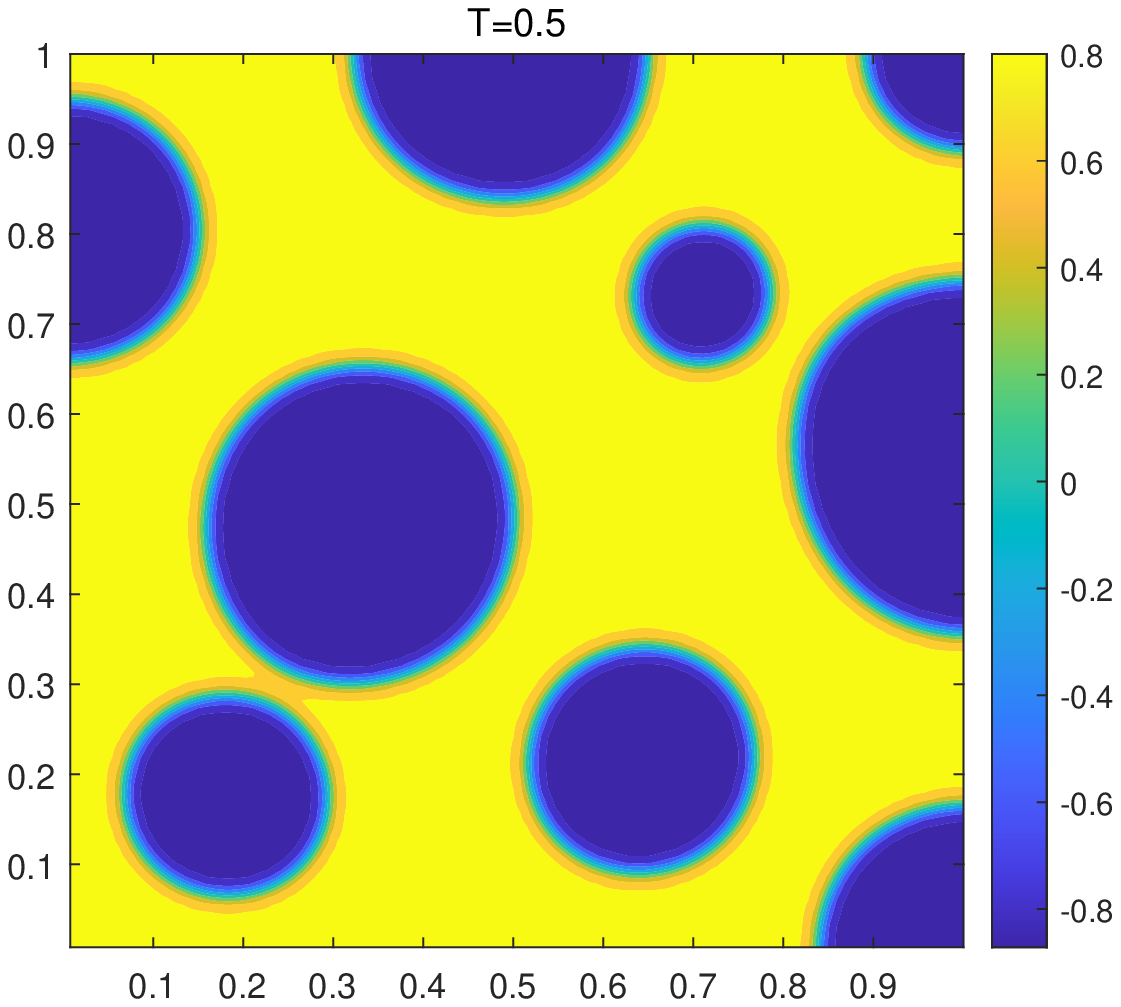}   
}

\subfigure[t=1]{
\includegraphics[height=3cm,width=4cm]{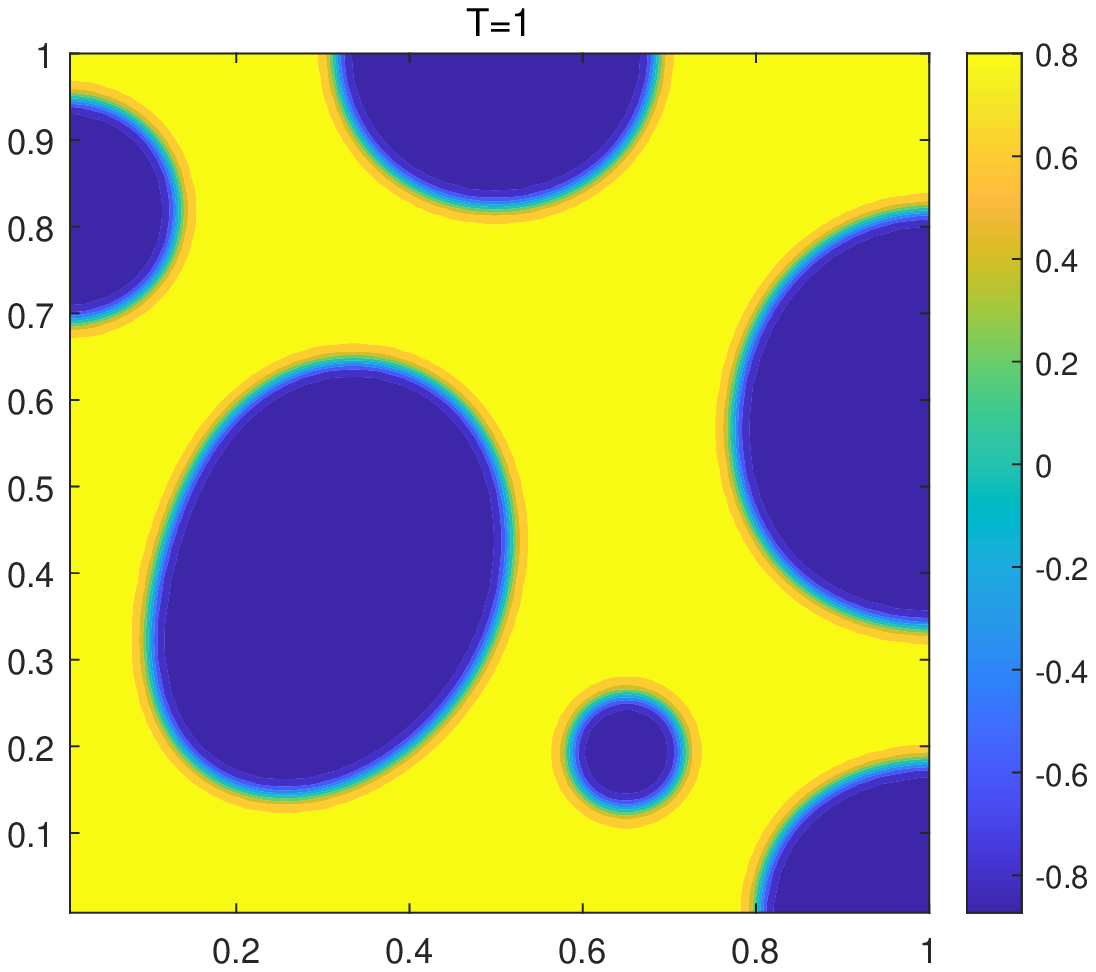}   
}
\subfigure[t=5]{
\includegraphics[height=3cm,width=4cm]{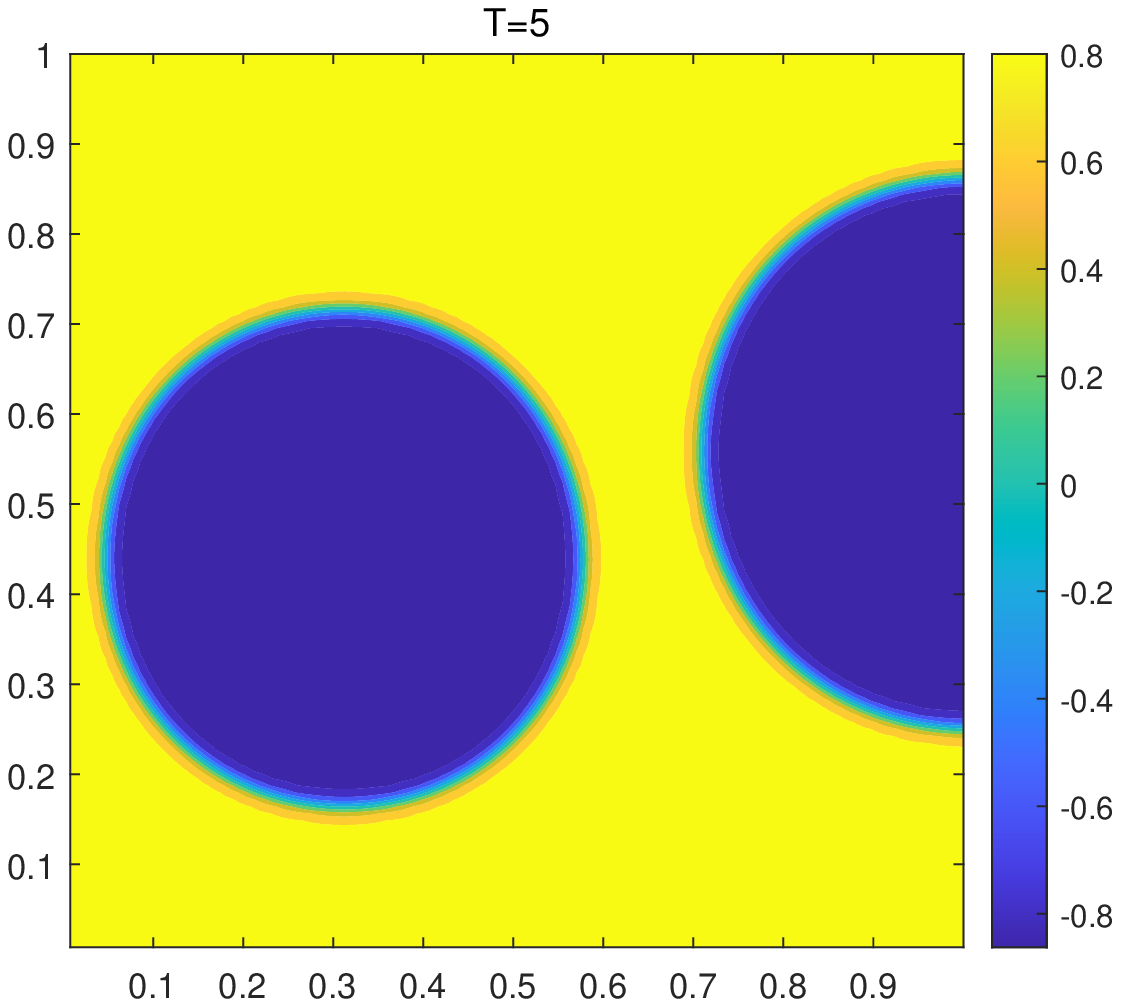}   
}
\subfigure[t=16]{
\includegraphics[height=3cm,width=4cm]{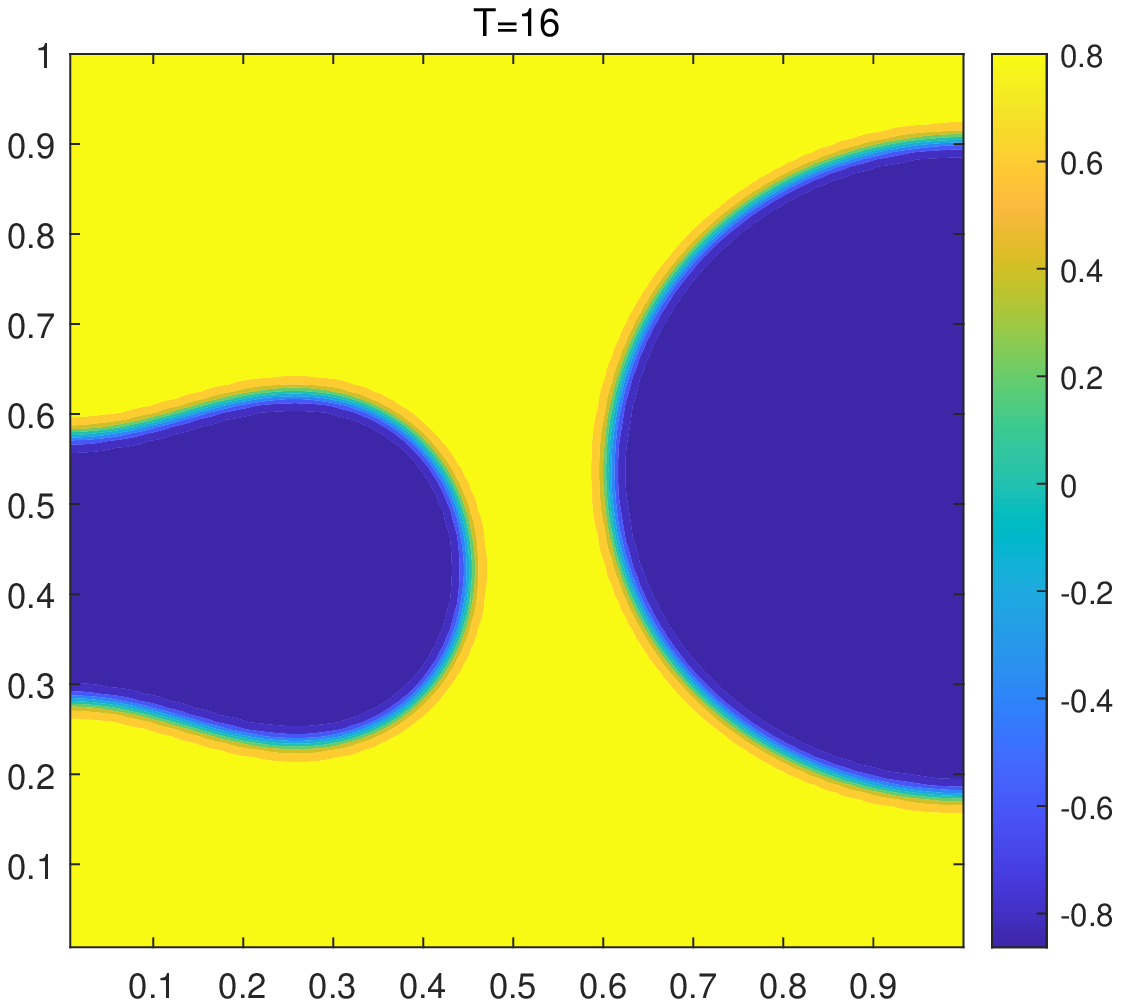}   
}

\subfigure[t=18]{
\includegraphics[height=3cm,width=4cm]{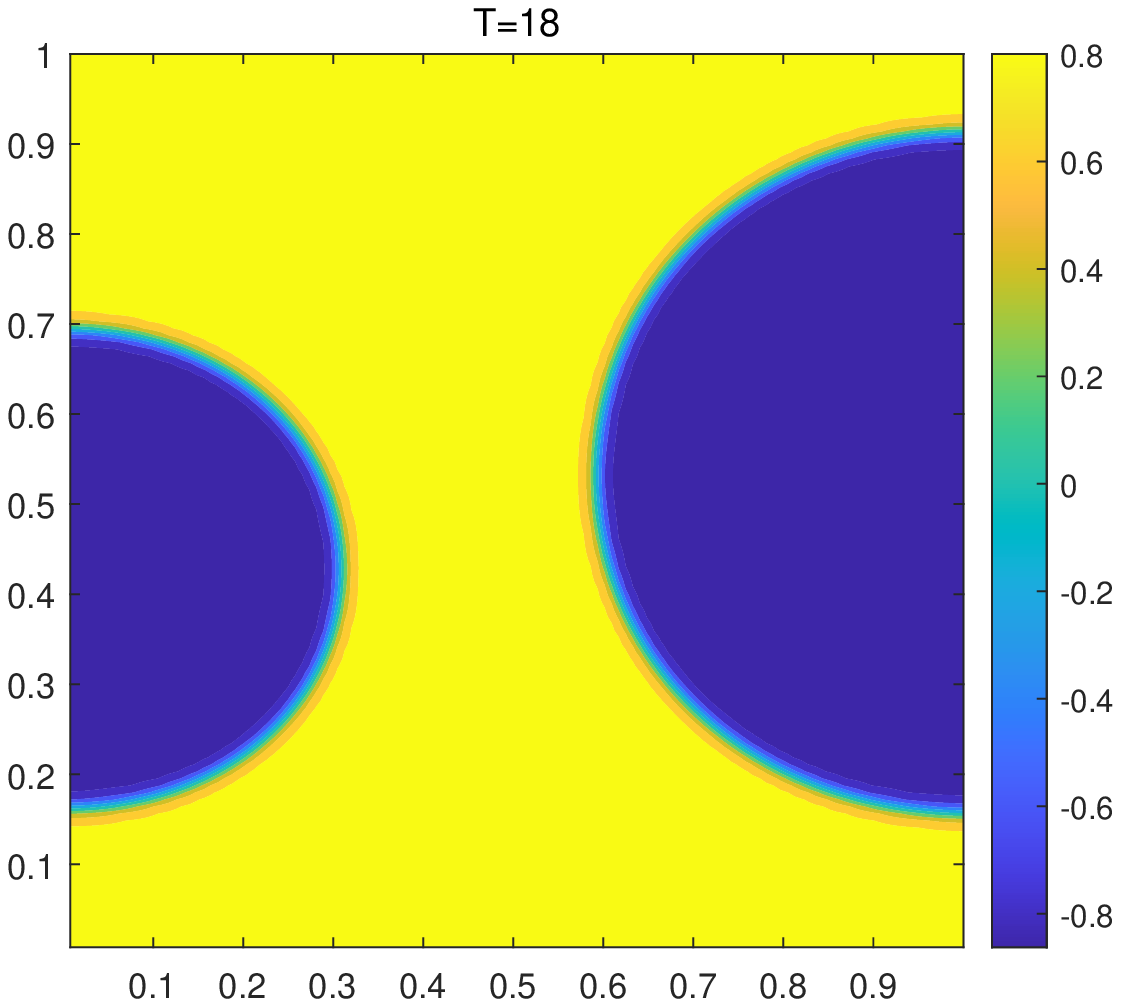}   
}
\subfigure[t=28]{
\includegraphics[height=3cm,width=4cm]{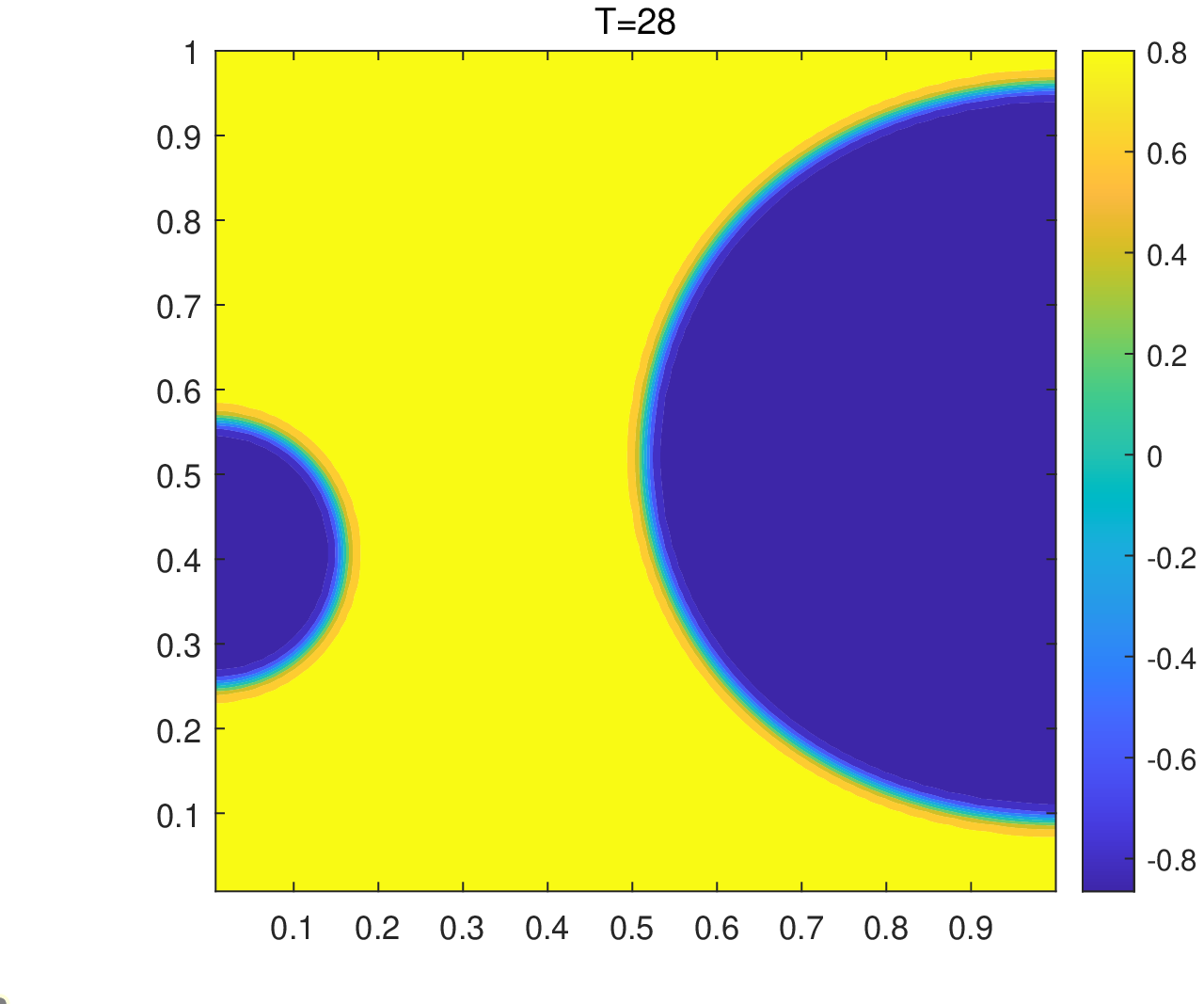}   
}
\subfigure[t=30]{
\includegraphics[height=3cm,width=4cm]{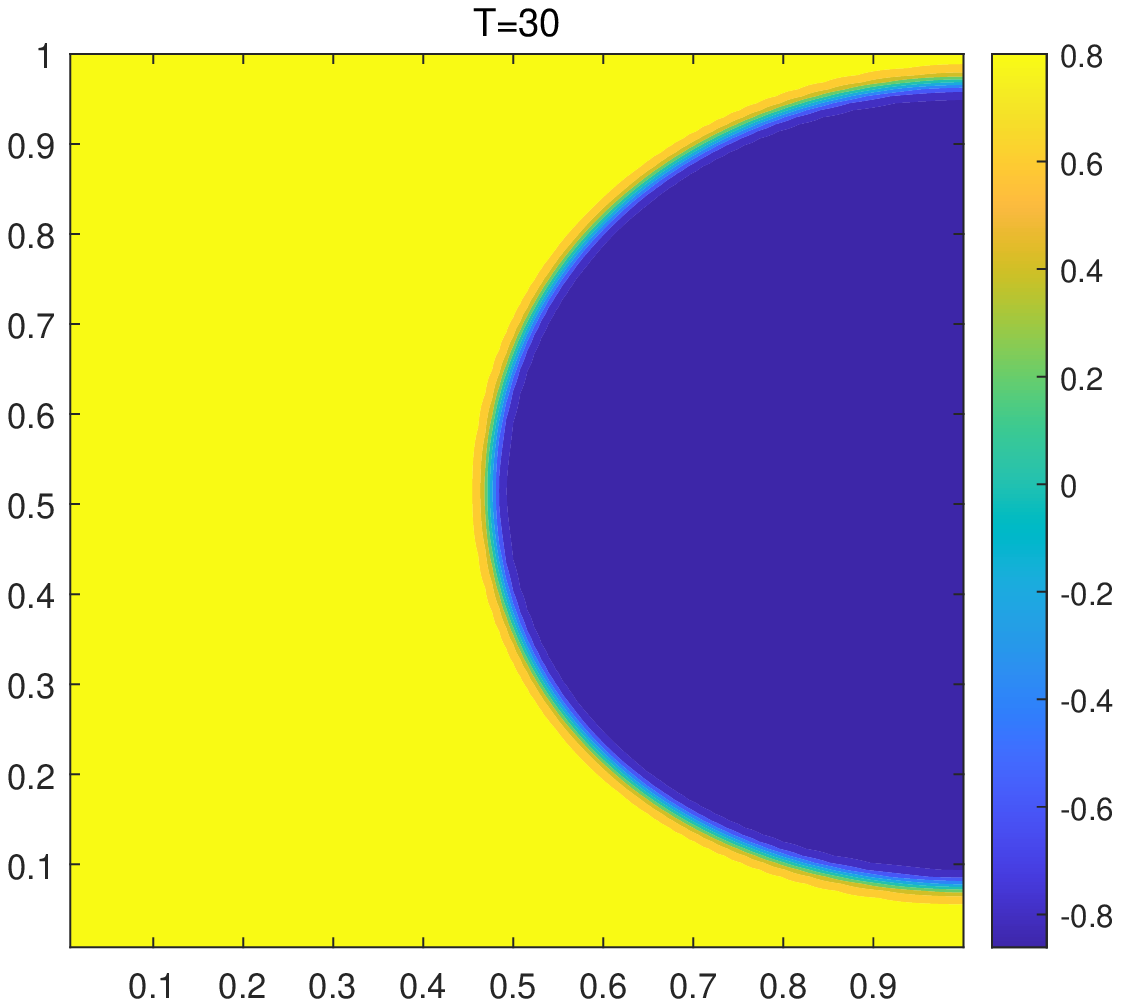}   
}
\caption{Evolution of the phase variable at selected times, with initial condition \eqref{rand initial}. Yellow corresponds to~$\phi\approx 0.8$~and blue corresponds to~$\phi\approx -0.8$~.}
	\label{Spinodal decomposition}
	\end{figure}

For the system with polynomial energy functional~\eqref{polynomial energy-1}, relevant numerical experiments have shown that concentration variable~$\phi$ can overshoot the values $\pm 1$\cite{wise10, collins13, feng12}. Meanwhile, a strict separation property is observed in Figure \ref{Spinodal decomposition}, so that a uniform distance exists between the phase variable extrema and the singular limit values $\pm 1$. This numerical result gives a clear evidence that the singular logarithmic energy potential model leads to a much more powerful phase separation property than the polynomial approximation one. 

\begin{figure}[h]
\center
\subfigure[Energy]{
\includegraphics[height=4cm,width=6cm]{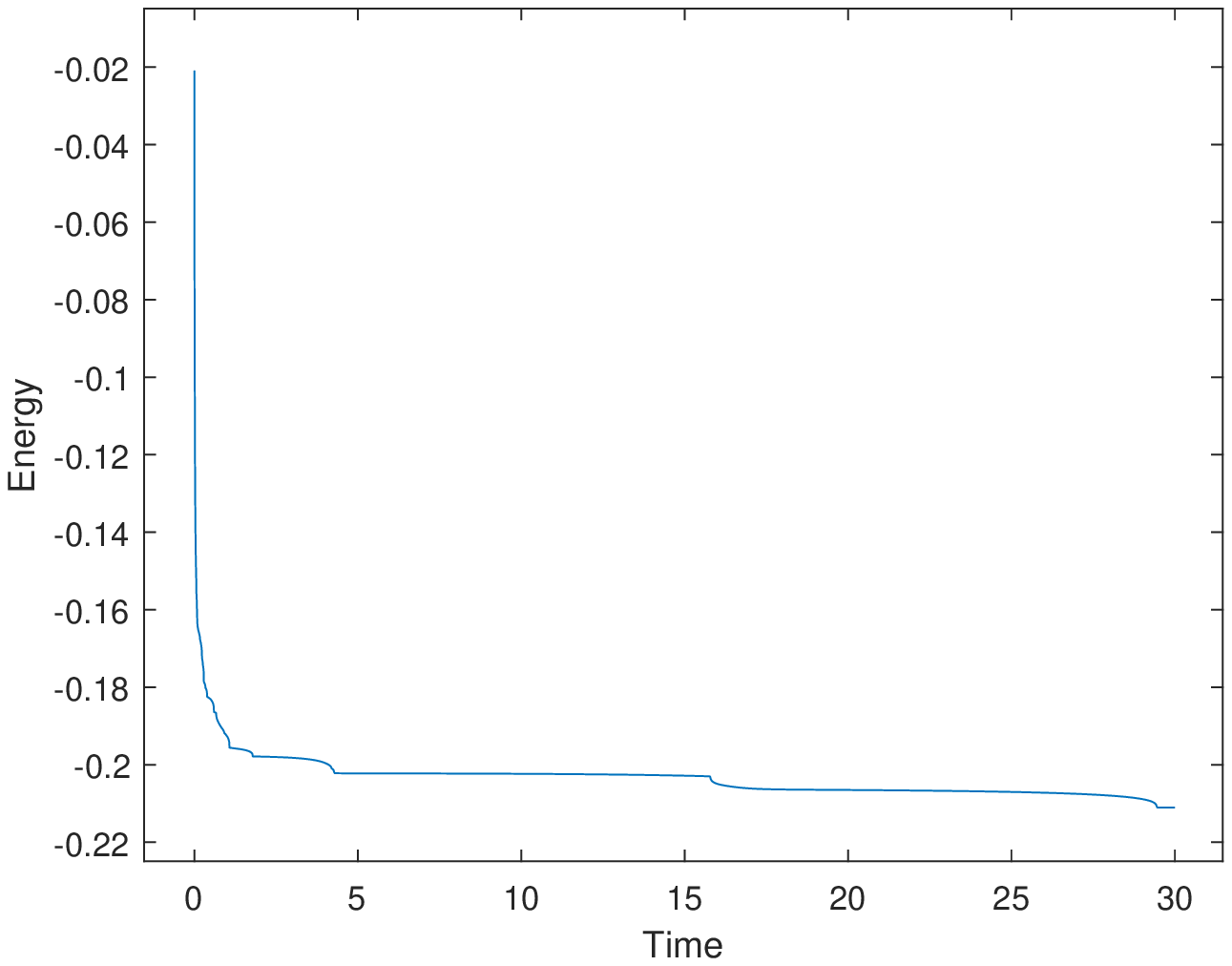}   
}\quad
\subfigure[Mass]{
\includegraphics[height=4cm,width=6cm]{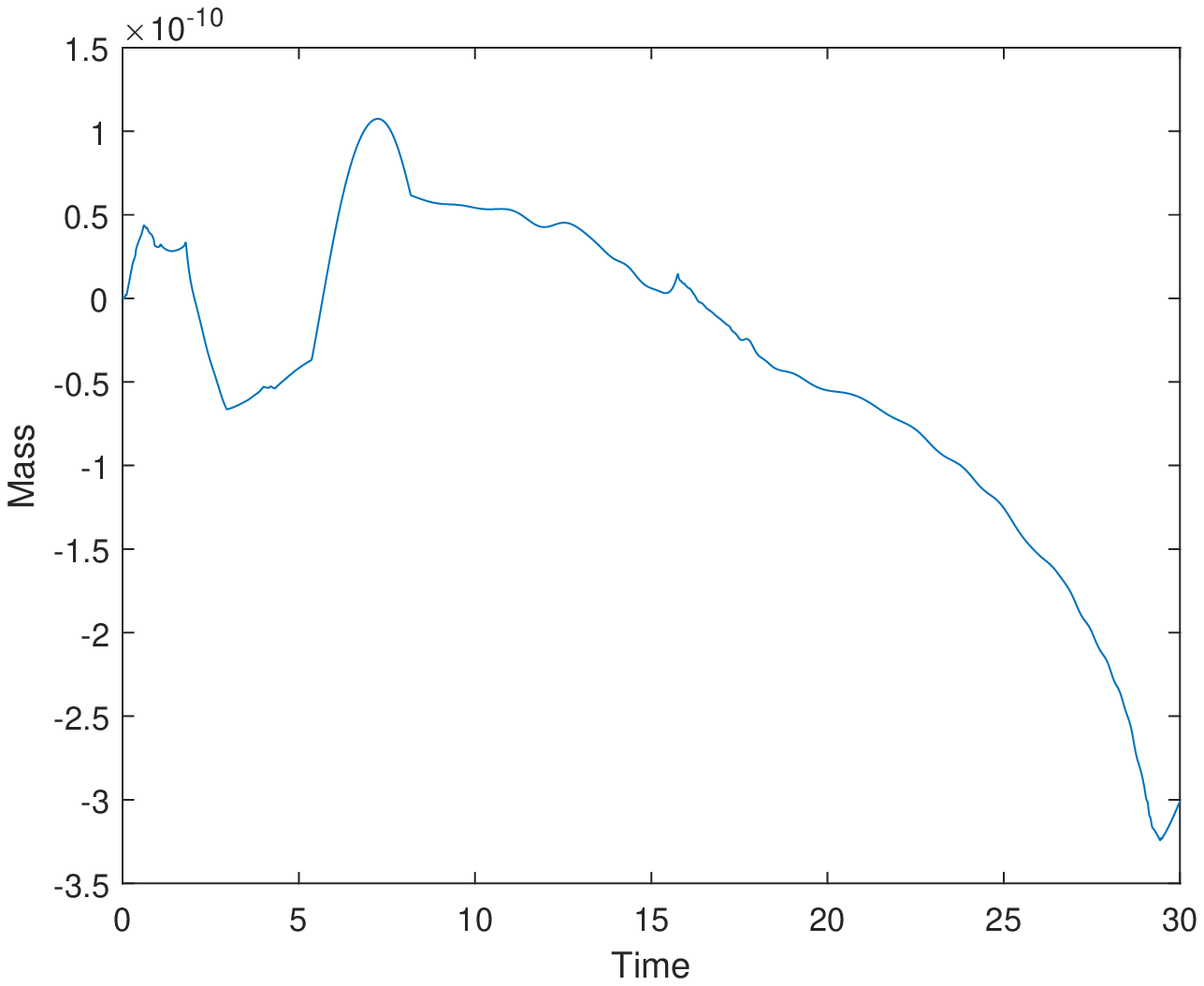}   
}
\caption{Energy decay and mass conservation with rand initial condition \eqref{rand initial}.}\label{energy decay and mass conservation}
\end{figure}

The left plot of Figure \ref{energy decay and mass conservation} illustrates the evolution of discrete energy in terms of time, which confirms the energy dissipation property. The rough estimate of the mass difference computed as $\bar{\phi}^n-\bar{\phi}^0$ is displayed in the right plot of Figure \ref{energy decay and mass conservation}, which numerically verifies the mass conservation property up to a machine error.

In addition, similar computations have been performed with trigonometric initial conditions,
\begin{equation}
\phi^0=0.9*\left(\frac{\left(1-\cos\left(4\pi x\right)\right)\left(1-\cos\left(4\pi y\right)\right)}{2}-1\right),
\label{tri initial data}
\end{equation}
the bound for which is adjusted to make the logarithmic energy meaningful. Parameters are the same as the last numerical test with random initial condition \eqref{rand initial}. Evolution of $\phi$ at selected time instants is displayed in Figure \ref{tri Spinodal decomposition}. Numerical verifications of energy dissipation and mass conservation are presented in Figure \ref{tri energy decay and mass conservation}. 

\begin{figure}[h]
\center
\subfigure[T=0]{
\includegraphics[height=3cm,width=4cm]{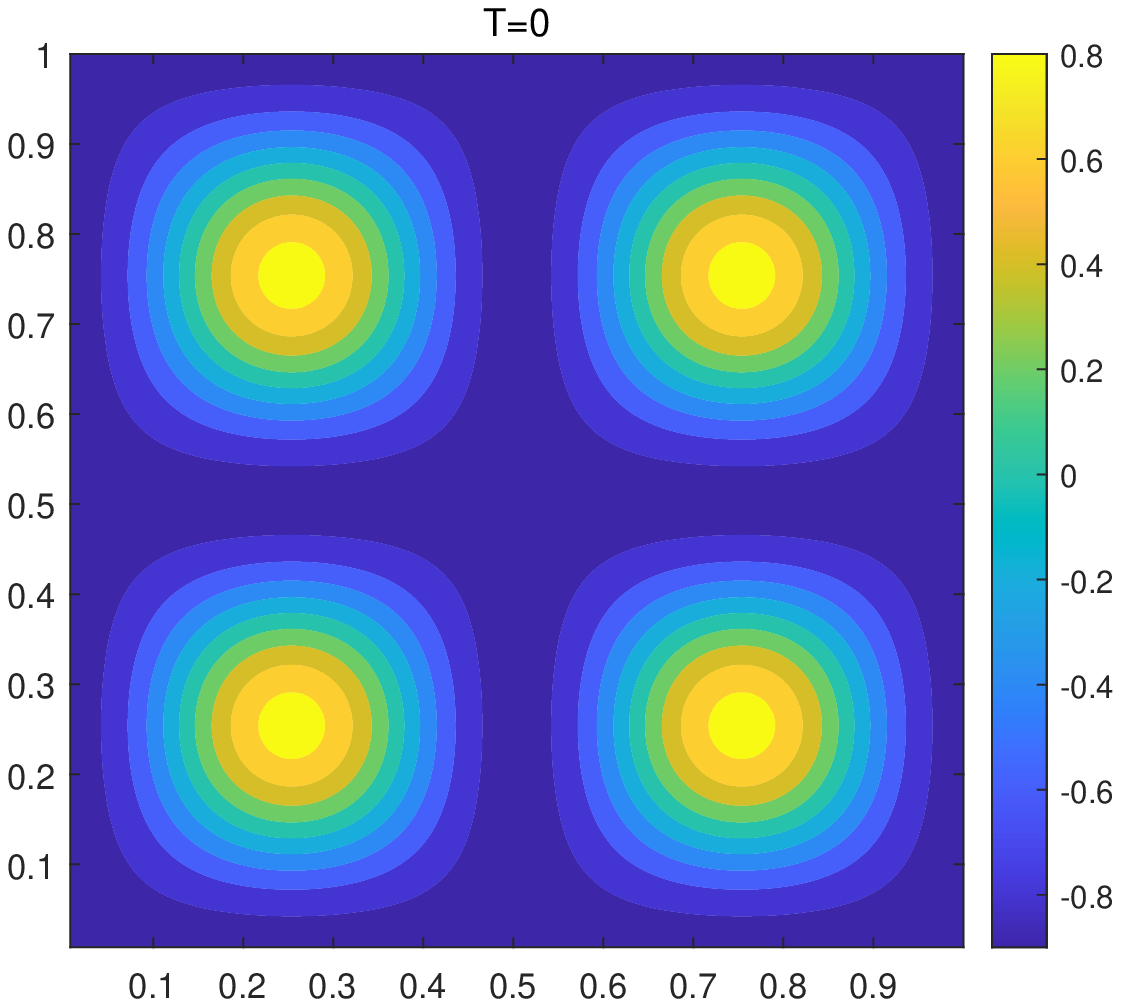}   
}
\subfigure[T=0.0006]{
\includegraphics[height=3cm,width=4cm]{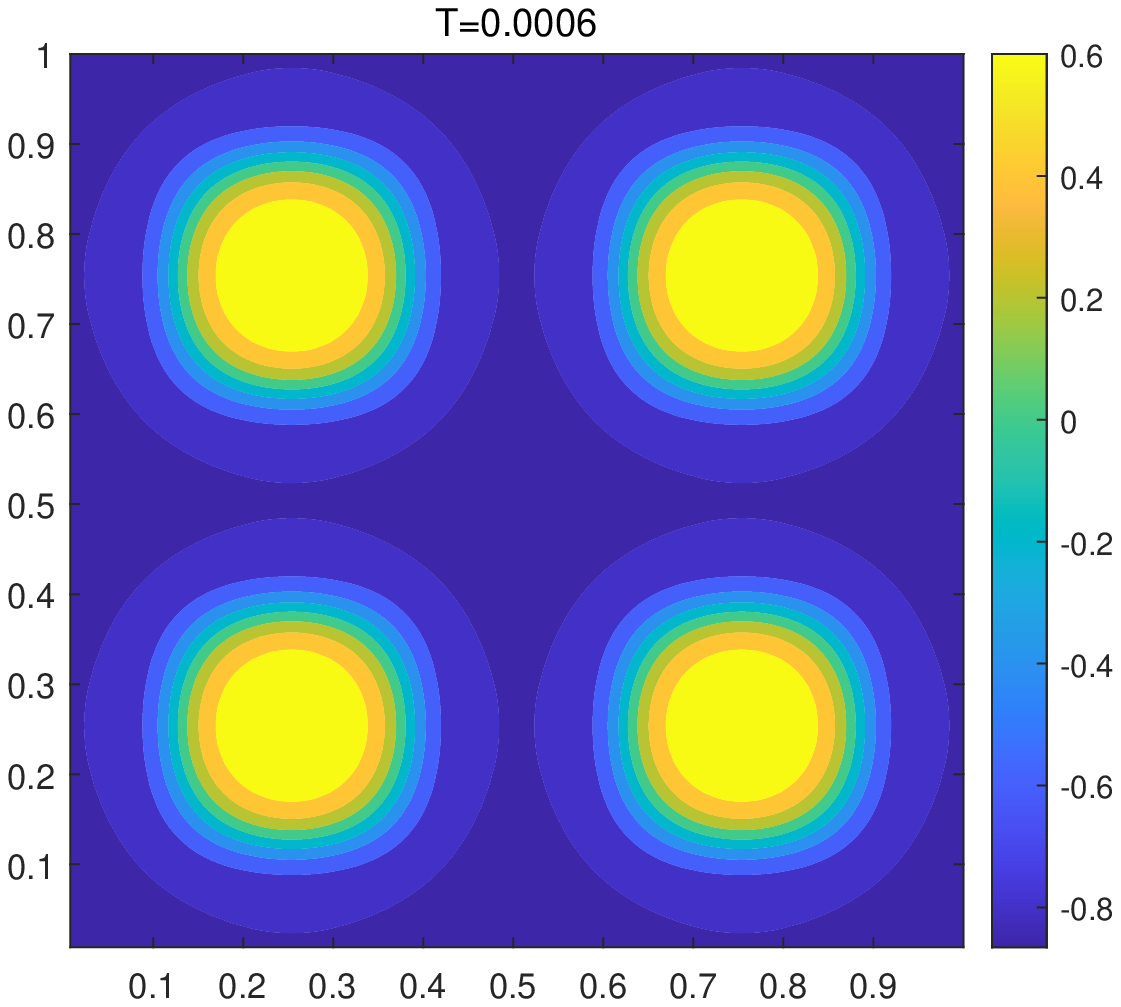}   
}
\subfigure[T=0.002]{
\includegraphics[height=3cm,width=4cm]{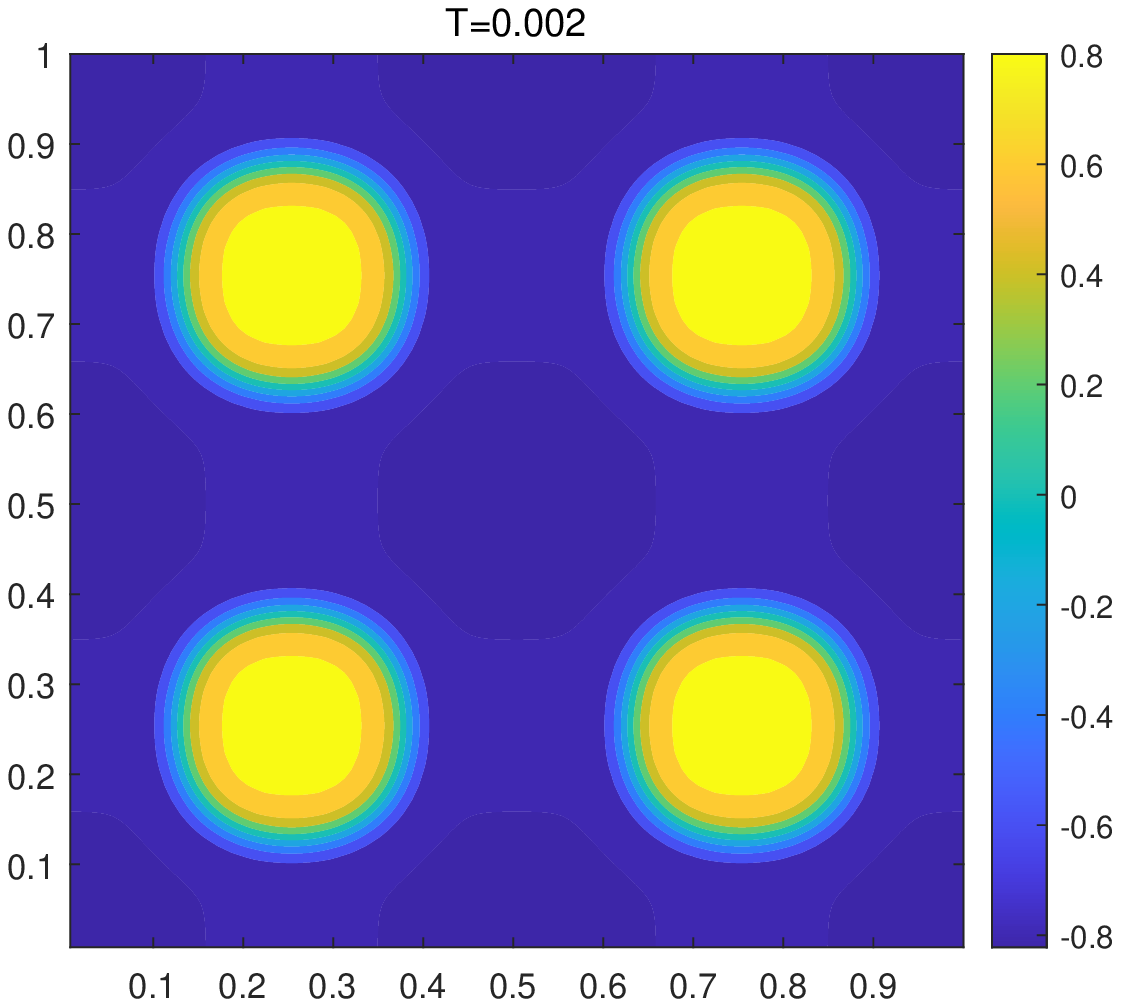}   
}

\subfigure[T=0.4]{
\includegraphics[height=3cm,width=4cm]{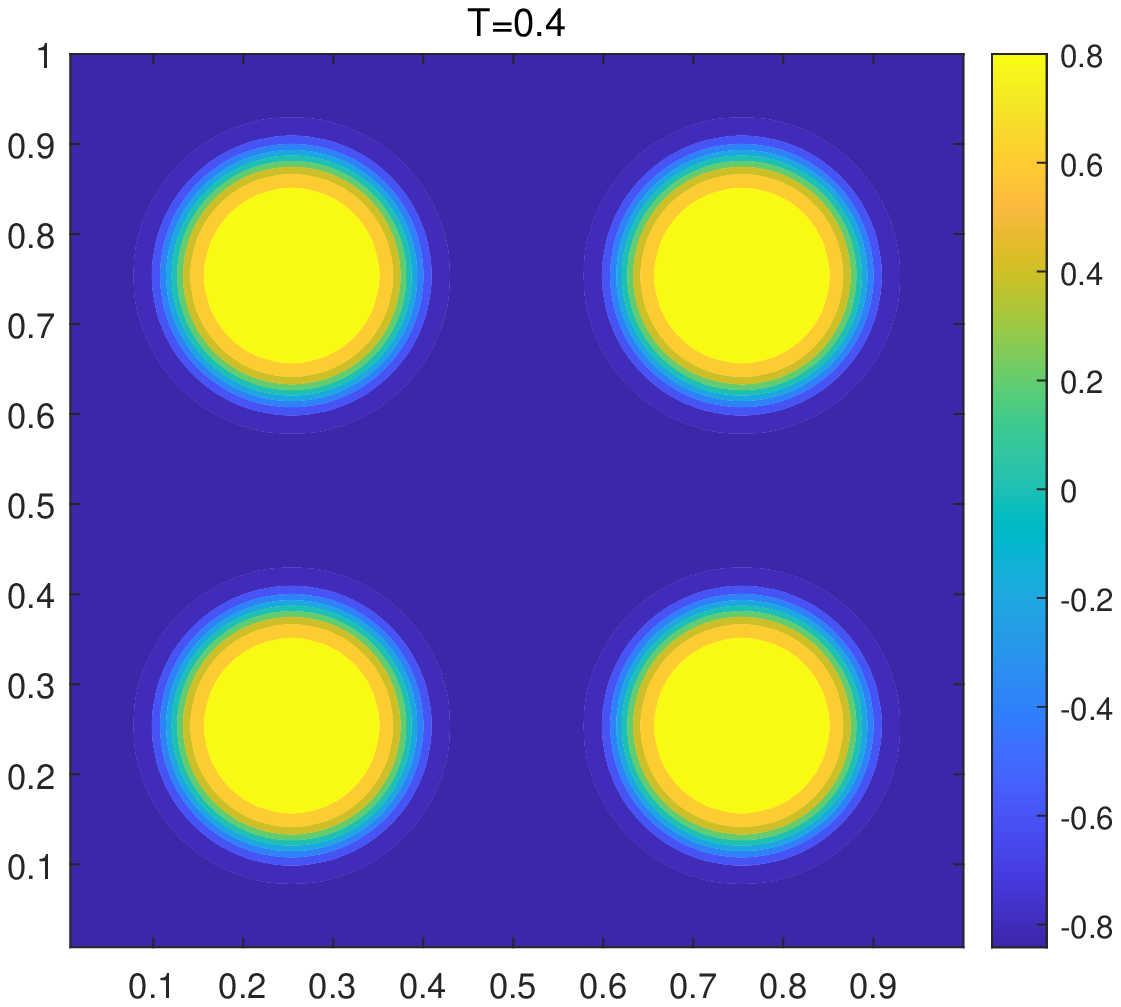}   
}
\subfigure[T=4]{
\includegraphics[height=3cm,width=4cm]{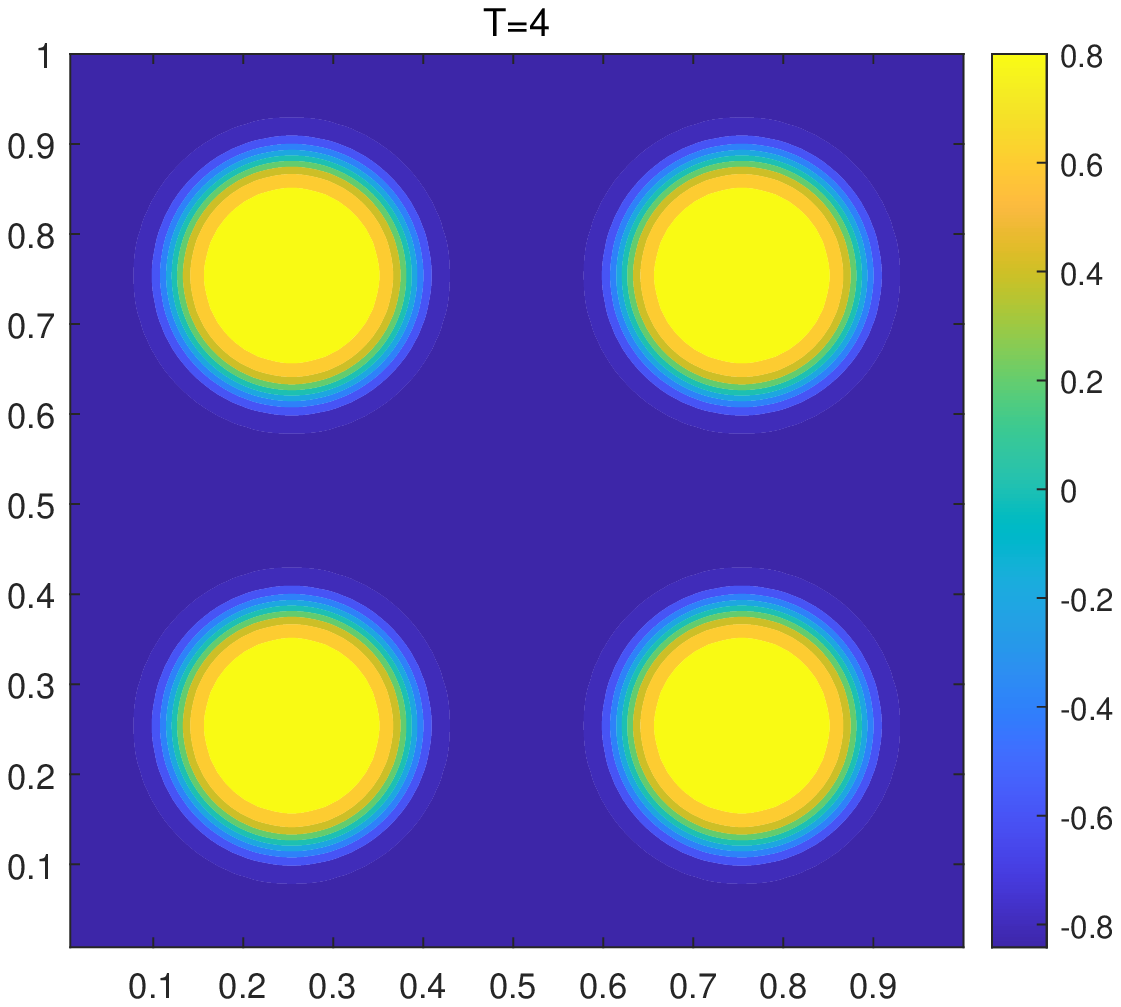}   
}
\caption{Evolution of the phase variable at selected time instants with trigonometric condition \eqref{tri initial data}. } \label{tri Spinodal decomposition}
\end{figure}

\begin{figure}[h]
\center
\subfigure[Energy]{
\includegraphics[height=3cm,width=4cm]{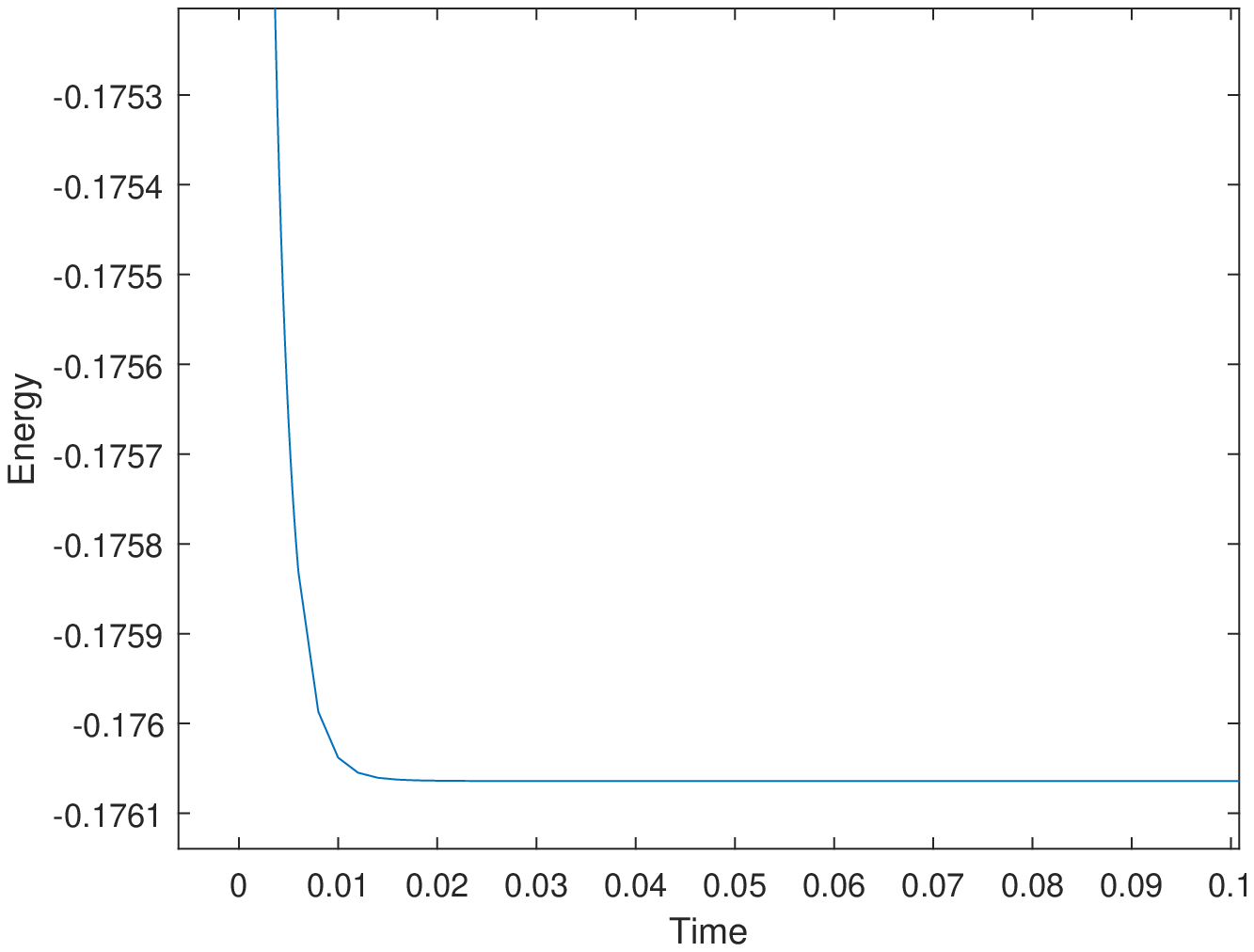}   
}\quad
\subfigure[Mass]{
\includegraphics[height=3cm,width=4cm]{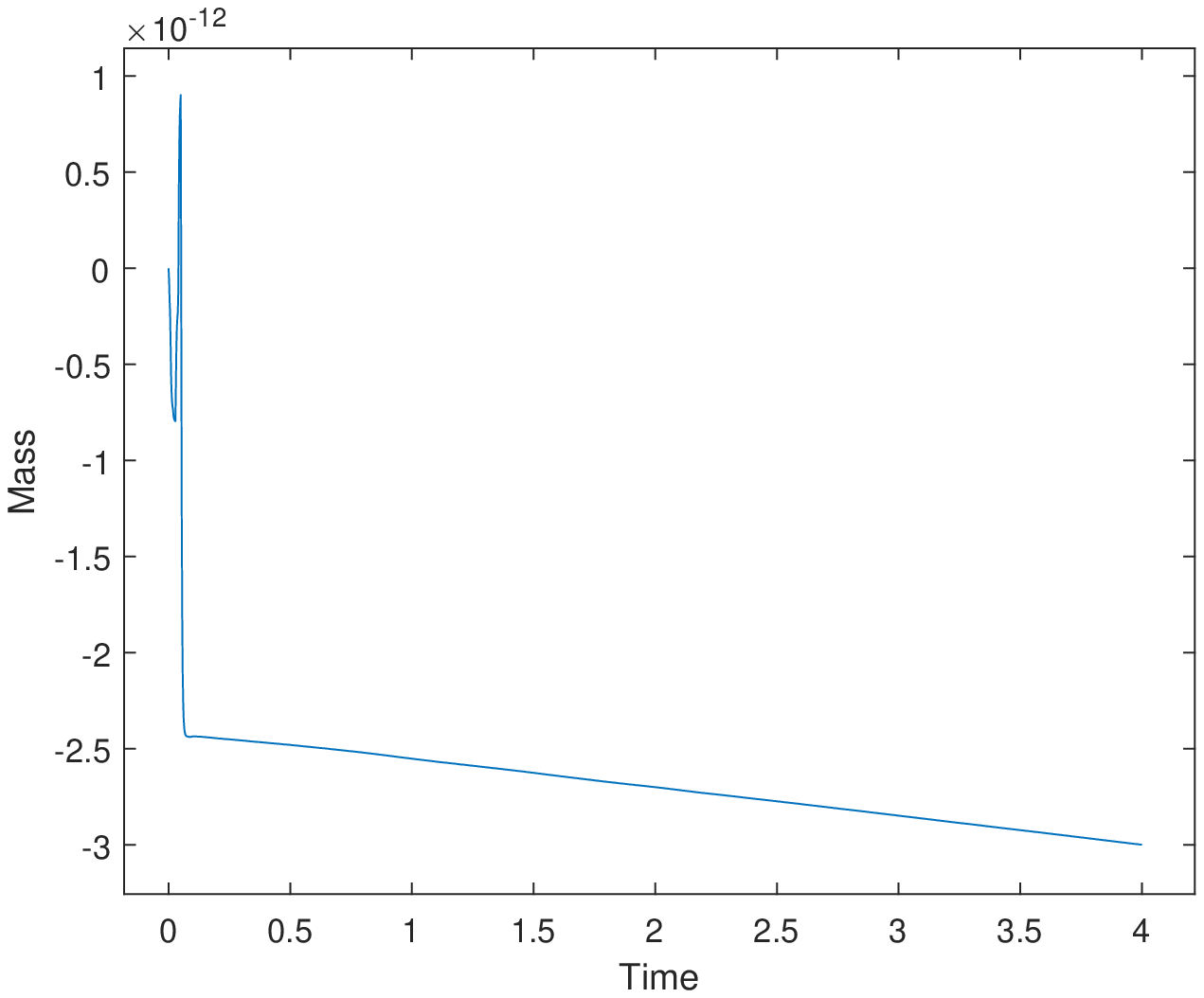}   
}
\caption{Test of energy decay and mass conservation with trigonometric initial condition \eqref{tri initial data}.}\label{tri energy decay and mass conservation}
\end{figure}

It is observed that the concentration variable $\phi$ stays stable, barely changing for a very long time. The same is true for the free energy. The left plot of Figure~\ref{tri energy decay and mass conservation} illustrates the energy evolution from~$t=0$~to~$t=0.1$, since there is an extremely sharp decline in this range. However, the free energy is dissipated up to~$t=4$.

\subsection{Convergence order}\label{Convergence order}

Now we present a convergence test for the numerical scheme \eqref{discrete-CHS-0.1} -- \eqref{discrete-CHS-0.4}, as $s,h \rightarrow 0$. Smooth initial data is taken via
\begin{equation}
\phi^0=0.24*\cos \left(2\pi x\right) \cos \left(2\pi y\right) + 0.4*\cos \left(\pi x\right) \cos \left( 3\pi y \right) . 
\label{convergence_initial}
\end{equation}
The diffuse interface coefficient is set to $\varepsilon=0.05$. We expect that the global error is of order $e_{t=T}=O(s)+O(h^2)$. In turn, with a refinement path $s=Ch^2$, we see that $e_{t=T}=O(h^2)$. In practice, we set $s=0.02h^2$, the tolerant error for the FAS approach is set as $\tau = 10^{-8}$ and the final time is given by $T=0.02$. Considering the multiple grid size and the definition of the cell-center function, the following error expression is proposed:
\begin{equation}
e^{h-h/2}_{i,j}=\phi^h_{i,j}-\frac{1}{4}\left(\phi^{h/2}_{2i,2j}+\phi^{h/2}_{2i-1,2j}+\phi^{h/2}_{2i,2j-1)}+\phi^{h/2}_{2i-1,2j-1}\right).
\end{equation}

The results are displayed in Table \ref{Convergence}, which confirms the second order accuracy in space, as well as the first order accuracy in time. 

\begin{table}[h]
\caption{Numerical convergence test with initial data \eqref{tri initial data}}\label{Convergence}.
\center
\begin{tabular}{@{}lllll@{}}
\toprule
Grid size       & $16^2-32^2 $    & $32^2-64^2$     & $64^2-128^2$    & $128^2-256^2$   \\ 
$L^2$ error     & 1.9287E-02      & 4.5851E-03      & 1.1269E-03      & 2.8061E-04  \\
$L^2$ rate      &                 & 2.0727          & 2.0245          & 2.0057     \\
$L^\infty$ error & 5.1703E-02     & 1.1344E-02      & 2.9196E-03      & 7.3025E-05  \\
$L^\infty$ rate &                 & 2.1882          & 1.9581          & 1.9993     \\ \bottomrule
\end{tabular}
\end{table}

\section{Conclusions}  \label{sec:conclusion} 
In this paper, we have presented a fully discrete finite difference numerical scheme of the Cahn-Hilliard-Stokes (CHS) system with Florry-Huggins energy potential. A convex splitting technique is applied to treat the chemical potential, combined with a semi-implicit computation of the nonlinear convection term, and an implicit update of the static Stokes equation. An implicit treatment of the logarithmic term ensures the positivity-preserving property, which comes from its singular nature as the phase variable approaches the singular limit values. An unconditional energy stability is derived by a careful energy estimate.  Moreover, an optimal rate convergence analysis and error estimate has been established at a theoretical level, with the help of higher order consistency analysis, combined with rough and refined error (RRE) estimates. Some numerical experiments have also been presented, which demonstrate the theoretical properties of the proposed numerical scheme.

	\section*{Acknowledgements}
	
C.~Wang is partially supported by the NSF DMS-2012269. S.M. Wise is partially supported by the NSF DMS-2012634. Z.R.~Zhang is partially supported by the NSFC No.11871105 and Science Challenge Project No. TZ2018002. In addition, Y.Z.~Guo also thanks the Hong Kong Polytechnic University for the generous support and hospitality during his visit.

\appendix
\section{Proof of Lemma \ref{rough convergence lemma}}\label{appendix A}
Taking a discrete inner product with \eqref{error equation-1.1} by $\tilde{\mu}^{n+1}$ leads to 
\begin{equation}
\begin{aligned}
\frac{1}{s}(\tilde{\phi}^{n+1},\tilde{\mu}^{n+1})&+\|\nabla_h \tilde{\mu}^{n+1}\|_2^2-(A_h\phi^n\nabla_h\tilde{\mu}^{n+1},\tilde{\boldsymbol{u}}^{n+1})\\
&=(A_h\tilde{\phi}^n\hat{\boldsymbol{U}}^{n+1},\nabla_h\tilde{\mu}^{n+1})+(\tau_{\phi}^{n+1},\tilde{\mu}^{n+1})+\frac{1}{s}(\tilde{\phi}^n,\tilde{\mu}^{n+1}).
\end{aligned}
\label{proof lemma 1}
\end{equation}
Based on the mean-free property \eqref{consistency-8.3} of truncation error, the following estimate could be obtained
\begin{equation}
(\tau_{\phi}^{n+1},\tilde{\mu}^{n+1}) \leq \|\tau_{\phi}^{n+1}\|_{-1,h}\cdot \|\nabla_h \tilde{\mu}^{n+1}\|_2 \leq 2\|\tau_{\phi}^{n+1}\|_{-1,h}^2+\frac{1}{8}\|\nabla_h \tilde{\mu}^{n+1}\|_2^2 . \label{proof lemma 2}
\end{equation}
For the $(\tilde{\phi}^n,\tilde{\mu}^{n+1})$ term, a similar analysis is valid
\begin{equation}
\frac{1}{s}(\tilde{\phi}^n,\tilde{\mu}^{n+1}) \leq \frac{1}{s}\|\tilde{\phi}^n\|_{-1,h} \cdot \|\nabla_h \tilde{\mu}^{n+1}\|_2 \leq \frac{2}{s^2}\|\tilde{\phi}^n\|_{-1,h}^2+\frac{1}{8}\|\nabla_h \tilde{\mu}^{n+1}\|_2^2.\label{proof lemma 3}
\end{equation}
For the first term of right hand of \eqref{proof lemma 1}, we see that
\begin{equation}
\begin{aligned}
(A_h\tilde{\phi}^n\hat{\boldsymbol{U}}^{n+1},\nabla_h\tilde{\mu}^{n+1}) &\leq \|\hat{\boldsymbol{U}}^{n+1}\|_\infty \cdot \|\tilde{\phi}^n\|_2 \cdot \|\nabla_h\tilde{\mu}^{n+1}\|_2 \leq C^*\|\tilde{\phi}^n\|_2 \cdot \|\nabla_h\tilde{\mu}^{n+1}\|_2 \\
&\leq 2(C^*)^2\|\tilde{\phi}^n\|_2^2 + \frac{1}{8}\|\nabla_h\tilde{\mu}^{n+1}\|_2^2. \label{proof lemma 4}
\end{aligned}
\end{equation}
For the last term of right hand of \eqref{proof lemma 1}, we begin with the following identity 
\begin{equation}
-A_h\phi^n\nabla_h\tilde{\mu}^{n+1}=\frac{1}{\gamma}\left((-\Delta_h+I)\tilde{\boldsymbol{u}}^{n+1}+\nabla_h\tilde{p}^{n+1}-\tau_v^{n+1}\right)+A_h \tilde{\phi}^n \nabla_h \hat{\mathcal{V}}^{n+1} , 
\label{proof lemma 5}
\end{equation}
so that the following estimates are available
\begin{equation}
\begin{aligned}
-(A_h&\phi^n\nabla_h\tilde{\mu}^{n+1},\tilde{\boldsymbol{u}}^{n+1})=\left(\frac{1}{\gamma}(\tilde{\boldsymbol{u}}^{n+1}-\Delta_h\tilde{\boldsymbol{u}}^{n+1}+\nabla_h\tilde{p}^{n+1}-\tau_v^{n+1})+A_h \tilde{\phi}^n \nabla_h \hat{\mathcal{V}}^{n+1},\tilde{\boldsymbol{u}}^{n+1}\right) \\
=&\frac{1}{\gamma}\|\tilde{\boldsymbol{u}}^{n+1}\|_2^2+\frac{1}{\gamma}\|\nabla_h\tilde{\boldsymbol{u}}^{n+1}\|_2^2-\frac{1}{\gamma}(\tau_v^{n+1}, \tilde{\boldsymbol{u}}^{n+1})+(A_h \tilde{\phi}^n \nabla_h \hat{\mathcal{V}}^{n+1},\tilde{\boldsymbol{u}}^{n+1})\\
\geq & \frac{1}{\gamma}\|\tilde{\boldsymbol{u}}^{n+1}\|_2^2+\frac{1}{\gamma}\|\nabla_h\tilde{\boldsymbol{u}}^{n+1}\|_2^2-\frac{1}{\gamma}\|\tau_v^{n+1}\|_2\cdot \|\tilde{\boldsymbol{u}}^{n+1}\|_2-\|\nabla_h \hat{\mathcal{V}}^{n+1}\|_\infty\cdot \|\tilde{\phi}^n\|_2 \cdot \|\tilde{\boldsymbol{u}}^{n+1}\|_2\\
\geq & \frac{1}{\gamma}\|\tilde{\boldsymbol{u}}^{n+1}\|_2^2+\frac{1}{\gamma}\|\nabla_h\tilde{\boldsymbol{u}}^{n+1}\|_2^2-\frac{2}{\gamma}\|\tau_v^{n+1}\|_2^2-\frac{1}{8\gamma}\|\tilde{\boldsymbol{u}}^{n+1}\|_2^2-C\|\tilde{\phi}^n\|_2^2-\frac{1}{8\gamma}\|\tilde{\boldsymbol{u}}^{n+1}\|_2^2\\
\geq & \frac{3}{4\gamma}\|\tilde{\boldsymbol{u}}^{n+1}\|_2^2+\frac{1}{\gamma}\|\nabla_h\tilde{\boldsymbol{u}}^{n+1}\|_2^2-\frac{2}{\gamma}\|\tau_v^{n+1}\|_2^2-C\|\tilde{\phi}^n\|_2^2.\label{proof lemma 6}
\end{aligned}
\end{equation}
Meanwhile, an application of intermediate value theorem implies a point-wise representation:  
\begin{equation} 
   \ln ( 1 + \hat{\Phi}^{n+1} )  - \ln (1 + \phi^{n+1} ) 
   = \frac{\tilde{\phi}^{n+1} }{ 1 + \eta^{(n+1)} } ,  \quad 
   \mbox{$\eta^{(n+1)}$ is between $\phi^{n+1}$ and $\hat{\Phi}^{n+1}$} . 
   \label{proof lemma 7}   
\end{equation} 
By the point-wise bound that $-1 < \phi^{n+1} , \hat{\Phi}^{n+1} < 1$, we have $ 0 < 1 + \eta^{(n+1)}  < 2 $ so that $\frac{1}{1+\eta^{(n+1)}} > \frac{1}{2}$,
\begin{equation}
(\ln(1 + \hat{\Phi}^{n+1} )  - \ln (1 +\phi^{n+1}), \tilde{\phi}^{n+1}) =(\frac{\tilde{\phi}^{n+1}}{1+\eta^{(n+1)}}, \tilde{\phi}^{n+1}) \ge \frac{1}{2} \|\tilde{\phi}^{n+1}\|_2^2 . 
\label{proof lemma 8} 
\end{equation}
A similar analysis could be derived: 
\begin{equation}
(-\ln(1 - \hat{\Phi}^{n+1} )  + \ln (1 -\phi^{n+1}), \tilde{\phi}^{n+1}) \ge \frac{1}{2} \|\tilde{\phi}^{n+1}\|_2^2.
\label{proof lemma 9} 
\end{equation}
The two linear terms in the expansion of $( \tilde{\phi}^{n+1} , \tilde{\mu}^{n+1} )$ could be analyzed in a more straightforward way:
\begin{align} 
     & 
 - \theta_0 (\tilde{\phi}^n ,  \tilde{\phi}^{n+1})
  \ge - \frac{1}{2} \theta_0^2  \| \tilde{\phi}^n \|_2^2 - \frac{1}{2} \|  \tilde{\phi}^{n+1}  \|_2^2   , 
    \label{proof lemma 10}    
\\
     & 
 - (  \Delta_h \tilde{\phi}^{n+1} ,  \tilde{\phi}^{n+1} ) 
  = \|  \nabla_h \tilde{\phi}^{n+1}  \|_2^2  . 
    \label{proof lemma 11}        
\end{align}
Then we conclude that
\begin{equation}
(\tilde{\phi}^{n+1} , \tilde{\mu}^{n+1}) \geq \frac{1}{2}\|\tilde{\phi}^{n+1}\|^2_2+\frac{\varepsilon^2}{2}\|\nabla_h \tilde{\phi}^{n+1}\|^2_2- \frac{\theta_0^2}{2} \| \tilde{\phi}^n\|_2^2.\label{proof lemma 12}
\end{equation}
A substitution of \eqref{proof lemma 1}-\eqref{proof lemma 4}, \eqref{proof lemma 6} and \eqref{proof lemma 12} shows that
\begin{equation}
\begin{aligned}
&
\frac{1}{2}\|\tilde{\phi}^{n+1}\|^2_2+\frac{\varepsilon^2}{2}\|\nabla_h \tilde{\phi}^{n+1}\|^2_2+s(\frac{5}{8}\|\nabla_h \tilde{\mu}^{n+1}\|_2^2+\frac{3}{4\gamma}\|\tilde{\boldsymbol{u}}^{n+1}\|_2^2+\frac{1}{\gamma}\|\nabla_h\tilde{\boldsymbol{u}}^{n+1}\|_2^2)\\
&
\leq \frac{2}{s}\|\tilde{\phi}^n\|_{-1,h}^2+C\| \tilde{\phi}^n\|_2^2+s(\frac{2}{\gamma}\|\tau_v^{n+1}\|_2^2+2\|\tau_{\phi}^{n+1}\|_{-1,h}^2).
\end{aligned}\label{proof lemma con}
\end{equation}
For the right hand side of \eqref{proof lemma con}, the following estimates are available, which come from the a-priori assumption \eqref{a priori-1}: 
\begin{equation}
\begin{aligned}
\frac{2}{s}\|\tilde{\phi}^n\|_{-1,h}^2 &\leq \frac{C}{s}\|\tilde{\phi}^n\|_2^2 \leq C(s^\frac{11}{4}+h^\frac{11}{4}) , 
\\
C\|\tilde{\phi}^n\|_2^2 &\leq C(s^\frac{15}{4}+h^\frac{15}{4}) , 
\\
s(\frac{2}{\gamma}\|\tau_v^{n+1}\|_2^2+2\|\tau_{\phi}^{n+1}\|_{-1,h}^2) &\leq C(s^5+h^5),
\end{aligned}
\end{equation}
where the fact that $\| f \|_{-1,h} \leq C \| f \|_2$, as well as the refinement constraint $C_1 h \le s \le C_2 h$, have been repeatedly used. Going back to \eqref{proof lemma con}, we have
\begin{equation}
\|\tilde{\phi}^{n+1}\|_2+\|\nabla_h \tilde{\phi}^{n+1}\|_2 \leq C(s^\frac{11}{8}+h^\frac{11}{8}) \leq \hat{C}(s^\frac{5}{4}+h^\frac{5}{4}), 
\end{equation}
under the linear refinement requirement $C_1 h \le s \le C_2 h$, provided that $s$ and $h$ are sufficiently small. In addition, $\hat{C}$ depends on the physical parameters, while it is independent on $s$ and $h$. This inequality is exactly the rough error estimate \eqref{rough convergence}. The proof of Lemma \ref{rough convergence lemma} is complete. 

\section{Proof of Lemma \ref{nonlinear error term}}\label{appendix B}
We focus on the nonlinear term $\ln ( 1 + \hat{\Phi}^{n+1} ) - \ln ( 1 + \phi^{n+1} )$. The other terms could be similarly analyzed. The decomposition identity \eqref{proof lemma 7} is still valid. Considering a single mesh cell, we make the following observation
\begin{equation} 
\begin{aligned}  
  &
  D_x (  \ln ( 1 + \hat{\Phi}^{n+1} ) - \ln ( 1 + \phi^{n+1} ) )_{i+\frac{1}{2},j,k} 
  = D_x (  \frac{1}{1 + \eta^{(n+1)} } \cdot \tilde{\phi}^{n+1} )_{i+\frac{1}{2},j,k}  
  =: {\cal NLE}_1 + {\cal NLE}_2 , 
\\
  &
  {\cal NLE}_1 = 
   \frac{1}{1 + \eta^{(n+1)}_{i+1,j,k} }  D_x \tilde{\phi}^{n+1}_{i+\frac{1}{2},j,k}  ,  \quad 
   {\cal NLE}_2 
  = \tilde{\phi}^{n+1}_{i,j,k}  D_x (  \frac{1}{1 + \eta^{(n+1)} } )_{i+\frac{1}{2},j,k}   .   
\end{aligned} 
  \label{nonlinear est-1}   
\end{equation} 
The bound for the first nonlinear expansion is straightforward: 
\begin{align} 
  & 
   0 < \frac{1}{1 + \eta^{(n+1)}_{i+1,j,k} } \le  ( \frac{1}{2} \epsilon_0^* )^{-1} 
   = 2 (\epsilon_0^*)^{-1} ,  \quad 
   \mbox{(by~\eqref{separation property hat}, \eqref{separation property numerical})} , 
   \label{nonlinear est-2-1}      
\\
  & 
  \mbox{so that} \quad 
   \| \frac{1}{1 + \eta^{(n+1)} }   \|_\infty \le 2 (\epsilon_0^*)^{-1} ,   
    \label{nonlinear est-2-2}      
\\
  &   
   \|  {\cal NLE}_1  \|_2 \le \| \frac{1}{1 + \eta^{(n+1)} }   \|_\infty 
   \cdot \| D_x \tilde{\phi}^{n+1} \|_2 \le 2 (\epsilon_0^*)^{-1}  \| D_x \tilde{\phi}^{n+1} \|_2 .  
   \label{nonlinear est-2-3}       
\end{align} 
Meanwhile, by the decomposition identity~\eqref{proof lemma 7}, we denote 
\begin{equation} 
  e_{i,j,k} = \hat{\Phi}^{n+1}_{i,j,k} - \eta^{(n+1)}_{i,j,k} .   \label{nonlinear est-3-1}       
\end{equation} 
It is clear that 
\begin{equation} 
  | e_{i,j,k} | \le | \hat{\Phi}^{n+1}_{i,j,k} - \phi^{n+1}_{i,j,k} | 
  = | \tilde{\phi}^{n+1}_{i,j,k} | ,  \quad \forall (i,j,k) .   \label{nonlinear est-3-2}       
\end{equation} 
By the discrete $\|\cdot\|_4$ rough estimation \eqref{convergence-rough-1.2}, an application of inverse inequation gives
\begin{equation} 
  \| \nabla_h e \|_4  \le \frac{ C \| e \|_4 }{h} 
  \le \frac{ C \| \tilde{\phi}^{n+1} \|_4 }{h}   \le \frac{C ( s^\frac{5}{4} + h^\frac{5}{4} ) }{h} 
  \le C ( s^\frac{1}{4} + h^\frac{1}{4} ) \le \frac{1}{2} ,   
  \label{nonlinear est-3-3}    
\end{equation} 
under the linear refinement requirement $C_1 h \le s \le C_2 h$, provide that $s$ and $h$ are sufficiently small. In turn, we see that
\begin{equation} 
\begin{aligned} 
  & 
     ( D_x \eta^{(n+1)} )_{i+\frac{1}{2},j,k}  = ( D_x \hat{\Phi}^{n+1} )_{i+\frac{1}{2},j,k}  
     -  ( D_x e )_{i+\frac{1}{2},j,k} , 
\\
  & 
     \| D_x \eta^{(n+1)} \|_4  \le \| D_x \hat{\Phi}^{n+1} \|_4  
     + \| D_x e \|_4  \le C^* + \frac{1}{2}. 
\end{aligned} 
   \label{nonlinear est-3-4}    
\end{equation}   
On the other hand, motivated by the following expansion 
\begin{equation} 
  D_x (  \frac{1}{1 + \eta^{(n+1)} } )_{i+\frac{1}{2},j,k}  
  = \frac{ - ( D_x \eta^{(n+1)} )_{i+\frac{1}{2},j,k} }{ ( 1+ \eta^{n+1}_{i,j,k} ) (1 +  \eta^{n+1}_{i+1,j,k} ) } , 
   \label{nonlinear est-3-5}    
\end{equation}   
we conclude that 
\begin{equation} 
  \| D_x (  \frac{1}{1 + \eta^{(n+1)} } ) \|_4   
  \le \max_{i,j,k} \frac{ 1 }{ ( 1 + \eta^{n+1}_{i,j,k} ) (1 + \eta^{n+1}_{i+1,j,k} ) } 
     \cdot \| D_x \eta^{(n+1)} \|_4   
     \le 2 ( \epsilon_0^* )^{-2}(C^* + \frac{1}{2}) ,   
   \label{nonlinear est-3-6}    
\end{equation}   
in which the phase separation estimates \eqref{separation property hat} and \eqref{separation property numerical} have been applied again. Then we arrive at 
\begin{equation} 
   \| {\cal NLE}_2 \|_2 
  \le   \| D_x (  \frac{1}{1 + \eta^{(n+1)} } ) \|_4 \cdot    \| \tilde{\phi}^{n+1} \|_4  
  \le 2 ( \epsilon_0^* )^{-2} \cdot C\| \tilde{\phi}^{n+1} \|_4 .  
   \label{nonlinear est-3-7}    
\end{equation}  
Subsequently, a combination of~\eqref{nonlinear est-2-3} and \eqref{nonlinear est-3-7} leads to 
\begin{equation} 
    \| D_x (  \ln ( 1 + \hat{\Phi}^{n+1} ) - \ln ( 1 + \phi^{n+1} ) ) \|_2 
    \le 2 (\epsilon_0^*)^{-1}  \| D_x \tilde{\phi}^{n+1} \|_2 
    +   2 ( \epsilon_0^* )^{-2} C \| \tilde{\phi}^{n+1} \|_4 . 
    \label{nonlinear est-4}    
\end{equation} 
Similar estimates could be derived in the $y$ and $z$ directions; the technical details are skipped for the sake of brevity: 
\begin{align} 
    & 
    \| D_y (  \ln ( 1 + \hat{\Phi}^{n+1} ) - \ln ( 1 + \phi^{n+1} ) ) \|_2 
    \le 2 (\epsilon_0^*)^{-1}  \| D_y \tilde{\phi}^{n+1} \|_2 
    +   2 ( \epsilon_0^* )^{-2} C \| \tilde{\phi}^{n+1} \|_4 , 
    \label{nonlinear est-5-1}    
\\
  & 
    \| D_z (  \ln ( 1 + \hat{\Phi}^{n+1} ) - \ln ( 1 + \phi^{n+1} ) ) \|_2 
    \le 2 (\epsilon_0^*)^{-1}  \| D_z \tilde{\phi}^{n+1} \|_2 
    +   2 ( \epsilon_0^* )^{-2} C \| \tilde{\phi}^{n+1} \|_4 .  
    \label{nonlinear est-5-2}   
\end{align}    
Therefore, a combination of~\eqref{nonlinear est-4}-\eqref{nonlinear est-5-2} yields 
\begin{equation} 
    \| \nabla_h (  \ln ( 1 + \hat{\Phi}^{n+1} ) - \ln ( 1 + \phi^{n+1} ) ) \|_2 
    \le 2 (\epsilon_0^*)^{-1}  \| \nabla_h \tilde{\phi}^{n+1} \|_2 
    +   2 \sqrt{3} ( \epsilon_0^* )^{-2} C \| \tilde{\phi}^{n+1} \|_4 . 
    \label{nonlinear est-6}    
\end{equation}  
A similar estimate could also derived for the error term of  $-\ln ( 1 - \hat{\Phi}^{n+1} ) + \ln ( 1 - \phi^{n+1} )$:       
\begin{equation} 
    \| \nabla_h (  - \ln ( 1 - \hat{\Phi}^{n+1} ) + \ln ( 1 - \phi^{n+1} ) ) \|_2 
    \le 2 (\epsilon_0^*)^{-1}  \| \nabla_h \tilde{\phi}^{n+1} \|_2 
    +   2 \sqrt{3} ( \epsilon_0^* )^{-2} C \| \tilde{\phi}^{n+1} \|_4 . 
    \label{nonlinear est-7}    
\end{equation}
Finally, a substitution of~\eqref{nonlinear est-6} and \eqref{nonlinear est-7} into the nonlinear error expansion \eqref{nonlinear error-1} results in the desired estimate \eqref{nonlinear error-2}.  This finishes the proof of Lemma \ref{nonlinear error term}.

	\bibliographystyle{plain}
	\bibliography{ref}

\begin{thebibliography}{10}

\bibitem{barrett99}
J.~Barrett and J.~Blowey.
\newblock Finite element approximation of the {Cahn-Hilliard} equation with
  concentration dependent mobility.
\newblock {\em Math. Comp.}, 68:487--517, 1999.

\bibitem{chen19a}
W.~Chen, W.~Feng, Y.~Liu, C.~Wang, and S.M. Wise.
\newblock A second order energy stable scheme for the {Cahn-Hilliard-Hele-Shaw}
  equation.
\newblock {\em Discrete Contin. Dyn. Syst. Ser. B}, 24(1):149--182, 2019.

\bibitem{chen22b}
W.~Chen, D.~Han, C.~Wang, S.~Wang, X.~Wang, and Y.~Zhang.
\newblock Error estimate of a decoupled numerical scheme for the
  {Cahn-Hilliard-Stokes-Darcy} system.
\newblock {\em IMA J. Numer. Anal.}, 42:2621--2655, 2022.

\bibitem{chen22c}
W.~Chen, J.~Jing, C.~Wang, and X.~Wang.
\newblock A positivity preserving, energy stable finite difference scheme for
  the {Flory-Huggins-Cahn-Hilliard-Navier-Stokes} system.
\newblock {\em J. Sci. Comput.}, 92:31, 2022.

\bibitem{chen22a}
W.~Chen, J.~Jing, C.~Wang, X.~Wang, and S.~Wise.
\newblock A modified {Crank-Nicolson} scheme for the {Flory-Huggins
  Cahn-Hilliard} model.
\newblock {\em Commun. Comput. Phys.}, 31:60--93, 2022.

\bibitem{chen16}
W.~Chen, Y.~Liu, C.~Wang, and S.M. Wise.
\newblock An optimal-rate convergence analysis of a fully discrete finite
  difference scheme for {Cahn-Hilliard-Hele-Shaw} equation.
\newblock {\em Math. Comp.}, 85:2231--2257, 2016.

\bibitem{chen19b}
W.~Chen, C.~Wang, X.~Wang, and S.M. Wise.
\newblock Positivity-preserving, energy stable numerical schemes for the
  {Cahn-Hilliard} equation with logarithmic potential.
\newblock {\em J. Comput. Phys.: X}, 3:100031, 2019.

\bibitem{collins13}
C.~Collins, J.~Shen, and S.M. Wise.
\newblock An efficient, energy stable scheme for the {Cahn-Hilliard-Brinkman}
  system.
\newblock {\em Commun. Comput. Phys.}, 13:929--957, 2013.

\bibitem{Della_nonlocalCHHS_2018}
P.F. Della, A.~Giorgini, and M.~Grasselli.
\newblock The nonlocal {Cahn-Hilliard-Hele-Shaw} system with logarithmic
  potential.
\newblock {\em Nonlinearity}, 31:4854--4881, 2018.

\bibitem{diegel15a}
A.~Diegel, X.~Feng, and S.M. Wise.
\newblock Convergence analysis of an unconditionally stable method for a
  {Cahn-Hilliard-Stokes} system of equations.
\newblock {\em SIAM J. Numer. Anal.}, 53:127--152, 2015.

\bibitem{diegel17}
A.~Diegel, C.~Wang, X.~Wang, and S.M. Wise.
\newblock Convergence analysis and error estimates for a second order accurate
  finite element method for the {Cahn-Hilliard-Navier-Stokes} system.
\newblock {\em Numer. Math.}, 137:495--534, 2017.

\bibitem{Dong2020b}
L.~Dong, C.~Wang, S.M. Wise, and Z.~Zhang.
\newblock A positivity-preserving, energy stable scheme for a ternary
  {Cahn-Hilliard} system with the singular interfacial parameters.
\newblock {\em J. Comput. Phys.}, 442:110451, 2021.

\bibitem{dong19b}
L.~Dong, C.~Wang, H.~Zhang, and Z.~Zhang.
\newblock A positivity-preserving, energy stable and convergent numerical
  scheme for the {Cahn-Hilliard} equation with a {Flory-Huggins-deGennes}
  energy.
\newblock {\em Commun. Math. Sci.}, 17:921--939, 2019.

\bibitem{dong20a}
L.~Dong, C.~Wang, H.~Zhang, and Z.~Zhang.
\newblock A positivity-preserving second-order {BDF} scheme for the
  {Cahn-Hilliard} equation with variable interfacial parameters.
\newblock {\em Commun. Comput. Phys.}, 28:967--998, 2020.

\bibitem{duan22b}
C.~Duan, W.~Chen, C.~Liu, C.~Wang, and X.~Yue.
\newblock A second order accurate, energy stable numerical scheme for
  one-dimensional porous medium equation by an energetic variational approach.
\newblock {\em Commun. Math. Sci.}, 20(4):987--1024, 2022.

\bibitem{duan22a}
C.~Duan, W.~Chen, C.~Liu, C.~Wang, and S.~Zhou.
\newblock Convergence analysis of structure-preserving numerical methods for
  nonlinear {Fokker-Planck} equations with nonlocal interactions.
\newblock {\em Math. Methods Appl. Sci.}, 45(7):3764--3781, 2022.

\bibitem{duan20a}
C.~Duan, C.~Liu, C.~Wang, and X.~Yue.
\newblock Convergence analysis of a numerical scheme for the porous medium
  equation by an energetic variational approach.
\newblock {\em Numer. Math. Theor. Meth. Appl.}, 13:1--18, 2020.

\bibitem{eyre98}
D.~Eyre.
\newblock Unconditionally gradient stable time marching the {C}ahn-{H}illiard
  equation.
\newblock In J.~W. Bullard, R.~Kalia, M.~Stoneham, and L.Q. Chen, editors, {\em
  Computational and Mathematical Models of Microstructural Evolution},
  volume~53, pages 1686--1712, Warrendale, PA, USA, 1998. Materials Research
  Society.

\bibitem{feng12}
X.~Feng and S.M. Wise.
\newblock Analysis of a fully discrete finite element approximation of a
  {Darcy-Cahn-Hilliard} diffuse interface model for the {Hele-Shaw} flow.
\newblock {\em SIAM J. Numer. Anal.}, 50:1320--1343, 2012.

\bibitem{guan17a}
Z.~Guan, J.S. Lowengrub, and C.~Wang.
\newblock Convergence analysis for second order accurate schemes for the
  periodic nonlocal {Allen-Cahn} and {Cahn-Hilliard} equations.
\newblock {\em Math. Methods Appl. Sci.}, 40(18):6836--6863, 2017.

\bibitem{guan14a}
Z.~Guan, C.~Wang, and S.M. Wise.
\newblock A convergent convex splitting scheme for the periodic nonlocal
  {Cahn-Hilliard} equation.
\newblock {\em Numer. Math.}, 128:377--406, 2014.

\bibitem{han15}
D.~Han and X.~Wang.
\newblock A second order in time, uniquely solvable, unconditionally stable
  numerical scheme for {Cahn-Hilliard-Navier-Stokes} equation.
\newblock {\em J. Comput. Phys.}, 290:139--156, 2015.

\bibitem{hu09}
Z.~Hu, S.M. Wise, C.~Wang, and J.S. Lowengrub.
\newblock Stable and efficient finite-difference nonlinear-multigrid schemes
  for the phase-field crystal equation.
\newblock {\em J. Comput. Phys.}, 228:5323--5339, 2009.

\bibitem{lee02a}
H.G. Lee, J.S Lowengrub, and J.~Goodman.
\newblock Modeling pinchoff and reconnection in a {Hele-Shaw} cell. {I.} {The}
  models and their calibration.
\newblock {\em Phys. Fluids}, 14:492--513, 2002.

\bibitem{LiX2022b}
X.~Li, Z.~Qiao, and C.~Wang.
\newblock Double stabilizations and convergence analysis of a second- order
  linear numerical scheme for the nonlocal {Cahn-Hilliard} equation.
\newblock {\em Sci. China Math.}, 2022.
\newblock Accepted and in press.

\bibitem{LiX2022a}
X.~Li, Z.~Qiao, and C.~Wang.
\newblock Stabilization parameter analysis of a second order linear numerical
  scheme for the nonlocal {Cahn-Hilliard} equation.
\newblock {\em IMA J. Numer. Anal.}, 2022.
\newblock Accepted and in press.

\bibitem{LiuC2021b}
C.~Liu, C.~Wang, and Y.~Wang.
\newblock A structure-preserving, operator splitting scheme for
  reaction-diffusion equations with detailed balance.
\newblock {\em J. Comput. Phys.}, 436:110253, 2021.

\bibitem{LiuC2021a}
C.~Liu, C.~Wang, S.M. Wise, X.~Yue, and S.~Zhou.
\newblock A positivity-preserving, energy stable and convergent numerical
  scheme for the {Poisson-Nernst-Planck} system.
\newblock {\em Math. Comp.}, 90:2071--2106, 2021.

\bibitem{LiuC2022a}
C.~Liu, C.~Wang, S.M. Wise, X.~Yue, and S.~Zhou.
\newblock An iteration solver for the {Poisson-Nernst-Planck} system and its
  convergence analysis.
\newblock {\em J. Comput. Appl. Math.}, 406:114017, 2022.

\bibitem{liuY17}
Y.~Liu, W.~Chen, C.~Wang, and S.M. Wise.
\newblock Error analysis of a mixed finite element method for a
  {Cahn-Hilliard-Hele-Shaw} system.
\newblock {\em Numer. Math.}, 135:679--709, 2017.

\bibitem{QianWangZhou_JCP20}
Y.~Qian, C.~Wang, and S.~Zhou.
\newblock A positive and energy stable numerical scheme for the
  {Poisson-Nernst-Planck-Cahn-Hilliard} equations with steric interactions.
\newblock {\em J. Comput. Phys.}, 426:109908, 2021.

\bibitem{shen12}
J.~Shen, C.~Wang, X.~Wang, and S.M. Wise.
\newblock Second-order convex splitting schemes for gradient flows with
  {Ehrlich-Schwoebel} type energy: Application to thin film epitaxy.
\newblock {\em SIAM J. Numer. Anal.}, 50:105--125, 2012.

\bibitem{shen2015}
J.~Shen and X.~Yang.
\newblock Decoupled, energy stable schemes for phase-field models of two-phase
  incompressible flows.
\newblock {\em SIAM J. Numer. Anal.}, 53(1):279--296, 2015.

\bibitem{wang11a}
C.~Wang and S.M. Wise.
\newblock An energy stable and convergent finite-difference scheme for the
  modified phase field crystal equation.
\newblock {\em SIAM J. Numer. Anal.}, 49:945--969, 2011.

\bibitem{wise10}
S.M. Wise.
\newblock Unconditionally stable finite difference, nonlinear multigrid
  simulation of the {Cahn-Hilliard-Hele-Shaw} system of equations.
\newblock {\em J. Sci. Comput.}, 44:38--68, 2010.

\bibitem{wise09a}
S.M. Wise, C.~Wang, and J.~Lowengrub.
\newblock An energy stable and convergent finite-difference scheme for the
  phase field crystal equation.
\newblock {\em SIAM J. Numer. Anal.}, 47:2269--2288, 2009.

\bibitem{yang_CH_2017}
X.~Yang and J.~Zhao.
\newblock On linear and unconditionally energy stable algorithms for variable
  mobility {Cahn-Hilliard} type equation with logarithmic {Flory-Huggins}
  potential.
\newblock {\em Commun. Comput. Phys.}, 25(3):703--728, 2019.

\bibitem{Yuan2021a}
M.~Yuan, W.~Chen, C.~Wang, S.M. Wise, and Z.~Zhang.
\newblock An energy stable finite element scheme for the three-component
  {Cahn-Hilliard-type} model for macromolecular microsphere composite
  hydrogels.
\newblock {\em J. Sci. Comput.}, 87:78, 2021.

\bibitem{ZhangJ2021}
J.~Zhang, C.~Wang, S.M. Wise, and Z.~Zhang.
\newblock Structure-preserving, energy stable numerical schemes for a liquid
  thin film coarsening model.
\newblock {\em SIAM J. Sci. Comput.}, 43(2):A1248--A1272, 2021.

\end{thebibliography}

\end{document}